\documentclass[12pt]{amsart}

\usepackage{amsmath,amssymb,amsthm,enumerate,mathrsfs}
\usepackage[margin=1in]{geometry}

\numberwithin{equation}{section}

\DeclareMathOperator{\codim}{codim}

\newcommand{\C}{\mathbb{C}}

\newcommand{\N}{\mathbb{N}}
\newcommand{\Z}{\mathbb{Z}}

\newcommand{\PP}{\mathbb{P}}

\newcommand{\F}{\mathbb{F}}

\newcommand{\E}{\mathbb{E}}

\newcommand{\ve}{\varepsilon}

\newcommand{\1}{^{-1}}

\newcommand{\f}[2]{\frac{#1}{#2}}

\newtheorem{theorem}{Theorem}[section]

\newtheorem{definition}[theorem]{Definition}

\newtheorem{corollary}[theorem]{Corollary}

\newtheorem{lemma}[theorem]{Lemma}

\numberwithin{theorem}{section}

\title{Subsets of $\F_p^n\times\F_p^n$ without L-shaped configurations}
\author{Sarah Peluse}
\address{School of Mathematics, Institute for Advanced Study, Princeton, NJ, USA}
\address{Department of Mathematics, Princeton University, Princeton, NJ, USA}
\email{speluse@princeton.edu}

\begin{document}
\begin{abstract}
  Fix a prime $p\geq 11$. We show that there exists a positive integer $m$ such that any subset of $\F_p^n\times\F_p^n$ containing no nontrivial configurations of the form $(x,y),(x,y+z),(x,y+2z),(x+z,y)$ must have density $\ll 1/\log_{m}{n}$, where $\log_{m}$ denotes the $m$-fold iterated logarithm. This gives the first reasonable bound in the multidimensional Szemer\'edi theorem for a two-dimensional four-point configuration in any setting.
\end{abstract}

\maketitle

\section{Introduction}\label{intro}

Szemer\'edi's famous theorem on arithmetic progressions, which states that any subset of
the integers with positive upper density contains arbitrarily long arithmetic
progressions, has the following multidimensional generalization due to Furstenberg and
Katznelson~\cite{FurstenbergKatznelson78}:

\begin{theorem}\label{FK}
  Let $X$ be a finite, nonempty subset of $\Z^d$. If $S\subset[N]^d$ contains no
  nontrivial homothetic copy $a+bX$ of $X$, then $|S|=o(N^d)$.
\end{theorem}
Here we use the standard notation $[N]:=\{1,\dots,N\}$. There has been great interest over
the past few decades in proving a quantitative version of this theorem with reasonable
bounds, i.e., with an upper bound for $|S|$ whose savings over the trivial bound of $N^d$
grows at least as quickly as a finite number of iterated logarithms. Indeed, Gowers has
posed the problem of proving such a result on several
occasions~\cite{Gowers98S,Gowers01S,Gowers01}, and others, such as Graham~\cite{Graham97},
have asked for bounds for sets lacking particular multidimensional configurations. While
reasonable bounds are known in Szemer\'edi's theorem due to work of
Gowers~\cite{Gowers98,Gowers01}, none are known in the Furstenberg--Katznelson theorem in
general. Furstenberg and Katznelson's original proof, which was via ergodic theory,
produces no explicit bounds, while the hypergraph regularity proofs of Nagle, R\"odl,
Schacht, and Skokan~\cite{NagleRodlSchacht06,RodlSkokan04}, Gowers~\cite{Gowers07}, and
Tao~\cite{Tao06} each give a saving over the trivial bound of inverse Ackermann type.

Reasonable bounds in Theorem~\ref{FK} are currently known for only one genuinely
multidimensional configuration: two-dimensional corners
\begin{equation}\label{corner}
  (x,y),(x,y+z),(x+z,y),
\end{equation}
(and, thus, their linear images,) due to work of Shkredov~\cite{Shkredov06I,Shkredov06II},
who proved that any subset of $[N]\times[N]$ containing no nontrivial corners has size at
most $\ll N^2/(\log\log{N})^c$ for some absolute constant $c>0$. No reasonable bounds are
known for any two-or-more-dimensional four-point configuration, such as three-dimensional
corners,
\begin{equation}\label{3dim}
  (x,y,z),(x,y,z+w),(x,y+w,z),(x+w,y,z),
\end{equation}
or axis-aligned squares,
\begin{equation}\label{square}
  (x,y),(x,y+z),(x+z,y),(x+z,y+z).
\end{equation}
The latter of these two configurations is the topic of a conjecture of
Graham~\cite{Graham97}, which states that any subset $S\subset\N\times\N$ for which
$\sum_{(x,y)\in S}\f{1}{x^2+y^2}$ diverges must contain an axis-aligned square. Graham
also conjectured, more generally, that if $\sum_{(x,y)\in S}\f{1}{x^2+y^2}$ diverges, then
$S$ must contain a homothetic copy of $[m]\times[m]$ for every positive integer $m$. This
is a two-dimensional generalization of the famous and still unresolved conjecture of
Erd\H{o}s that every subset $T\subset\N$ for which $\sum_{n\in T}\f{1}{n}$ diverges must
contain arbitrarily long arithmetic progressions.

Proving reasonable bounds for sets lacking the four-point configurations~\eqref{3dim}
and~\eqref{square} seems to be out of reach. This is because no one has managed yet to
prove anything useful about a certain two-dimensional directional uniformity norm that
naturally appears in the study of these configurations. Details on this difficulty can be
found in the work of Austin~\cite{Austin13I,Austin13II}, where he demonstrates how
enormously complicated and difficult even 100\% and 99\% inverse theorems can be for
directional uniformity norms.

The purpose of this paper is to identify the first two-dimensional four-point
configuration for which reasonable bounds in the multidimensional Szemer\'edi theorem can
be proven, and to prove such bounds in the finite field model setting. We will study the
configuration
\begin{equation}\label{Lshape}
  (x,y),(x,y+z),(x,y+2z),(x+z,y),
\end{equation}
which, when plotted on a two-dimensional integer grid, takes the shape of the capital
letter ``L''. Because of this, we refer to~\eqref{Lshape} as an \textit{L-shaped
  configuration}, and an L-shaped configuration with $z\neq 0$ as a \textit{nontrivial}
L-shaped configuration.

\begin{theorem}\label{main}
  There exists a natural number $m$ and a constant $C>0$ such that the following
  holds. Fix a prime $p\geq 11$, and set $N:=p^n$. If $n\geq C$, then all
  $S\subset\F_p^n\times\F_p^n$ containing no nontrivial L-shaped configurations satisfy
  \[
    |S|\ll\f{N^2}{\log_{m}{N}}.
  \]
\end{theorem}
The $m$ obtained in the theorem is huge, so we do not attempt to compute it. The bulk of the size of $m$ comes from our use of a recent quantitative inverse theorem for the $U^{10}$-norm on $\F_p^n$ due to Gowers and Mili\'cevi\'c, who in~\cite{GowersMilicevic20} give a rough upper bound for the number of iterated exponentials appearing in their result. Based on this, $m$ is likely at least 24 trillion. The use of this inverse theorem is necessary in our proof, and no amount of care to argue efficiently in the rest of the argument can reduce $m$ by much. So, we have not tried to optimize the proof of Theorem~\ref{main}, choosing instead to present the simplest argument that gives a reasonable upper bound.

It is likely that the proof of Theorem~\ref{main} can be adapted to the integer setting to prove a reasonable bound for subsets of $[N]\times[N]$ lacking L-shaped configurations, with the bound obtained being far more reasonable than the bound in Theorem~\ref{main}. This is because the quantitative aspects of Manners's~\cite{Manners2018} inverse theorem for the $U^{s}$-norm on cyclic groups are better than those of Gowers and Mili\'cevi\'c's inverse theorem when $s>4$. It is also likely that Theorem~\ref{main} can be extended to more general L-shaped configurations with a longer vertical ``leg'',
\[
  (x,y),(x,y+z),\dots,(x,y+mz),(x+z,y),
\]
in both the finite field model and integer settings. We expect, however, that understanding L-shaped configurations with two longer ``legs'',
\[
  (x,y),(x,y+z),\dots,(x,y+mz),(x+z,y),\dots,(x+\ell z,y),
\]
is significantly more difficult, for some of the same reasons that proving reasonable bounds for sets lacking three-dimensional corners or axis-aligned squares seems out of reach.

While progress in proving a quantitative version of the multidimensional Szemer\'edi
theorem has so far been extremely limited, there has been a bit more success in proving
reasonable bounds for sets lacking multidimensional configurations with more degrees of
freedom than those in Theorem~\ref{FK}. Prendiville~\cite{Prendiville15} has proven
reasonable bounds for subsets of $[N]^d$ lacking any sufficiently nondegenerate three- or
four-term matrix progression, and one consequence of his work is that any subset of
$[N]\times [N]$ containing no four vertices of any square (not necessarily axis-aligned)
has size at most $\ll N^2/(\log\log{N})^{c'}$ for some absolute constant
$c'>0$.

The remainder of this paper is organized as follows. In Section~\ref{outline}, we give a detailed outline of our proof of Theorem~\ref{main}, including statements of the three main components of the density-increment argument: control of the count of L-shaped configurations by directional uniformity norms, obtaining a density-increment on a structured set, and pseudorandomizing the structured set previously obtained. After introducing additional technical preliminaries in Sections~\ref{prelims} and~\ref{phipsd}, we prove these three main components in Sections~\ref{gvn},~\ref{inverse}, and~\ref{pseudo}, respectively. We then carry out the density increment argument in Section~\ref{dinc}, proving Theorem~\ref{main}.

\section*{Acknowledgments}
The author thanks Ben Green and Freddie Manners for helpful discussions, and Ben Green,
Noah Kravitz, Terry Tao, and the anonymous referees for useful comments on earlier drafts.

The author is supported by the NSF Mathematical Sciences Postdoctoral Research Fellowship Program under Grant No. DMS-1903038 and the Oswald Veblen fund, and also gratefully acknowledges the support and hospitality of the Hausdorff Institute for Mathematics, where the bulk of this paper was written.

\section{Outline of the proof of Theorem~\ref{main}}\label{outline}

We begin this section by introducing the minimum amount of notation and preliminaries
needed to understand our proof outline. We will use the standard asymptotic notation
$O,\Omega,$ and $o$, along with Vinogradov's notation $\ll,\gg,$ and $\asymp$. For any two
quantities $A$ and $B$, the relations $A=O(B)$, $B=\Omega(A)$, $A\ll B$, and $B\gg A$ all
mean that $|A|\leq C|B|$ for some absolute constant $C>0$. We will write $O(B)$ to
represent a quantity that is $\ll B$ and $\Omega(A)$ to represent a quantity that is
$\gg A$. When any of these asymptotic symbols appears with a subscript, the implied
constant is allowed to depend on the parameters in the subscript. Since we fix a prime $p$
in Theorem~\ref{main}, the implied constants appearing throughout the paper will sometimes
depend on $p$ even though we will not alert the reader to this with a subscript. We will
use $\log_m$ to denote the $m$-fold iterated logarithm, so that $\log_1:=\log$ and
$\log_{i}:=\log\circ\log_{i-1}$ for all $i>1$, as well as $\exp^m$ to denote the $m$-fold
iterated exponential, so that $\exp^1=\exp$ and $\exp^{i}:=\exp\circ\exp^{i-1}$ for all
$i>1$.

We will frequently denote the indicator function of a set $A$ by the letter $A$ itself, so that
\[
  A(x):=\begin{cases} 1 & x\in A \\ 0 & x\notin A \end{cases}.
\]
For any pair of finite sets $X\subset Y$ with $Y\neq\emptyset$, we denote the density of $X$ in $Y$ by
\[
  \mu_Y(X):=\f{|X|}{|Y|}.
\]
For any function $f:X\to\C$, we denote the average of $f$ over $X$ by
\[
  \E_{x\in X}f(x):=\f{1}{|X|}\sum_{x\in X}f(x).
\]
When $X=\F_p^n$, we will usually drop ``$\in X$'' and just write $\E_x$ for $\E_{x\in \F_p^n}$. Whenever $f$ satisfies $|f(x)|\leq 1$ for all $x$ in its domain, we say that it is \textit{$1$-bounded}. Note that the indicator function of any set is $1$-bounded.

For any $f:\F_p^n\to\C$ and $\xi\in \F_p^n$, we define the Fourier coefficient of $f$ at $\xi$ using the normalization
\[
  \widehat{f}(\xi):=\E_{x}f(x)e_p(-\xi\cdot x),
\]
where $e_p(z):= e^{2\pi iz/p}$ and $\cdot$ denotes the usual dot product in $\F_p^n$. With this choice of normalization, the Fourier inversion formula and Parseval's identity read
\[
  f(x)=\sum_{\xi\in\F_p^n}\widehat{f}(\xi)e_p(\xi\cdot x)
\]
and
\[
  \E_{x}|f(x)|^2=\sum_{\xi\in\F_p^n}\left|\widehat{f}(\xi)\right|^2,
\]
respectively. 

Let $H$ be any abelian group and $g:H\to \C$. For any $h\in H$, we define the function $\Delta_hg:H\to\C$ by
\[
  \Delta_hg(x):=g(x)\overline{g(x+h)},
\]
and, for any $h_1,\dots,h_s\in H$, define the $s$-fold iterated differencing operator $\Delta_{h_1,\dots,h_s}$ by
\[
  \Delta_{h_1,\dots,h_s}g:=\Delta_{h_1}\cdots\Delta_{h_s}g.
\]
Note that $\Delta_{h_1,\dots,h_s}g=\Delta_{h_{\sigma(1)},\dots,h_{\sigma(s)}}g$ for any permutation $\sigma$ of $\{1,\dots,s\}$.

Now we can recall the definition of the Gowers uniformity norms.

\begin{definition}
  Let $s\in\N$, $H$ be an abelian group, and $f:H\to\C$. The $U^s$-norm of $f$ is defined by
  \[
    \|f\|_{U^s(H)}^{2^s}:=\E_{x,h_1,\dots,h_s\in H}\Delta_{h_1,\dots,h_s}f(x)
  \]
\end{definition}
The basic properties of these norms can be found in~\cite{Tao12}. One such fact needed in the upcoming outline is the inverse theorem for the $U^2$-norm, which is a simple consequence of Fourier inversion and Parseval's identity.

\begin{lemma}\label{inverseU2}
  Let $H$ be an abelian group and $f:H\to\C$ be $1$-bounded. If $\|f\|_{U^2(H)}\geq\delta$, then there exists a $\psi\in \widehat{H}$ such that
  \[
    |\E_{x\in H}f(x)\psi(x)|\geq\delta^2.
  \]
\end{lemma}

We will also sometimes need the notion of the $U^2$-norm on an affine subspace $w+V$ of $\F_p^n$, which is defined by $\|f\|_{U^2(w+V)}:=\|f(\cdot -w)\|_{U^2(V)}$. The corresponding inverse theorem for these norms follows from Lemma~\ref{inverseU2}.

\subsection{A review of Shkredov's argument in the finite field model setting}

Before we outline the proof of Theorem~\ref{main}, it will be instructive to review Shkredov's argument for corners~\eqref{corner} in the finite field model setting. A detailed account of the argument can be found in the expositions of Green~\cite{Green05note,Green05}.

Shkredov's proof proceeds via a density-increment argument. As in all analytic approaches to Szemer\'edi's theorem and its generalizations, we begin by defining a multilinear average over the configuration of interest. For $g_0,g_1,g_2:\F_p^n\times \F_p^n\to\C$, set
\[
  \Lambda_{\llcorner}(g_0,g_1,g_2):=\E_{x,y,z}g_0(x,y)g_1(x,y+z)g_2(x+z,y).
\]
Then, for any $S\subset \F_p^n\times \F_p^n$, the quantity $\Lambda_{\llcorner}(S,S,S)$ equals the normalized count,
\[
  \f{\#\left\{x,y,z\in \F_p^n:(x,y),(x,y+z),(x+z,y)\in S\right\}}{p^3},
\]
of the number of corners in $S$. Setting $N:=p^n=|\F_p^n|$, we let $\sigma:=|S|/N^2$ denote the density of $S$ in $\F_p^n\times \F_p^n$ and $g_S:=S-\sigma$ denote the balanced function of $S$. It follows from the trilinearity of $\Lambda_{\llcorner}$ that
\[
  \Lambda_{\llcorner}(S,S,S)=\sigma\Lambda_{\llcorner}(1,S,S)+\Lambda_{\llcorner}(g_S,S,S).
\]
Since $\Lambda_{\llcorner}(1,S,S)\geq\sigma^2$ by the Cauchy--Schwarz inequality, if the normalized count of corners in $S$ is far below the $\sim\sigma^3$ expected for a random set of density $\sigma$, which is the case when $S$ has no nontrivial corners and $N$ is sufficiently large in terms of $\sigma$, then $|\Lambda_{\llcorner}(g_S,S,S)|$ must be large.

It can then be shown, by an appropriate sequence of applications of the Cauchy--Schwarz inequality, that $g_S$ must have large \textit{box norm}
\[
  \|g_S\|_{\square}:=\left(\E_{x,y,x',y'}g_S(x,y)\overline{g_S(x,y')g_S(x',y)}g_S(x',y')\right)^{1/4}.
\]
If $\|g_S\|_{\square}$ is large, it follows by an averaging argument that $S$ has density at least $\sigma+\Omega(\sigma^{O(1)})$ on a product set $A\times B$ for some large $A,B\subset \F_p^n$.

One may then hope to continue the density-increment argument by proving the following generalization of the result just sketched: if $S$ is a subset of density $\sigma$ of a product set $T=A\times B$ and contains no nontrivial corners, then $S$ has density at least $\sigma+\Omega(\sigma^{O(1)})$ on a product set $T'$ contained in $T$.

It turns out, however, that the Cauchy--Schwarz argument mentioned previously yields a lower bound on the box norm of large enough size only when $A$ and $B$ are sufficiently \textit{Fourier pseudorandom}, meaning that their balanced functions $A-|A|/N$ and $B-|B|/N$ both have small $U^2$-norm. The components of the product set just obtained are essentially arbitrary aside from being large. They are, in particular, not guaranteed to be Fourier pseudorandom.

To overcome this difficulty, Shkredov introduced a \textit{pseudorandomizing} step into his proof. He used an energy increment argument incorporating the $U^2$-inverse theorem to partition $\F_p^n\times \F_p^n$ into products of large affine subspaces of the form
\begin{equation}\label{productofsubspaces}
  (u+V)\times (w+V),
\end{equation}
for most of which the sets $(A-u)\cap V$ and $(B-w)\cap V$ are Fourier pseudorandom in $V$. By an averaging argument, there must exist such a product of affine subspaces~\eqref{productofsubspaces} on which the restrictions of $A$ and $B$ are both sufficiently dense and Fourier pseudorandom, and such that $S$ still has increased density $\sigma+\Omega(\sigma^{O(1)})$ on the intersection of $T$ with $(u+V)\times (w+V)$.

By passing to this product of cosets and using that corners are preserved by translation
and invertible linear transformations of the form $(x,y)\mapsto (Ex,Ey)$, one can then
continue the density-increment argument with $\F_p^n\times \F_p^n$ replaced by
$\F_p^{n'}\times \F_p^{n'}$, where $n'=\dim{V}$. If $S\subset T$ contains no nontrivial
corners and $A$ and $B$ are sufficiently Fourier pseudorandom, then $g_S$ must have large
box norm localized to $T$. One must then prove that $S$ has a further density-increment on
a product set contained in $T$, which is, fortunately, of exactly the same difficulty
whether $T=\F_p^n\times\F_p^n$ or some other large product set. By applying the
pseudorandomizing procedure to the factors of the product set just produced, one can then
deduce that if $S$ is a subset of density $\sigma$ of a product set $T=A\times B$, where
$A$ and $B$ are large and sufficiently Fourier pseudorandom, and $S$ contains no
nontrivial corners, then $S$ has density at least $\sigma+\Omega(\sigma^{O(1)})$ on a
product set $T'=A'\times B'$ contained in $T$, where $A'$ and $B'$ are also large and
sufficiently Fourier pseudorandom. The density increment iteration can be carried out
repeatedly to produce a good bound for subsets of $\F_p^n\times \F_p^n$ lacking corners.

\subsection{An outline of our argument}\label{ssoutline}

The obstructions to uniformity for L-shaped configurations are not just (skew) product sets, as was the case for corners, but also very general sets of the form
\[
  \left\{(x,y)\in \F_p^n\times \F_p^n: y\in u_x+V_x\right\},
\]
where each $u_x+V_x$ is an affine subspace of $\F_p^n$. For example, assume that $n\geq 3$, and consider the set
\[
  \left\{(x,y)\in \F_p^n\times \F_p^n: x\cdot y=0\right\}.
\]
This set has density
\[
  \f{(N-1)N/p+N}{N^2}\sim\f{1}{p}
\]
in $\F_p^n\times \F_p^n$, but
\[
  \left[(N-1)-(p-1)\right]\left(\f{N}{p}-1\right)\f{N}{p^2}+\left(\f{N}{p}-1\right)(p-1)\f{N}{p}+N+2(N-1)\f{N}{p}\sim\f{N^3}{p^3}
\]
L-shaped configurations, in contrast to the $\sim N^3/p^4$ expected in a random subset of $\F_p^n\times \F_p^n$ of density $1/p$. Similarly, the number of L-shaped configurations in the sets
\[
  \left\{(x,y)\in \F_p^n\times \F_p^n: \phi(x)\cdot y=0\right\}
\]
and
\[
  \left\{(x,y)\in \F_p^n\times \F_p^n: y_1=u(x)\right\},
\]
where $\phi(x)\in \F_p^n$ and $u(x)\in\F_p$ are now chosen uniformly at random, is also $\sim N^3/p^3$ with high probability, while the sets have density $\sim 1/p$ with high probability. These new sorts of obstructions to uniformity are the main reason why the study of L-shaped configurations is significantly more difficult than that of corners, and must be taken into account to prove Theorem~\ref{main}.

For any functions $g_0,g_1,g_2,g_3:\F_p^n\times \F_p^n\to\C$, we define
\begin{equation}\label{lambda}
  \Lambda(g_0,g_1,g_2,g_3):=\E_{x,y,z}g_0(x,y)g_1(x,y+z)g_2(x,y+2z)g_3(x+z,y),
\end{equation}
so that $\Lambda(S,S,S,S)$ equals the normalized count of L-shaped configurations in any subset $S$ of $\F_p^n\times \F_p^n$. The multilinearity of $\Lambda$ implies that
\begin{equation}\label{telescope}
  |\Lambda(S,S,S,S)-\sigma^4|\leq \sigma^2|\Lambda(1,1,g_S,S)|+\sigma|\Lambda(1,g_S,S,S)|+|\Lambda(g_S,S,S,S)|,
\end{equation}
where, as before, $g_S=S-\sigma$ is the balanced function of $S$. Thus, if the normalized count of L-shaped configurations in $S$ is far from the random normalized count $\sigma^4$, one of $|\Lambda(1,1,g_S,S)|$, $|\Lambda(1,g_S,S,S)|$, or $|\Lambda(g_S,S,S,S)|$ must be large. In particular, when $S$ contains no nontrivial L-shaped configurations and $N$ is sufficiently large in terms of $\sigma$, one of these quantities will be larger than $\sigma^4/2$. It then follows from several applications of the Cauchy--Schwarz inequality that one of the following directional uniformity norms of $g_S$ must be larger than $\sigma^4/2$:
\begin{equation}\label{eq:star1}
  \|g\|_{\star_1}:=\left(\E_{x,y,h_1,h_2,h_3}\Delta_{(0,h_1),(0,h_2),(h_3,0)}g(x,y)\right)^{1/8},
\end{equation}
\begin{equation}\label{eq:star2}
  \|g\|_{\star_2}:=\left(\E_{x,y,h_1,h_2}\Delta_{(0,h_1),(-h_2,h_2)}g(x,y)\right)^{1/4},
\end{equation}
or
\begin{equation}\label{eq:star3}
  \|g\|_{\star_3}:=\left(\E_{x,y,h_1}\Delta_{(-h_1,2h_1)}g(x,y)\right)^{1/2}.
\end{equation}
Here $\|\cdot\|_{\star_3}$ is only a semi-norm, while $\|\cdot\|_{\star_1}$ and
$\|\cdot\|_{\star_2}$ are genuine norms. Since these are all Gowers box norms, one can
find a proof that they are (semi-)norms in Appendix B of~\cite{GreenTao10}. The norm
$\|\cdot\|_{\star_1}$ had previously been studied, in the setting of cyclic groups, in work of Shkredov~\cite{Shkredov09}.

Directional uniformity norms with two differencing parameters,
\[
\left[\E_{x,y,h,k\in H}\Delta_{hv_1,kv_2}g(x,y)\right]^{1/4},
\]
for fixed nonzero $v_1,v_2\in H\times H$, are well-understood. Either $v_1$ and $v_2$ are
scalar multiples of each other, in which case the norm is just the $U^2$-norm on
$\langle v_1\rangle$ averaged over cosets of $\langle v_1\rangle$, or they are linearly
independent, as in the definition of $\|\cdot\|_{\star_2}$, in which case the norm is,
after a change of variables, equivalent to the two-dimensional box norm. Directional
uniformity norms with three differencing parameters,
\[
\left[\E_{x,y,h_1,h_2,h_3\in H}\Delta_{h_1v_1,h_2v_2,h_3v_3}g(x,y)\right]^{1/8},
\]
for fixed nonzero $v_1,v_2,v_3\in H\times H$ analogously fall into one of three cases:
either $v_1,v_2,$ and $v_3$ are collinear, lie on exactly two lines, or are in general
position. In the first case, the norm is just the $U^3$-norm on $\langle v_1\rangle$
averaged over cosets of $\langle v_1\rangle$. In the third case, the norm is linearly
equivalent to the intractable norm that arises in the study of $3$-dimensional corners and
axis-aligned squares. The norm $\|\cdot\|_{\star_1}$ we encounter falls into the second
case, and the study and fruitful use of this norm turns out to be possible (though still
complicated) due to its structure as a ``$U^1\times U^2$-norm''.

The upshot is that if $S$ contains no nontrivial L-shaped configurations, then it must
have density at least $\sigma+\Omega(\sigma^{O(1)})$ on a set of the form
\begin{equation}\label{Tform}
  T:=\left\{(x,y)\in \F_p^n\times \F_p^n: B(y)C(x+y)D(2x+y)\Phi(x,y)=1\right\},
\end{equation}
where $\Phi\subset A\times\F_p^n$ is of the form
\begin{equation}\label{Phiform}
  \Phi:=\left\{(x,y)\in A\times \F_p^n: y\in u+V_x\right\},
\end{equation}
for some element $u\in\F_p^n$ and collection of subspaces $\{V_x:x\in A\}$ of
$\mathbb{F}_p^n$, where $A,B,C,D\subset \F_p^n$ are large and $\codim V_x$ is small for
each $x\in A$. Note that this set $\Phi$ is not quite as general as the one appearing at
the very beginning of this subsection, as the element $u$ of $\F_p^n$ does not vary with
$x$. It takes some extra work to show that we can guarantee $\Phi$ to be of this special
form, which turns out to be necessary for our density-increment iteration. We say more
about this point in Section~\ref{inverse}.

We would like to continue the density increment iteration and show that $S':= T\cap S$,
which also lacks L-shaped configurations, has a further density increment of at least the
same size as the first on a subset $T'$ of $T$ of the same general
form~\eqref{Tform}. Analogously to Shkredov's argument for corners, we can only hope to do
this if $A,B,C,D,$ and $\Phi$ are sufficiently pseudorandom, for some appropriate notions
of pseudorandomness. We will need to control the count of L-shaped configurations by the
norms $\|\cdot\|_{\star_1}$, $\|\cdot\|_{\star_2}$, and $\|\cdot\|_{\star_3}$ defined
in~\eqref{eq:star1},~\eqref{eq:star2}, and~\eqref{eq:star3} with no loss of density factors, i.e., show that
\begin{equation}\label{csoutline1}
\frac{|\Lambda(f_0,f_1,f_2,f_3)|}{\Lambda(T,T,T,T)}\geq\delta\implies   \f{\|f_0\|_{\star_1}}{\|T\|_{\star_1}}\gg_\delta 1,
\end{equation}
\begin{equation}\label{csoutline2}
\frac{|\Lambda(T,f_1,f_2,f_3)|}{\Lambda(T,T,T,T)}\geq\delta\implies   \f{\|f_1\|_{\star_2}}{\|T\|_{\star_2}}\gg_\delta 1,
\end{equation}
and
\begin{equation}\label{csoutline3}
\frac{|\Lambda(T,T,f_2,f_3)|}{\Lambda(T,T,T,T)}\geq\delta\implies \f{\|f_2\|_{\star_3}}{\|T\|_{\star_3}}\gg_\delta 1,
\end{equation}
and also obtain a density-increment with no loss of density factors when some localized norm $\|\cdot\|_{\star_i}$ of the balanced function of a set is large, i.e., show that if
\[
  \f{\|g_S\|_{\star_1}}{\|T\|_{\star_1}}, \f{\|g_S\|_{\star_2}}{\|T\|_{\star_2}},\text{ or }\f{\|g_S\|_{\star_3}}{\|T\|_{\star_3}}\geq \delta,
\]
where now $g_S:=S-\sigma T$, then there exists a subset $T'\subset T$ of the same general form,
\[
  T':=\{(x,y)\in \F_p^n\times \F_p^n:B'(y)C'(x+y)D'(2x+y)\Phi'(x,y)=1\},
\]
as $T$ on which $S$ has a density increment
\[
  \E_{(x,y)\in T'}S(x,y)\geq\sigma+\Omega_\delta(1)
\]
depending only on $\delta$. Such results are needed so that the density increment obtained at each step of the iteration is independent of the step. If one is not sufficiently careful, it is easy to end up with a density increment that gets smaller as the subset $T$ of $\F_p^n\times\F_p^n$ gets sparser, which is not enough to close the density increment iteration.

To carry out these arguments, we will need $A,B,C,$ and $D$ to be pseudorandom with
respect to the $U^{10}(\F_p^n)$-norm. The situation for $\Phi$ is more complicated, and
deciding on a good measure of pseudorandomness for $\Phi$ that is amenable to a
Shkredov-like pseudorandomization procedure and can also be used to analyze the various
averages appearing throughout our argument is one of the challenges of the proof of
Theorem~\ref{main}. A suitable condition on $\Phi$ turns out to be that it is pseudorandom
with respect to the $U^{8}(\F_p^n\times \F_p^n)$-norm. This condition is not, on its own,
immediately useful in the arguments of Sections~\ref{gvn} and~\ref{inverse}, since the
various averages that appear are not controllable by the $U^{8}(\F_p^n\times\F_p^n)$-norm
of $\Phi$. It takes a bit of work to show that it implies a roughly equivalent statement
about the typical codimensions of certain affine subspaces obtained from $\Phi$. We prove
this in Section~\ref{phipsd}, deriving some new results on the combinatorics of
approximate polynomials along the way.

The proof of the implications~\eqref{csoutline1},~\eqref{csoutline2}, and~\eqref{csoutline3} when $A,B,C,D,$ and $\Phi$ are sufficiently pseudorandom consists of many careful applications of the Cauchy--Schwarz inequality, along with appeals to standard facts about the number of linear configurations of bounded Cauchy--Schwarz complexity in products of pseudorandom sets intersected with subspaces of bounded codimension. We carry out this argument in Section~\ref{gvn}.

Obtaining a large enough density-increment when $\|g_S\|_{\star_1}$ is large for $S\subset T$ requires some new ideas and a significant amount of extra work beyond the proof of the non-localized case, in contrast to the situation for the box norm localized to product sets, where the argument is the same as the non-localized case. In order to get such a density-increment that only depends on $\delta$ and not on the densities of $A,B,C,D,$ or $\Phi$, one of the key ingredients is a density-preserving inverse theorem for the $U^2(\Phi(x,\cdot))$-norms on pseudorandom sets derived from $A,B,C,$ and $D$, which we prove using a version of the transference principle. We carry out this argument in Section~\ref{inverse}.

As was the case for corners, the sets $A',B',C',D',$ and $\Phi'$ obtained in the previous paragraph are not guaranteed to be pseudorandom. We must also carry out a pseudorandomizing procedure to locate a product of large affine subspaces of the form $(u+V)\times (w+V)$ on which $A',B',C',D',$ and $\Phi'$ are sufficiently pseudorandom and $S$ still has a large density increment on $T'\cap [(u+V)\times(w+V)]$. Our pseudorandomization procedure is similar to Shkredov's, but with some new complications coming from our desire for $A',B',C',$ and $D'$ and $\Phi'$ to be pseudorandom with respect to the $U^{10}(\F_p^n)$- and $U^{8}(\F_p^n\times\F_p^n)$-norms, respectively, and from $\Phi'$'s particular structure as a union of affine subspaces in the second factor of $\F_p^n\times\F_p^n$. To handle the first complication, we use a recent quantitative inverse theorem of Gowers and Mili\'cevi\'c~\cite{GowersMilicevic20} for the $U^{s}$-norms on vector spaces over finite fields, combined with a result of Cohen and Tal~\cite{CohenTal15} that allows us to partition $\F_p^n$ into large affine subspaces on which any finite collection of bounded degree polynomials are all constant. The structure of $\Phi'$ has the potential to cause issues in a Shkredov-like pseudorandomization argument, since the intersection of $\Phi'$ with a cell may no longer be the union of affine subspaces all having the same dimension. We will explain how this complication is dealt with in Section~\ref{pseudo}, since it requires a bit of set up.

\subsection{Key intermediate results}

We finish this section by stating the key intermediate results needed to prove
Theorem~\ref{main} that we just described in the outline. Recall that $g_S=S-\sigma$
denotes the balanced function of $S$.
\begin{lemma}[Estimation of $\Lambda(T,T,T,S)$]\label{lower}
  There exist absolute constants $0<c_1<1<c_2$ such that the following holds.  Let $d$ be a nonnegative integer, and set $\rho:=p^{-d}$.  Suppose that $A,B,C,D\subset \F_p^n$ have densities $\alpha,\beta,\gamma,\delta$, respectively, and that $\Phi\subset \F_p^n\times \F_p^n$ takes the form
  \[
  \Phi=\left\{(x,y)\in A\times \F_p^n: y\in u+V_x\right\},
  \]
  where each $V_x$ is a subspace of $\F_p^n$ of codimension $d$. Define $T\subset \F_p^n\times \F_p^n$ by~\eqref{Tform} and suppose that $S\subset T$ has density $\sigma$ in $T$. Let $\ve\leq c_1(\sigma\alpha\beta\gamma\delta\rho)^{c_2}$ and assume that 
  \[
    \|A-\alpha\|_{U^{5}(\F_p^n)},\|B-\beta\|_{U^{5}(\F_p^n)},\|C-\gamma\|_{U^{5}(\F_p^n)},\|D-\delta\|_{U^{5}(\F_p^n)},\|\Phi-\alpha\rho\|_{U^{2}(\F_p^n\times\F_p^n)}<\ve.
  \]
  Then
  \[
    \Lambda(T,T,T,S)\gg \sigma\alpha^2\beta^3\gamma^3\delta^3\rho^3.
  \]
\end{lemma}

As a consequence, we get that if $\ve$ is small enough, $n$ is large enough, and $S\subset T$ has no nontrivial L-shaped configurations, then
\[
  \max(|\Lambda(g_S,S,S,S)|,|\Lambda(T,g_S,S,S)|,|\Lambda(T,T,g_S,S)|)\gg\sigma^4\alpha^2\beta^3\gamma^3\delta^3\rho^3.
\]

\begin{lemma}[Control by $\|\cdot\|_{\star_i}$ norms]\label{control}
   Let $d$ be a nonnegative integer, and set $\rho:=p^{-d}$. Suppose that $A,B,C,D\subset \F_p^n$ have densities $\alpha,\beta,\gamma,\delta$, respectively, and that $\Phi\subset \F_p^n\times \F_p^n$ takes the form
  \[
  \Phi=\left\{(x,y)\in A\times \F_p^n: y\in u+V_x\right\},
  \]
  where each $V_x$ is a subspace of $\F_p^n$ of codimension $d$. Assume that 
  \[
    \|A-\alpha\|_{U^{8}(\F_p^n)},\|B-\beta\|_{U^{8}(\F_p^n)},\|C-\gamma\|_{U^{8}(\F_p^n)},\|D-\delta\|_{U^{8}(\F_p^n)},\|\Phi-\alpha\rho\|_{U^{6}(\F_p^n\times\F_p^n)}<\ve.
  \]
  Define $T\subset \F_p^n\times \F_p^n$ by~\eqref{Tform} and suppose that $f_0,f_1,f_2,f_3:\F_p^n\times \F_p^n\to\C$ are $1$-bounded functions supported on $T$. Then
  \begin{equation}\label{gvn1}
    |\Lambda(f_0,f_1,f_2,f_3)|^8\leq\alpha^{14}\beta^{20}\gamma^{16}\delta^{16}\rho^{18}\|f_0\|_{\star_1}^8+O\left(\f{\ve^{\Omega(1)}}{\rho^{O(1)}}\right)
  \end{equation}
  \begin{equation}\label{gvn2}
    |\Lambda(T,f_1,f_2,f_3)|^4\leq \alpha^{8}\beta^{8}\gamma^{10}\delta^{8}\rho^{8}\|f_1\|_{\star_2}^4+O\left(\f{\ve^{\Omega(1)}}{\rho^{O(1)}}\right)
  \end{equation}
  and
  \begin{equation}\label{gvn3}
    \left|\Lambda(T,T,f_2,f_3)\right|^2\leq\alpha\beta^{3}\gamma^{3}\delta^{4}\rho^{3}\|f_2\|_{\star_3}^2+O\left(\f{\ve^{\Omega(1)}}{\rho^{O(1)}}\right).
  \end{equation}
\end{lemma}

Thus, if $\ve$ is small enough, $n$ is large enough, and $S\subset T$ has no nontrivial L-shaped configurations, then one of
\[
\f{\|g_S\|_{\star_1}}{\|T\|_{\star_1}},\f{\|g_S\|_{\star_2}}{\|T\|_{\star_2}},\text{ or }\f{\|g_S\|_{\star_3}}{\|T\|_{\star_3}}
\]
is $\gg\sigma^{O(1)}$.
\begin{theorem}[$\|g_S\|_{\star_1}$, $\|g_S\|_{\star_2}$, or $\|g_S\|_{\star_3}$ large implies a density-increment]\label{starinverse}
  There exist absolute constants $0<c_1<1<c_2,c_3$ such that the following holds.  Let $d$ be a nonnegative integer, and set $\rho:=p^{-d}$.  Suppose that $A,B,C,D\subset \F_p^n$ have densities $\alpha,\beta,\gamma,\delta$, respectively, and that $\Phi\subset \F_p^n\times \F_p^n$ takes the form
  \[
  \Phi=\left\{(x,y)\in A\times \F_p^n: y\in u+V_x\right\},
  \]
  where each $V_x$ is a subspace of $\F_p^n$ of codimension $d$. Let $\sigma,\tau>0$ and
\[
  \ve\leq c_1(\sigma\tau\alpha\beta\gamma\delta\rho)^{c_2}\exp(-(64/\tau^8)^{c_3}),
\]
  and assume that 
  \[
    \|A-\alpha\|_{U^{10}(\F_p^n)},\|B-\beta\|_{U^{10}(\F_p^n)},\|C-\gamma\|_{U^{10}(\F_p^n)},\|D-\delta\|_{U^{10}(\F_p^n)},\|\Phi-\alpha\rho\|_{U^{8}(\F_p^n\times\F_p^n)}<\ve.
  \]

  Define $T\subset \F_p^n\times \F_p^n$ by~\eqref{Tform} and assume that $S\subset T$ has density $\sigma$ in $T$. Suppose that
  \[
    \|g_S\|_{\star_1}\geq\tau\alpha^{1/4}\beta^{1/2}\gamma\delta\rho^{3/4},
  \]
  \[
    \|g_S\|_{\star_2}\geq\tau\alpha^{1/2}\beta\gamma^{1/2}\delta\rho,
  \]
  or
  \[
    \|g_S\|_{\star_3}\geq\tau\alpha\beta\gamma\delta^{1/2}\rho.
  \]
  Then, $S$ has density at least $\sigma+\Omega(\tau^{O(1)})$ on a subset $T'\subset T$ of the form
  \[
    T':=\left\{(x,y)\in \F_p^n\times\F_p^n:B'(y)C'(x+y)D'(2x+y)\Phi'(x,y)=1\right\},
  \]
  where the densities of $A',B',C',D'\subset \F_p^n$ are all $\gg(\sigma\tau\alpha\beta\gamma\delta\rho)^{O(1)}$, and the set $\Phi'\subset \F_p^n\times \F_p^n$ takes the form
  \[
    \Phi'=\{(x,y)\in A'\times \F_p^n:y\in u'+V_x'\},
  \]
  where each $V_x'$ is a subspace of $\F_p^n$ of codimension $d+1$.
\end{theorem}
The first three lemmas combined tell us that if $S$ has density $\sigma$ and contains no nontrivial L-shaped configurations, then one can find a subset $T$ of $\F_p^n\times \F_p^n$ of the form~\eqref{Tform} on which $S$ has density at least $\sigma+\Omega(\sigma^{O(1)})$. The next lemma tells us that, after restricting to a product of large affine subspaces, we can get this same conclusion with $A,B,C,D,$ and $\Phi$ as pseudorandom as we need, which will allow us to continue the density-increment iteration.
\begin{lemma}\label{pseudoprop}
  There exist absolute constants $0<c_1<1<c_2,c,c'$ such that the following holds.  Let $d$ be a nonnegative integer, and set $\rho:=p^{-d}$.  Let $\ve'>0$, and suppose that $A,B,C,D\subset \F_p^n$ have densities $\alpha,\beta,\gamma,\delta$, respectively, and that $\Phi\subset \F_p^n\times \F_p^n$ and takes the form
  \[
    \Phi=\left\{(x,y)\in A\times \F_p^n:y\in u+V_x\right\},
  \]
  where each $V_x$ is a subspace of $\F_p^n$ of codimension $d$. Define $T\subset \F_p^n\times \F_p^n$ by~\eqref{Tform}, and assume that $S\subset T$ has density $\sigma+\tau$ in $T$, as well as that
  \[
    n\geq\exp^2\left(c_2\f{\exp^c(c'/\ve')}{d\tau\mu_{\F_p^n\times\F_p^n}(T)}\right).
  \]
  Then there exists a subspace $V\leq \F_p^n$ of dimension
  \[
    \dim{V}\gg n^{c_1^{O\left(\exp^c\left(c'/\ve'\right)/d\tau\mu_{\F_p^n\times\F_p^n}(T)\right)}},
  \]
  $u,w\in \F_p^n$, and $0\leq i\leq d$ such that, on setting $\mathcal{C}=(u+V)\times (w+V)$,
  \begin{itemize}
  \item $B':= B\cap(w+V)$,
  \item $C':=C\cap (u+w+V)$,
  \item $D':= D\cap (2u+w+V)$,
  \item $\Psi':=\Phi\cap\mathcal{C}$ and $\Phi':=\left\{(x,y)\in\Psi':\E_{z\in w+V}\Psi'(x,z)=p^{-i}\right\}$,
  \item $A':=\left\{x\in u+V:\E_{z\in w+V}\Phi'(x,z)\neq 0\right\}$,
  \item $\alpha'=\mu_{u+V}(A')$,
  \item $\beta':=\mu_{w+V}(B')$,
  \item $\gamma':=\mu_{u+w+V}(C')$,
  \item $\delta':=\mu_{2u+w+V}(D')$,
  \item $\rho':=p^{-i}$, and
  \item $T':=\left\{(x,y)\in\mathcal{C}:B'(y)C'(x+y)D'(2x+y)\Phi'(x,y)=1\right\}$,
  \end{itemize}
  we have
  \[
    \|A'-\rho'\|_{U^{10}(u+V)},\|B'-\beta'\|_{U^{10}(w+V)},\|C'-\gamma'\|_{U^{10}(u+w+V)},\|D'-\delta'\|_{U^{10}(2u+w+V)},\|\Phi'-\alpha'\rho'\|_{U^{8}(\mathcal{C})}<2\ve,
\]
$\alpha',\beta',\gamma',\delta'\gg\tau \mu_{\F_p^n\times \F_p^n}(T)/4$, and
\[
  \mu_{\mathcal{C}}(S\cap T')\geq\left(\sigma+\f{\tau}{4}\right)\mu_{\mathcal{C}}(T').
\]
\end{lemma}

By combining the previous four lemmas and using that L-shaped configurations are preserved
by translation and invertible linear transformations of the form $(x,y)\mapsto (Ex,Ey)$,
we thus deduce the following density-increment lemma, which we will iterate in
Section~\ref{dinc} to prove Theorem~\ref{main}.
\begin{lemma}\label{densityinc}
  There exist absolute constants $0<c_1<1<c_2,c_3,c_4,c,c'$ such that the following holds. Let $d$ be a nonnegative integer, and set $\rho:=p^{-d}$. Suppose that that $A,B,C,D\subset\F_p^n$ have densities $\alpha,\beta,\gamma,\delta$, respectively, and that $\Phi\subset\F_p^n\times\F_p^n$ takes the form
  \[
    \Phi=\{(x,y)\in A\times\F_p^n:y\in u+V_x\},
  \]
  where each $V_x$ is a subspace of $\F_p^n$ of codimension $d$. Define $T$ by~\eqref{Tform} and let $S\subset T$ have density $\sigma$ in $T$. Let $\ve\leq(\sigma\alpha\beta\gamma\delta\rho)^{c_2}\exp(-(64/\sigma)^{c_3})$, and assume that
  \[
    \|A-\alpha\|_{U^{10}(\F_p^n)},\|B-\beta\|_{U^{10}(\F_p^n)},\|C-\gamma\|_{U^{10}(\F_p^n)},\|D-\delta\|_{U^{10}(\F_p^n)},\|\Phi-\alpha\rho\|_{U^8(\F_p^n\times\F_p^n)}<\ve.
  \]

  Let $\ve'>0$, and suppose that $S$ has no nontrivial L-shaped configurations. Then either
  \begin{enumerate}
  \item $n<\exp^2\left(\f{c_4\exp^c(c'/\ve')}{(\sigma\alpha\beta\gamma\delta\rho)^{c_2}}\right)$ or
  \item there exists natural numbers $n'$ and $d'$ satisfying
    \[
      n'\gg n^{c_1^{\exp^c\left(c'/\ve'\right)/(\sigma\alpha\beta\gamma\delta\rho)^{c_2}}}
    \]
    and $0\leq d'\leq d+1$, subsets $A',B',C',D'\subset\F_p^{n'}$ of densities $\alpha',\beta',\gamma',\delta'$, respectively, a subset $\Phi'\subset\F_p^{n'}\times\F_p^{n'}$ of the form
    \[
      \Phi'=\left\{(x,y)\in A'\times\F_p^{n'}:y\in u'+V_x'\right\},
    \]
    where each $V_x'$ is a subspace of $\F_p^{n'}$ of codimension $d'$ (so that $\Phi'$ has density
    $\alpha'\rho'$, where $\rho':=p^{-d'}$), and a subset $S'\subset T'$, where
    \[
      T':=\left\{(x,y)\in\F_p^{n'}\times\F_p^{n'}:B'(y)C'(x+y)D'(2x+y)\Phi'(x,y)=1\right\},
    \]
    of density at least $\sigma+\Omega(\sigma^{O(1)})$ in $T'$, such that
    \[
      \|A'-\alpha'\|_{U^{10}(\F_p^{n'})},\|B'-\beta'\|_{U^{10}(\F_p^{n'})},\|C'-\gamma'\|_{U^{10}(\F_p^{n'})},\|D'-\delta'\|_{U^{10}(\F_p^{n'})},\|\Phi'-\alpha'\rho'\|_{U^8(\F_p^{n'}\times\F_p^{n'})}<\ve',
    \]
    $\alpha',\beta',\gamma',\delta'\geq(\sigma\alpha\beta\gamma\delta\rho)^{c_2}$, and $S'$ contains no nontrivial L-shaped configurations.
  \end{enumerate}
\end{lemma}

\section{Additional preliminaries}\label{prelims}

In this section, we present some more preliminaries that were not needed for the outline of the proof of Theorem~\ref{main}, but will be convenient to have for the proof itself. We begin with the notion of \textit{Cauchy--Schwarz complexity}, first defined by Green and Tao in~\cite{GreenTao10}.

\begin{definition}[Cauchy--Schwarz complexity]
  Let $\psi_1,\dots,\psi_d:(\F_p^n)^r\to \F_p^n$ be a collection of linear forms in $r$ variables. We say that $\psi_1,\dots,\psi_d$ has Cauchy--Schwarz complexity at most $s$ if, for every $j\in[d]$, there exists a partition of $\{\psi_1,\dots,\psi_d\}\setminus\{\psi_j\}$ into at most $s+1$ subsets such that $\psi_j$ is not contained in the linear span of any of the subsets.

  The smallest $s$ such that $\{\psi_1,\dots,\psi_d\}$ has Cauchy--Schwarz complexity at most $s$ is called the Cauchy--Schwarz complexity of $\{\psi_1,\dots,\psi_d\}$.
\end{definition}
For example, four term arithmetic progressions,
\[
  x,x+y,x+2y,x+3y,
\]
have Cauchy--Schwarz complexity $2$.

Any system of linear forms of complexity at most $s$ can be shown to be controlled by the
$U^{s+1}$-norm using repeated applications of the Cauchy--Schwarz inequality. In
particular, carrying out the proof of the generalized von Neumann theorem of Green and Tao
in~\cite{GreenTao10} in the finite field model setting (where the technical details are
much simpler) produces the following useful result.

\begin{theorem}\label{gtgvn}
  Let $\psi_1,\dots,\psi_d:(\F_p^n)^r\to \F_p^n$ be a collection of linear forms in $r$
  variables with Cauchy--Schwarz complexity at most $s$. For any $1$-bounded functions
  $f_1,\dots,f_d:\F_p^n\to\C$, we have
  \[
    \left|\E_{x_1,\dots,x_r}\prod_{j=1}^df_j(\psi_j(x_1,\dots,x_r))\right|\leq \min_{1\leq j\leq d}\|f_j\|_{U^{s+1}(\F_p^n)}.
  \]
  Further, if all of $f_1,\dots,f_d$ are supported on a set $A\subset\F_p^n$ of density
  $\alpha$ that satisfies $\|A-\alpha\|_{U^{s+1}(\mathbb{F}_p^n)}<\alpha\ve$, then
  \[
    \left|\E_{x_1,\dots,x_r}\prod_{j=1}^df_j(\psi_j(x_1,\dots,x_r))\right|\leq \alpha^{d-1}\min_{1\leq j\leq d}\|f_j\|_{U^{s+1}(\F_p^n)}+O_{d,s}\left(\ve^{\Omega_{d,s}(1)}\right).
  \]
\end{theorem}

We will use the following immediate corollary of Theorem~\ref{gtgvn} numerous times
throughout the proof of Theorem~\ref{main}.

\begin{corollary}\label{Usuniformity}
  Let $\psi_1,\dots,\psi_d:(\F_p^n)^r\to \F_p^n$ be a collection of linear forms in $r$
  variables with Cauchy--Schwarz complexity at most $s$. Suppose that
  $f_1,\dots,f_d:\F_p^n\to\C$ are $1$-bounded functions having average values
  $\alpha_1,\dots,\alpha_d$, respectively, and that
  \[
    \|f_j-\alpha_j\|_{U^{s+1}(\F_p^n)}\leq \ve_j
  \]
  for all $1\leq j\leq d$. Then
  \[
    \left|\E_{x_1,\dots,x_r}\prod_{j=1}^df_j(\psi_j(x_1,\dots,x_r))-\prod_{j=1}^d\alpha_j\right|\leq d\max_{1\leq j\leq d}\ve_j.
  \]
\end{corollary}

We will also need the Gowers--Cauchy--Schwarz inequality.

\begin{lemma}[Gowers--Cauchy--Schwarz inequality]\label{GCSU}
  Let $s$ be a natural number, $H$ be an abelian group, and $f_\omega:H\to\C$ for every $\omega\in\{0,1\}^s$. We have
  \[
    \left|\E_{x,h_1,\dots,h_s\in H}\prod_{\omega\in\{0,1\}^s}f_\omega(x+\omega\cdot(h_1,\dots,h_s))\right|\leq\prod_{\omega\in\{0,1\}^s}\|f_\omega\|_{U^s(H)}.
  \]
\end{lemma}

Next, we record the basic fact that a function on $\F_p^n$ with small $U^2$-norm has small average on affine subspaces of small codimension.

\begin{lemma}\label{subspaceavg}
  Let $f:\F_p^n\to \C$ be a $1$-bounded function satisfying $\|f\|_{U^2(\F_p^n)}<\ve$ and $w+V\subset \F_p^n$ be an affine subspace of codimension $d$. Then
  \[
    \left|\E_{x\in w+V}f(x)\right|< p^d\ve.
  \]
\end{lemma}
\begin{proof}
  The indicator function of $V$ can be written as
  \[
    \f{1}{p^d}\sum_{\xi\in V^\perp}e_p(\xi\cdot x),
  \]
  so we have that
  \[
    \left|\E_{x}f(x-w)V(x)\right|\leq\f{1}{p^d}\sum_{\xi\in V^\perp}\left|\E_{x}f(x)e_p\left(\xi\cdot x\right)\right|\leq \|f\|_{U^2(\F_p^n)},
  \]
  by the Gowers--Cauchy--Schwarz inequality. Since $|V|=|\F_p^n|/p^d$, the desired result follows.
\end{proof}

The last lemma of this section will be used to analyze the various averages appearing in the proofs of Lemmas~\ref{lower},~\ref{control}, and~\ref{starinverse}.

\begin{lemma}\label{cs2}
  Let $\psi_1,\dots,\psi_d\in\mathbb{F}_p[x_1,\dots,x_t,y]$ be a collection of linear
  forms such that the coefficient of $y$ in each of $\psi_1,\dots,\psi_d$ is nonzero, and
  $F:(\F_p^n)^t\times \F_p^n\to[0,1]$ be a function of the form
  \[
    F(\mathbf{x},y)=\prod_{j=1}^df_j(\psi_j(\mathbf{x},y))
  \]
  for some $1$-bounded functions $f_j:\F_p^n\to\C$ with average value $\beta_j$, each
  satisfying $\|f_j-\beta_j\|_{U^s(\F_p^n)}<\ve$. If the Cauchy--Schwarz complexity of the
  set of linear forms
  \begin{equation}\label{eq:linearforms}
  \bigcup_{j=1}^d\left\{\psi_j(\mathbf{x},y),\psi_j(\mathbf{x},y+h),\psi_j(\mathbf{x},y+k),\psi_j(\mathbf{x},y+h+k)\right\}
  \end{equation}
  in the variables $x_1,\dots,x_t,y,h,k$, is at most $s-1$, then
  \[
    \PP\left(\mathbf{x}\in (\F_p^n)^t:\left\|F(\mathbf{x},\cdot)-\prod_{j=1}^d\beta_j\right\|_{U^2(\F_p^n)}\geq\ve^{1/8}\right)\ll_{d}\sqrt{\ve}.
  \]
\end{lemma}
\begin{proof}
  Set $\beta:=\prod_{j=1}^d\beta_j$. Note that $\E_{\mathbf{x}\in(\F_p^n)^t}\|F(\mathbf{x},\cdot)-\beta\|_{U^2(\F_p^n)}^4$ equals
  \[
  \sum_{\omega\in\{0,1\}^4}(-\beta)^{|\omega|}\E_{\mathbf{x}\in(\F_p^n)^t}\E_{y,h,k}f_1^{\omega}(\mathbf{x},y)f_2^{\omega}(\mathbf{x},y+h)f_3^{\omega}(\mathbf{x},y+k)f_4^{\omega}(\mathbf{x},y+h+k),
  \]
  where
  \[
    f_i^{\omega}(\mathbf{x},y)=\begin{cases} 1 & \omega_i=1 \\ F(\mathbf{x},y) & \omega_i=0 \end{cases}
  \]
  for each $\omega\in\{0,1\}^4$ and $1\leq i\leq 4$. Since~\eqref{eq:linearforms} has Cauchy--Schwarz
  complexity at most $s-1$ by hypothesis, Corollary~\ref{Usuniformity} implies that
  \[
    \E_{\mathbf{x}\in(\F_p^n)^t}\E_{y,h,k}f_1^{\omega}(\mathbf{x},y)f_2^{\omega}(\mathbf{x},y+h)f_3^{\omega}(\mathbf{x},y+k)f_4^{\omega}(\mathbf{x},y+h+k)=\beta^{4-|\omega|}+O_{d}(\ve)
  \]
  for every $\omega\in\{0,1\}^4$. Thus,
  \[
    \E_{\mathbf{x}\in(\F_p^n)^t}\|F(\mathbf{x},\cdot)-\beta\|_{U^2(\F_p^n)}^4\ll_{d}\ve
  \]

  Markov's inequality then gives
  \[
    \PP(\mathbf{x}\in (\F_p^n)^t : \|F(\mathbf{x},\cdot)-\beta\|_{U^2(\F_p^n)}\geq r)\ll_{d}\f{\ve}{r^4}
  \]
  for every $r>0$. Taking $r=\ve^{1/8}$ gives the conclusion of the lemma.
\end{proof}

\section{Pseudorandomness of $\Phi$}\label{phipsd}

The main purpose of this section is to show that if
$\|\Phi-\alpha\rho\|_{U^{2s+2}(\F_p^n\times\F_p^n)}$ is small and $\Phi$ is a subset of
the form
\[
  \Phi=\left\{(x,y)\in A\times \F_p^n: y\in u+V_x\right\},
\]
where each $V_x\leq\mathbf{F}_p^n$ is a subspace of density $\rho$ in $\mathbf{F}_p^n$ and
$A\subset\mathbf{F}_p^n$ has density $\alpha$ in $\mathbf{F}_p^n$, then, whenever
$\psi_1,\dots,\psi_r\in\F_p[x_1,\dots,x_m]$ is a collection of linear forms of
Cauchy--Schwarz complexity at most $s$ and $w_1,\dots,w_r\in\F_p^n$, the affine subspaces
\[
  \left\{y\in \F_p^n:\prod_{i=1}^r\Phi(\psi_i(\mathbf{x}),y+w_i)=1\right\}
\]
typically have maximum possible codimension. This allows us to transform the condition
that $\Phi$ is $U^{8}(\F_p^n\times\F_p^n)$-pseudorandom into a more useful property for
evaluating the various averages that arise in the proof of Theorem~\ref{main}.

\begin{lemma}\label{intersection}
  For each nonnegative integer $s$ and positive integer $r$, there exist constants
  $C_{s,r},c_{s,r}>0$ such that the following holds.  Let $d$ be a nonnegative integer, and set $\rho:=p^{-d}$.  Let $\delta>0$, $A\subset\F_p^n$ have density $\alpha$, and $\Phi\subset \F_p^n\times \F_p^n$ be a set of the form
  \[
    \Phi=\left\{(x,y)\in A\times \F_p^n:y\in u+V_x\right\},
  \]
  where each $V_x\leq \F_p^n$ is a subspace of codimension $d$. Assume that 
  \[
    \|\Phi-\alpha\rho\|_{U^{2s+2}(\F_p^n\times\F_p^n)}<C_{s,r}(\alpha\delta\rho)^{c_{s,r}}.
  \]
  Let $\psi_1,\dots,\psi_r\in\F_p[x_1,\dots,x_m]$ be a collection of linear forms of Cauchy--Schwarz complexity at most $s$. Then, for all but at most a $\delta$-proportion of $m$-tuples $\mathbf{x}\in(\F_p^{n})^m$ for which $\psi_1(\mathbf{x}),\dots,\psi_r(\mathbf{x})\in A$, we must have that
  \[
    \codim{\left\{y\in \F_p^n:\prod_{i=1}^r\Phi(\psi_i(\mathbf{x}),y+w_i)=1\right\}}= rd
  \]
  for all $w_1,\dots,w_r\in\F_p^n$.
\end{lemma}

We begin by showing that if $\Phi$ is pseudorandom with respect to the $U^{s}(\F_p^n\times\F_p^n)$-norm, then $A$ is pseudorandom with respect to the $U^{s}(\F_p^n)$-norm. This result will also be useful at a few other points in the proof of Theorem~\ref{main}.

\begin{lemma}\label{PhipsdApsd}
    Let $d$ be a nonnegative integer, and set $\rho:=p^{-d}$. Let $A\subset\F_p^n$ have
    density $\alpha$ and $\Phi\subset\F_p^n\times\F_p^n$ be of the form
  \[
    \Phi=\left\{(x,y)\in A\times\F_p^n: y\in u+V_x\right\},
  \]
  where each $V_x$ is a subspace of $\F_p^n$ of codimension $d$. Then, for every natural number $s$, we have
  \[
    \|A-\alpha\|_{U^s(\F_p^n)}\leq\f{1}{\rho}\|\Phi-\alpha\rho\|_{U^s(\F_p^n\times\F_p^n)}.
  \]
\end{lemma}
\begin{proof}
  We write
  \[
    \|A-\alpha\|_{U^s(\F_p^n)}^{2^s}=\E_{x,h_1,\dots,h_s}\prod_{\omega\in\{0,1\}^s}(A-\alpha)(x+\omega\cdot(h_1,\dots,h_s))
  \]
  and then, for each $\omega\in\{0,1\}^s$, insert the identity
  \[
    (A-\alpha)(x+\omega\cdot(h_1,\dots,h_s))=\f{1}{\rho}\E_{k_\omega}(\Phi-\alpha\rho)(x+\omega\cdot(h_1,\dots,h_s),k_\omega)
  \]
  to get that $\|A-\alpha\|_{U^s(\F_p^n)}^{2^s}$ equals
  \[
    \f{1}{\rho^{2^s}}\E_{x,h_1,\dots,h_s}\E_{\substack{k_\omega\in\F_p^n \\ \omega\in\{0,1\}^s}}\prod_{\omega\in\{0,1\}^s}(\Phi-\alpha\rho)(x+\omega\cdot(h_1,\dots,h_s),k_\omega).
  \]
  We can average the above quantity over $y,\ell_1,\dots,\ell_s$ and make the change of variables $k_\omega\mapsto k_\omega+y+\omega\cdot(\ell_1,\dots,\ell_s)$ to get that it equals
  \begin{equation}\label{bigavg}
    \f{1}{\rho^{2^s}}\E_{\substack{k_\omega\in\F_p^n \\ \omega\in\{0,1\}^s}}\E_{x,y,h_1,\dots,h_s,\ell_1,\dots,\ell_s}\prod_{\omega\in\{0,1\}^s}(\Phi-\alpha\rho)((x,k_\omega+y)+\omega\cdot((h_1,\ell_1),\dots,(h_s,\ell_s))).
  \end{equation}
  Applying the Gowers--Cauchy--Schwarz inequality bounds~\eqref{bigavg} above by
  \[
    \f{1}{\rho^{2^s}}\E_{\substack{k_\omega\in\F_p^n \\ \omega\in\{0,1\}^s}}\prod_{\omega\in\{0,1\}^s}\|(\Phi-\alpha\rho)(\cdot,k_\omega+\cdot)\|_{U^{s}(\F_p^n\times\F_p^n)}=\f{1}{\rho^{2^s}}\|\Phi-\alpha\rho\|^{2^s}_{U^{s}(\F_p^n\times\F_p^n)},
  \]
  which gives us the conclusion of the lemma.
\end{proof}

To prove Lemma~\ref{intersection}, we will need the notion of an \textit{approximate polynomial} of bounded degree, which we define using the additive discrete difference operator $\partial$. For $\phi:H\to H$ and $h\in H$, define $\partial_h\phi:H\to H$ by
\[
  \partial_h\phi(x):=\phi(x)-\phi(x+h),
\]
and, for $h_1,\dots,h_s\in H$, the $s$-fold additive difference operator $\partial_{h_1,\dots,h_s}$ by
\[
  \partial_{h_1,\dots,h_s}f:=\partial_{h_1}\cdots\partial_{h_s}f.
\]

\begin{definition}
  Let $H$ be an abelian group, $A\subset H$, and $\phi:A\to H$. We say that $\phi$ is an \textit{$\ve$-approximate polynomial of degree at most $s-1$} on $A$ if
  \[
    \partial_{h_1,\dots,h_{s}}\phi(x)=0
  \]
  for at least an $\ve$-proportion of $(s+1)$-tuples $(x,h_1,\dots,h_s)\in H^{s+1}$ for which $x+\omega\cdot(h_1,\dots,h_s)\in A$ for all $\omega\in\{0,1\}^s$.
\end{definition}
We will also need the following result, which is the key combinatorial input into the
proof of Lemma~\ref{intersection}.
\begin{lemma}\label{cubes}
  For each nonnegative integer $s$, there exist constants $C_s,c_s>0$ such that the
  following holds. Let $A\subset\F_p^n$ have density $\alpha$, with
  \[
    \|A-\alpha\|_{U^{2s+2}(\F_p^n)}<C_s(\alpha\delta)^{c_s}
  \]
  and $\phi: A\to\F_p^n$ be a $\delta$-approximate polynomial of degree at most $s$ on
  $A$. Then, for at least a $\Omega_s(\delta^{O_s(1)})$-proportion of $(2s+2)$-dimensional
  parallelopipeds
  \[
    (x+\omega\cdot(h_1,\dots,h_{2s+2}))_{\omega\in\{0,1\}^{2s+2}}
  \]
  in $A^{2^{2s+2}}$, the derivative of $\phi$ on the $(2s+1)$-dimensional face
  $(x+\omega\cdot(h_1,\dots,h_{2s+2}))_{\substack{\omega\in\{0,1\}^{2s+2} \\
      \omega_i=\epsilon}}$,
  \[
    \sum_{\substack{\omega\in\{0,1\}^{2s+2} \\ \omega_{i}=\epsilon}}(-1)^{|\omega|}\phi(x+\omega\cdot(h_1,\dots,h_{2s+2})),
  \]
  vanishes for all $1\leq i\leq 2s+2$ and $\epsilon=0,1$.
\end{lemma}
\begin{proof}
  We proceed by induction on $s$, beginning with the case $s=0$. If $\phi$ is a
  $\delta$-approximate polynomial of degree at most $0$ on $A$, then
  $\E_{x,y\in A}1_{\phi(x)=\phi(y)}\geq\delta$, so that, by the pigeonhole principle,
  there exists some $z\in\F_p^n$ for which
  $\mu_{A}\left(\left\{x\in A:\phi(x)=z\right\}\right)\geq \delta$. Set
  $X:=\{x\in A:\phi(x)=z\}$, and consider the set of quadruples
  \[
    X':=\left\{(x,x+h,x+k,x+h+k)\in A^4:x,x+h,x+k,x+h+k\in X\right\}.
  \]
   Note that if $(x,x+h,x+k,x+h+k)\in X'$, then
  \[
    \phi(x)=\phi(x+h)=\phi(x+k)=\phi(x+h+k)=z,
  \]
  so certainly the derivatives
  \[
    \phi(x)-\phi(x+h),\phi(x)-\phi(x+k),\phi(x+h)-\phi(x+h+k),\text{ and }\phi(x+k)-\phi(x+h+k)
  \]
  of $\phi$ on each of the $1$-dimensional faces of the parallelopiped $(x,x+h,x+k,x+h+k)$
  vanish. Since
  \[
    \alpha\delta\leq\|X\|_{U^1(\F_p^n)}\leq\|X\|_{U^2(\F_p^n)}=\left(\f{|X'|}{p^{3n}}\right)^{1/4},
  \]
  we must have $|X'|\geq (\alpha\delta)^4p^{3n}$. Taking $c_0=8$, the total number of
  quadruples $(x,x+h,x+k,x+h+k)$ in $A^4$ is $(\alpha^4+O(C_0\alpha^8))p^{3n}$ by
  Corollary~\ref{Usuniformity}, which means that $X'$ consists of at least a
  $\delta^4/2$-proportion of parallelograms $(x,x+h,x+k,x+h+k)$ in $A^4$ if $C_0$ is
  chosen small enough. Thus, for at least a $\delta^4/2$-proportion of parallelograms
  $(x,x+h,x+k,x+h+k)$ in $A^4$, the derivative of $\phi$ on each $1$-dimensional face
  vanishes, as desired.

  Now suppose that the result holds for a general degree $s-1\geq 0$, and let $\phi$ be a
  $\delta$-approximate polynomial of degree at most $s$ on $A$. By
  Corollary~\ref{Usuniformity},
  \[
    \E_{h\in\F_p^n}\|\Delta_{h}A-\alpha^2\|_{U^{2s}(\F_p^n)}^{2^{2s}}\ll_s C_s(\alpha\delta)^{c_s},
  \]
  so that, as long as $C_s$ is sufficiently small and $c_s$ is sufficiently large, it
  follows from Markov's inequality that
  \[
    \|\Delta_hA-\alpha^2\|_{U^{2s}(\F_p^n)}<\frac{C_{s-1}(\alpha\delta/2)^{2c_{s-1}}}{2},
  \]
  and thus
\[
  |\mu_{\F_p^n}(A\cap (A-h))-\alpha^2|<\frac{C_{s-1}(\alpha\delta/2)^{2c_{s-1}}}{2}
\]
as well, for all but a $O(\delta^2)$-proportion of $h\in\F_p^n$ . Thus, for at least a
$\Omega(\delta)$-proportion of $h_{s+1}\in\F_p^n$, we have
$\|\Delta_{h_{s+1}}A-\alpha^2\|_{U^{2s}(\F_p^n)}<C_{s-1}(\alpha\delta/2)^{2c_{s-1}}/2$,
$|\mu_{\F_p^n}(A\cap (A-h_{s+1}))-\alpha^2|<C_{s-1}(\alpha\delta/2)^{2c_{s-1}}/2$, and
that the function $\partial_{h_{s+1}}\phi$ is a $O(\delta)$-approximate polynomial of
degree at most $s-1$ on $A\cap(A-h_{s+1})$.  Denoting the set of such $h_{s+1}$ by $H$, so
that $\mu_{\F_p^n}(H)\gg \delta$, the induction hypothesis then says that, for each
$h\in H$, there are at least a $\Omega_{s}(\delta^{O_{s}(1)})$-proportion of
$2s$-dimensional parallelopipeds
$(x+\omega\cdot (h_1,\dots,h_{2s}))_{\omega\in\{0,1\}^{2s}}$ in $(A\cap(A-h))^{2^{2s}}$
for which the derivative of $\partial_h\phi$ on each $(2s-1)$-dimensional face vanishes.

Summing over all $h\in H$, it follows that, for at least a
$\Omega_{s}(\delta^{O_s(1)})$-proportion of $(2s+2)$-tuples $(x,y,h_1,\dots,h_{2s})$ for
which $(x+\omega\cdot(h_1,\dots,h_{2s}))_{\omega\in\{0,1\}^{2s}}$ and
$(y+\omega\cdot(h_1,\dots,h_{2s}))_{\omega\in\{0,1\}^{2s}}$ are both in $A^{2^{2s}}$, one
has
  \[
    \partial_{h_1,\dots,\widehat{h_i},\dots,h_{2s}}\phi(x+\epsilon h_i)=\partial_{h_1,\dots,\widehat{h_i},\dots,h_{2s}}\phi(y+\epsilon h_i)
  \]
  for all $i=1,\dots,2s$ and $\epsilon=0,1$. By Corollary~\ref{Usuniformity},
  \[
    \E_{h_1,\dots,h_{2s}}\|\Delta_{h_1,\dots,h_{2s}}A-\alpha^{2^{2s}}\|_{U^{1}(\F_p^n)}^2\ll_sC_s(\alpha\delta)^{c_s},
  \]
  and so, by Markov's inequality,
  \begin{equation}\label{deltaAdensity}
    \mu_{\F_p^n}\left(\bigcap_{\omega\in\{0,1\}^{2s}}[A-\omega\cdot(h_1,\dots,h_{2s})]\right)=\alpha^{2^{2s}}+O\left([C_s(\alpha\delta)^{c_s}]^{\Omega(1)}\right)
  \end{equation}
  for all but a $O([C_s(\alpha\delta)^{c_{s}}]^{\Omega(1)})$-proportion of
  $(h_1,\dots,h_{2s})$ in $(\F_p^n)^{2s}$. By taking $C_s$ small enough and $c_s$ large
  enough, there are therefore at least a $\Omega_{s}(\delta^{O_s(1)})$-proportion of
  $2s$-tuples $(h_1,\dots,h_{2s})$ in $(\F_p^n)^{2s}$ for which~\eqref{deltaAdensity}
  holds and, for at least a $\Omega_{s}(\delta^{O_s(1)})$-proportion of pairs
  $(x,y)\in A^2$ such that $(x+\omega\cdot(h_1,\dots,h_{2s}))_{\omega\in\{0,1\}^{2s}}$ and
  $(y+\omega\cdot(h_1,\dots,h_{2s}))_{\omega\in\{0,1\}^{2s}}$ are both in $A^{2^{2s}}$,
  one also has
  \[
    \partial_{h_1,\dots,\widehat{h_i},\dots,h_{2s}}\phi(x+\epsilon h_i)=\partial_{h_1,\dots,\widehat{h_i},\dots,h_{2s}}\phi(y+\epsilon h_i)
  \]
  for all $i=1,\dots,2s$ and $\epsilon=0,1$. For each such $2s$-tuple $\mathbf{h}$, it
  follows from the pigeonhole principle that there exists a $y_{\mathbf{h}}$ in the set
  \[
    A_{\mathbf{h}}:=\bigcap_{\omega\in\{0,1\}^{2s}}(A-\omega\cdot(h_1,\dots,h_{2s}))
  \]
    such that, for at least a $\Omega_{s}(\delta^{O_s(1)})$-proportion of $x\in A_{\mathbf{h}}$, one has
  \[
    \partial_{h_1,\dots,\widehat{h_i},\dots,h_{2s}}\phi(x+\epsilon h_i)=\partial_{h_1,\dots,\widehat{h_i},\dots,h_{2s}}\phi(y_{\mathbf{h}}+\epsilon h_i)
  \]
  for all $i=1,\dots,2s$ and $\epsilon=0,1$.

  Now set $v_{i,\epsilon,\mathbf{h}}:=\partial_{h_1,\dots,\widehat{h_i},\dots,h_{2s}}\phi(y_{\mathbf{h}}+\epsilon h_i)$,
  \[
    X_{\mathbf{h}}:=\left\{x\in A_{\mathbf{h}}:\partial_{h_1,\dots,\widehat{h_i},\dots,h_{2s}}\phi(x+\epsilon h_i)=v_{i,\epsilon,\mathbf{h}}\text{ for all }i=1,\dots,2s\text{ and }\epsilon=0,1\right\},
  \]
  so that $\mu_{A_{\mathbf{h}}}(X_{\mathbf{h}})\gg_s\delta^{O_s(1)}$, and
  \[
    X_{\mathbf{h}}':=\left\{(x,x+k,x+k',x+k+k')\in A_{\mathbf{h}}^4:x,x+k,x+k',x+k+k'\in X_{\mathbf{h}}\right\}.
  \]
  Note that $(x,x+k,x+k',x+k+k')\in X'_{\mathbf{h}}$ if and only if
  \[
    \partial_{h_1,\dots,\widehat{h_i},\dots,h_{2s}}\phi(x+\epsilon h_i+\omega'\cdot(k,k'))=v_{i,\epsilon,\mathbf{h}}
  \]
  for all $i=1,\dots,2s$, $\epsilon=0,1$, and $\omega'\in\{0,1\}^2$. Thus, whenever $(x,x+k,x+k',x+k+k')\in X'_{\mathbf{h}}$,  we have
  \begin{align*}
    \partial_{h_1,\dots,h_{2s},k}\phi(x)=\partial_{h_1,\dots,h_{2s},k'}\phi(x)=\partial_{h_1,\dots,h_{2s},k}\phi(x+k')&=\partial_{h_1,\dots,h_{2s},k'}\phi(x+k) \\
                                                                                                             &=v_{1,0,\mathbf{h}}-v_{1,1,\mathbf{h}}-(v_{1,0,\mathbf{h}}-v_{1,1,\mathbf{h}}) \\
    &=0
  \end{align*}
  and
  \[
    \partial_{h_1,\dots,\widehat{h_i},\dots,h_{2s},k,k'}\phi(x+\epsilon h_i)=v_{i,\ve,\mathbf{h}}-v_{i,\ve,\mathbf{h}}-v_{i,\ve,\mathbf{h}}+v_{i,\ve,\mathbf{h}}=0
  \]
  for all $i=1,\dots,2s$ and $\epsilon=0,1$. That is, the derivative of $\phi$ vanishes on
  all $(2s+1)$-dimensional faces of the $(2s+2)$-dimensional parallelopiped
  $(x+\omega\cdot(h_1,\dots,h_{2s},k,k'))_{\omega\in\{0,1\}^{2s+2}}$.

  As in the $s=0$ case,
  \[
    \delta^{O_s(1)}\mu_{\F_p^n}(A_{\mathbf{h}})\ll_s\|X_{\mathbf{h}}\|_{U^1(\F_p^n)}\leq \|X_{\mathbf{h}}\|_{U^2(\F_p^n)}=\left(\f{\# X'_{\mathbf{h}}}{p^{3n}}\right)^{1/4},
  \]
  so that
  \[
    \# X'_{\mathbf{h}}\gg_s \delta^{O_s(1)}\mu_{\F_{p}^n}(A_{\mathbf{h}})^4p^{3n}
  \]
  for at least a $\Omega_s(\delta^{O_s(1)})$-proportion of $2s$-tuples $\mathbf{h}$. Each
  ordered quadruple $(x,\mathbf{h},k,k')$ for which
  $(x,x+k,x+k',x+k+k')\in X_{\mathbf{h}}'$ corresponds to a unique $(2s+2)$-dimensional
  parallelopiped
  \[
    P(x,\mathbf{h},k,k')=(x+\omega\cdot(h_1,\dots,h_{2s},k,k'))_{\omega\in\{0,1\}^{2s+2}}
  \]
  in $A^{2s+2}$. Thus
  \begin{align*}
    \#\{P(x,\mathbf{h},k,k'):\mathbf{h}\in(\F_p^n)^{2s}\text{ and }&(x,x+k,x+k',x+k+k')\in
                                                         X_{\mathbf{h}}'\} \\
    &\gg_s\delta^{O_s(1)}\left(\alpha^{2^{2s+2}}+O\left([C_s(\alpha\delta)^{c_s}]^{\Omega(1)}\right)\right)p^{(2s+3)n}.
  \end{align*}
  In comparison, the number of $(2s+2)$-dimensional parallelopipeds in $A$ is
  \[
    \left(\alpha^{2^{2s+2}}+O\left(C_s(\alpha\delta)^{c_s}\right)\right)p^{(2s+3)n}
  \]
    by Corollary~\ref{Usuniformity}. The conclusion of
  the lemma now follows as long as $C_s$ is sufficiently small and $c_s$ is sufficiently large.
\end{proof}
With a bit more work, it is possible to prove a version of Lemma~\ref{cubes} with $(2s+2)$-dimensional parallelopipeds replaced by $(s+2)$-dimensional parallelopipeds (which is optimal), and thus a version of Lemma~\ref{intersection} with the $U^{2s+2}$-norm replaced by the $U^{s+2}$-norm, but this would make a negligible difference in Theorem~\ref{main}.

Now we can prove Lemma~\ref{intersection}.
\begin{proof}[Proof of Lemma~\ref{intersection}]
  We proceed by induction on $r$ and $s$, beginning with the $r=1$, $s=0$ case\footnote{Note that if a system of $r$ linear forms has finite Cauchy--Schwarz complexity, then it has Cauchy--Schwarz complexity at most $r-1$.}. Since $\codim\{y\in\F_p^n:\Phi(x,y)=1\}=d$ for all $x\in A$, certainly
  \[
    \codim\{y\in\F_p^n:\Phi(\psi_1(\mathbf{x}),y+w_1)=1\}=\codim\left(\{y\in\F_p^n:\Phi(\psi_1(\mathbf{x}),y)=1\}-w_1\right)=d
  \]
  for all $\mathbf{x}\in(\F_p^n)^m$ for which $\psi_1(\mathbf{x})\in A$, and this case follows trivially without even needing the assumption that $\|\Phi-\alpha\rho\|_{U^2(\F_p^n\times\F_p^n)}$ is small.

  Now let $r\geq 2$ or $s\geq 1$, and assume that the result holds for all pairs of integers $(r',s')$ satisfying
  \begin{enumerate}
  \item $0\leq r'<r$ and $1\leq s'\leq s$ or
  \item $1\leq s'<s$,
  \end{enumerate}
  and let $C'$ be at most the minimum of $C_{r',s'}$ and $c'$ be at least the maximum of
  $c_{r',s'}$ over all such pairs with $r'<\max\left(r,2^{s+1}\right)$. As long as
  $\|\Phi-\alpha\rho\|_{U^{2s+2}(\F_p^n\times\F_p^n)}<C'(\alpha\delta\rho^{2r}/2)^{c'2^r}$,
  it follows from the induction hypothesis that for all but a $O(\delta)$-proportion of
  $\mathbf{x}\in(\F_p^n)^m$ for which $\psi_1(\mathbf{x}),\dots,\psi_r(\mathbf{x})\in A$,
  we must have
\[
\E_{y}\prod_{i=1}^r(\Phi-\rho)(\psi_i(\mathbf{x}),y+w_i)=\E_{y}\prod_{i=1}^r\Phi(\psi_i(\mathbf{x}),y+w_i)-\rho^r
\]
for all $w_1,\dots,w_r\in\F_p^n$. If the codimension of some
$\{y\in\F_p^n:\prod_{i=1}^{r}\Phi(\psi_i(\mathbf{x}),y+w_i)=1\}$ is not $rd$ for one of
these typical $\mathbf{x}$, then it is either $n$ or at most $rd-1$, which means that
\begin{equation}\label{phirho}
  \left|\E_{y}\prod_{i=1}^r(\Phi-\rho)(\psi_i(\mathbf{x}),y+w_i)\right|\geq\frac{\rho^r}{2}
\end{equation}
in either case, since $p\geq 2$. Squaring both sides of~\eqref{phirho}, multiplying by
$\prod_{i=1}^rA(\psi_i(\mathbf{x}))$, averaging over all $\mathbf{x}\in(\F_p^n)^m$,
swapping the order of summation, and applying Lemma~\ref{PhipsdApsd} (to deduce the
uniformity of $A$) and Theorem~\ref{gtgvn} yields
\[
\E_{y,z}\|(\Phi-\rho A)(\cdot,y)(\Phi-\rho A)(\cdot,z)\|_{U^{s+1}(\F_p^n)}\gg\delta\rho^{2r}\left(\alpha+O_{s,r}\left((C')^{\Omega_s(1)}\alpha^{c'-1}\right)\right).
\]
By H\"older's inequality, we then have
\[
\E_{x,h_1,\dots,h_{s+1}}\left|\E_{y}\Delta_{(h_1,0),\dots,(h_{s+1},0)}(\Phi-\rho A)(x,y)\right|^2\gg_s\left(\delta\rho^{2r}\right)^{2^{s+1}}\left(\alpha^{2^{s+1}}+O_{s,r}\left((C')^{\Omega_s(1)}\alpha^{c'-1}\right)\right).
\]
It follows from this, the induction hypothesis, our assumption that
$\|\Phi-\alpha\rho\|_{U^{2s+2}(\F_p^n\times\F_p^n)}<C'(\alpha\delta\rho^{2r}/2)^{c'2^r}$,
and Lemma~\ref{PhipsdApsd} that
\[
\E_{x,h_1,\dots,h_{s+1}}\left|\E_{y}\Delta_{(h_1,0),\dots,(h_{s+1},0)}\Phi(x,y)-\rho^{2^{s+1}}A(x)\right|^2\gg_s(\delta\rho^{2r})^{2^{s+1}}\left(\alpha^{2^{s+1}}+O_{s,r}\left((C')^{\Omega_s(1)}\alpha^{c'-1}\right)\right),
\]
since the Cauchy--Schwarz complexity of any proper subset of
\[
  \left\{x+\omega\cdot(h_1,\dots,h_{s+1}):\omega\in\{0,1\}^{s+1}\right\}
\]
is at most $s-1$.

Thus, by taking $C'$ sufficiently small and $c'$ sufficiently large, we get that
\begin{equation}\label{codimwrong}
\codim\left\{y\in\F_p^n:\Delta_{(h_1,0),\dots,(h_{s+1},0)}\Phi(x,y)=1\right\}\neq 2^{s+1}d
\end{equation}
for at least a $\Omega_s((\delta\rho)^{O_s(1)})$-proportion of $(s+2)$-tuples $(x,h_1,\dots,h_{s+1})\in(\F_p^n)^{s+2}$ for which $(x+\omega\cdot(h_1,\dots,h_{s+1}))_{\omega\in\{0,1\}^{s+1}}\in A^{2^{s+1}}$. The condition~\eqref{codimwrong} implies that, for each such $(x,h_1,\dots,h_{s+1})$, there exist vectors $v_{x,\mathbf{h},\omega}\in V^{\perp}_{x+\omega\cdot\mathbf{h}}$, $\omega\in\{0,1\}^{s+1}$, not all of which are zero, such that
\[
\sum_{\omega\in\{0,1\}^{s+1}}v_{x,\mathbf{h},\omega}=0.
\]
Fix a basis $\{\gamma_{z,1},\dots,\gamma_{z,d}\}$ of $V_z^\perp$ for each $z\in\F_p^n$. We can write every vector $v_{x,\mathbf{h},\omega}$ in terms of this basis, giving us that
\[
\sum_{\omega\in\{0,1\}^{s+1}}\sum_{j=1}^db_{x,\mathbf{h},\omega,j}\gamma_{x+\omega\cdot\mathbf{h},j}=0
\]
for some $2^{s+1}d$-tuple of constants $(b_{x,\mathbf{h},\omega,j})_{\omega\in\{0,1\}^{s+1},1\leq j\leq d}$, not all of which are zero.

We apply the pigeonhole principle to deduce that there is some $2^{s+1}d$-tuple of constants $(b_{\omega,j})_{\omega\in\{0,1\}^{s+1},1\leq j\leq d}$, not all of which are zero, such that
\[
\sum_{\omega\in\{0,1\}^{s+1}}\sum_{j=1}^db_{\omega,j}\gamma_{x+\omega\cdot\mathbf{h},j}=0
\]
for at least a $\Omega_s((\delta\rho)^{O_s(1)})$-proportion of $(s+2)$-tuples $(x,h_1,\dots,h_{s+1})\in(\F_p^n)^{s+2}$ for which $(x+\omega\cdot(h_1,\dots,h_{s+1}))_{\omega\in\{0,1\}^{s+1}}\in A^{2^{s+1}}$. Defining
\[
  \phi_{\omega}(z):=\sum_{j=1}^db_{\omega,j}\gamma_{z,j}
\]
for each $\omega\in\{0,1\}^{s+1}$, the above says that, for at least a
$\Omega_s((\delta\rho)^{O_s(1)})$-proportion of $(s+2)$-tuples
$(x,h_1,\dots,h_{s+1})\in(\F_p^n)^{s+2}$ for which
$(x+\omega\cdot(h_1,\dots,h_{s+1}))_{\omega\in\{0,1\}^{s+1}}\in A^{2^{s+1}}$, we must have
\[
\sum_{\omega\in\{0,1\}^{s+1}}\phi_\omega(x+\omega\cdot\mathbf{h})=0,
\]
where at least one of the functions $\phi_{\omega}:\F_p^n\to\F_p^n$ does not have the zero
vector in its image (because $\phi_\omega(z)$ is always a nontrivial linear combination of
linearly independent vectors). Let $\phi$ be any such $\phi_\omega$. It then follows by
applying Corollary~\ref{Usuniformity} to the inside average of
\[
\E_{y}\E_{x,h_1,\dots,h_{s+1}}e_p\left(y\cdot\sum_{\omega\in\{0,1\}^{s+1}}\phi_\omega(x+\omega\cdot \mathbf{h})\right)A\left(\phi_\omega(x+\omega\cdot \mathbf{h})\right)
\]
that $\partial_{h_1,\dots,h_{s+1}}\phi(x)=0$ for at least a
$\Omega_s((\delta\rho)^{O_s(1)})$-proportion of $(s+2)$-tuples
$(x,h_1,\dots,h_{s+1})\in(\F_p^n)^{s+2}$ for which
$(x+\omega\cdot(h_1,\dots,h_{s+1}))_{\omega\in\{0,1\}^{s+1}}\in A^{2^{s+1}}$, again
provided that $C'$ is sufficiently small and $c'$ is sufficiently large. That is, $\phi$
is a $\Omega_s((\delta\rho)^{O_s(1)})$-approximate polynomial of degree at most $s$ on
$A$.

If $C'$ is small enough and $c'$ is large enough, Lemmas~\ref{PhipsdApsd} and~\ref{cubes} then imply that for at least a $\Omega_{s}((\alpha\delta\rho)^{O_s(1)})$-proportion of $(2s+2)$-dimensional parallelopipeds $(x+\omega\cdot(h_1,\dots,h_{2s+2}))_{\omega\in\{0,1\}^{2s+2}}$ in $A^{2^{2s+2}}$, the derivative of $\phi$ on each $(2s+1)$-dimensional face vanishes, i.e.,
\[
\sum_{\substack{\omega\in\{0,1\}^{2s+2} \\ \omega_i=\epsilon}}(-1)^{|\omega|}\phi(x+\omega\cdot(h_1,\dots,h_{2s+2}))=0
\]
for all $i=1,\dots,2s+2$ and $\epsilon=0,1$. Call the set of $(2s+3)$-tuples
$(x,h_1,\dots,h_{2s+2})$ corresponding to such $(2s+2)$-dimensional parallelopipeds
$X$. Consider $\|\Phi-\rho A\|_{U^{2s+2}(\F_p^n\times\F_p^n)}^{2^{2s+2}}$, which, plugging
in the expression
\[
  \rho A(x)\sum_{0\neq v\in V_x^\perp}e_p(v\cdot (y-u))
\]
for $\Phi-\rho A$, equals
\[
  \rho^{2^{2s+2}}\E_{x,h_1,\dots,h_{2s+2}}\Delta_{h_1,\dots,h_{2s+2}}A(x)\sum_{\substack{0\neq v_\omega\in V_{x+\omega\cdot\mathbf{h}}^\perp \\ \omega\in\{0,1\}^{2s+2}}}\prod_{\substack{1\leq i\leq 2s+2 \\ \epsilon=0,1}}1_{\mathbf{0}}\left(\sum_{\substack{\omega\in\{0,1\}^{2s+2} \\ \omega_i=\epsilon}}(-1)^{|\omega|}v_\omega\right).
\]
The above has size
$\gg_s(p-1)(\alpha\rho)^{2^{2s+2}}(\alpha\delta\rho)^{O_s(1)}\gg_s(\alpha\delta\rho)^{O_s(1)}$,
coming from the contribution of $v_\omega=\lambda\phi(x+\omega\cdot(h_1,\dots,h_{2s+2}))$
for each $\lambda\in\F_p^\times$ and $(x,h_1,\dots,h_{2s+2})\in X$. On the other hand, we
have
\[
  \|\Phi-\rho A\|_{U^{2s+2}(\F_p^n\times\F_p^n)}=\|\Phi-\rho\alpha+\rho\alpha-\rho A\|_{U^{2s+2}(\F_p^n\times\F_p^n)}\leq\|\Phi-\rho\alpha\|_{U^{2s+2}(\F_p^n\times\F_p^n)}+\rho\|A-\alpha\|_{U^{2s+2}(\F_p^n)},
\]
so that
\[
  (\alpha\delta\rho)^{O_s(1)}\ll_s \|\Phi-\rho A\|_{U^{2s+2}(\F_p^n\times\F_p^n)}< 2C_{s,r}(\alpha\delta\rho)^{c_{s,r}}.
\]
Taking $C_{s,r}$ sufficiently small and $c_{s,r}$ sufficiently large will thus yield a contradiction if
\[
  \codim\left\{y\in\F_p^n:\prod_{i=1}^r\Phi(\psi_i(\mathbf{x}),y+w_i)=1\right\}\neq rd
\]
for some $w_1,\dots,w_r\in\F_p^n$ for a $\delta$-proportion of $\mathbf{x}\in(\F_p^n)^m$ for which $\psi_1(\mathbf{x}),\dots,\psi_r(\mathbf{x})\in A$.
\end{proof}

\section{Control by directional uniformity norms}\label{gvn}

This section is devoted to proving Lemmas~\ref{lower} and~\ref{control}. Our arguments mostly consist of careful, repeated applications of the Cauchy--Schwarz inequality to ensure that there is no loss of density factors and using the results of Sections~\ref{prelims} and~\ref{phipsd} to analyze the resulting averages. As a simple warm-up, we begin by showing that if $T\subset \F_p^n\times\F_p^n$ has the form~\eqref{Tform} and $B,C,$ and $D$ are sufficiently pseudorandom, then $T$ has density close to the product density $\alpha\beta\gamma\delta\rho$.

\begin{lemma}\label{Tdensity}
   Let $d$ be a nonnegative integer, and set $\rho:=p^{-d}$. Let $\ve>0$ and assume that $A,B,C,D\subset \F_p^n$ have densities $\alpha,\beta,\gamma,$ and $\delta$, respectively, and satisfy
  \[
    \|B-\beta\|_{U^{4}(\F_p^n)},\|C-\gamma\|_{U^{4}(\F_p^n)},\|D-\delta\|_{U^{4}(\F_p^n)}<\ve
  \]
  and that $\Phi\subset \F_p^n\times \F_p^n$ takes the form
  \[
    \Phi=\left\{(x,y)\in A\times \F_p^n: y\in u+V_x\right\},
  \]
  where each $V_x$ is a subspace of $\F_p^n$ of codimension $d$. Then the set $T\subset \F_p^n\times \F_p^n$ defined by~\eqref{Tform} has density
  \[
    \alpha\beta\gamma\delta\rho+O(\ve^{1/8}).
  \]
\end{lemma}
\begin{proof}
  The density of $T$ in $\F_p^n\times\F_p^n$ can be written as
  \[
    \E_{x,y}F(x,y)\Phi(x,y),
  \]
  where $F(x,y):=B(y)C(x+y)D(2x+y)$. Set $L(x,y):=\{y,x+y,2x+y\}$. Lemma~\ref{cs2} says that
  \[
    \PP(x\in \F_p^n:\|F(x,\cdot)-\beta\gamma\delta\|_{U^2(\F_p^n)}>\ve^{1/8})\ll\sqrt{\ve},
  \]
  since
  \[
  L(x,y) \cup L(x,y+h) \cup L(x,y+k) \cup L(x,y+h+k)
  \]
  has Cauchy--Schwarz complexity at most $3$. Thus, Lemma~\ref{subspaceavg} yields
  \[
    \E_{y}F(x,y)\Phi(x,y)=\left(\beta\gamma\delta\rho+O(\ve^{1/8})\right)A(x)
  \]
  for all but a $O(\sqrt{\ve})$-proportion of $x\in \F_p^n$, so that
  \[
  \E_{x,y}F(x,y)\Phi(x,y)=\alpha\beta\gamma\delta\rho+O(\ve^{1/8}).
  \]
\end{proof}

Now we can prove Lemma~\ref{lower}.
\begin{proof}[Proof of Lemma~\ref{lower}]
  The quantity of interest $\Lambda(T,T,T,S)$ is
  \[
    \E_{x,y,z} B(y+z)B(y+2z)C(x+y)C(x+y+2z)D(2x+y)D(2x+y+z)\Phi(x,y)\Phi(x,y+z)S(x+z,y),
  \]
  which, after a change of variables, can be written as
  \[
    \E_{x,y}S(x,y)\mu(x,y),
  \]
  where $\mu(x,y)$ equals
  \[
    \E_{z}B(y+z)B(y+2z)C(x+y-z)C(x+y+z)D(2x+y-2z)D(2x+y-z)\Phi(x-z,y)\Phi(x-z,y+z).
  \]
  We will show that $\mu(x,y)$ is very close to the constant value $\alpha\beta^2\gamma^2\delta^2\rho^2$ for almost every pair $(x,y)\in \F_p^n\times\F_p^n$, from which it will then follow that $\Lambda(T,T,T,S)$ is close to $\sigma\alpha^2\beta^3\gamma^3\delta^3\rho^3$.

  The first moment $\E_{x,y}\mu(x,y)$ equals
  \[
    \E_{x,y,z}B(y+z)B(y+2z)C(x+y)C(x+y+2z)D(2x+y)D(2x+y+z)\Phi(x,y)\Phi(x,y+z).
  \]
  Applying Lemma~\ref{cs2} yields
  \[
    \PP\left((x,y)\in \F_p^n\times\F_p^n:\|F(x,y,\cdot)-\beta^2\gamma\delta\|_{U^2(\F_p^n)}\geq\ve^{1/8}\right)\ll\sqrt{\ve},
  \]
  where
  \[
    F(x,y,z):=B(y+z)B(y+2z)C(x+y+2z)D(2x+y+z).
  \]
  It therefore follows from Lemma~\ref{subspaceavg} that
  \[
  \E_{z}F(x,y,z)\Phi(x,y+z)=\left(\beta^2\gamma\delta\rho+O(\ve^{1/8})\right)A(x)
  \]
  for all but a $O(\sqrt{\ve})$-proportion of $(x,y)\in\F_p^n\times\F_p^n$. Thus,
  \[
  \E_{x,y}\mu(x,y)=\beta^2\gamma\delta\rho\E_{x,y}C(x+y)D(2x+y)\Phi(x,y)+O\left(\ve^{1/8}\right).
  \]
  By arguing as in the proof of Lemma~\ref{Tdensity}, we have
  \[
    \E_{x,y}C(x+y)D(2x+y)\Phi(x,y)=\alpha\gamma\delta\rho+O\left(\ve^{1/8}\right),
  \]
  and thus conclude that
  \[
  \E_{x,y}\mu(x,y)=\alpha\beta^2\gamma^2\delta^2\rho^2+O\left(\ve^{1/8}\right).
  \]
  
  The second moment $\E_{x,y}\mu(x,y)^2$ equals
  \begin{align*}
    \E_{x,y,z,h}\bigg(&\Delta_{h}B(y+z)\Delta_{2h}B(y+2z) \Delta_{-h}C(x+y)\Delta_{h}C(x+y+2z) \\
                                 &\Delta_{-2h}D(2x+y)\Delta_{-h}D(2x+y+z)\Delta_{(-h,0)}\Phi(x,y)\Delta_{(-h,h)}\Phi(x,y+z)\bigg).
  \end{align*}
  Applying Lemma~\ref{cs2} again yields
  \[
    \PP\left((x,y,h)\in\F_p^n\times\F_p^n\times\F_p^n:\|G(x,y,h,\cdot)-\beta^4\gamma^2\delta^2\|_{U^2(\F_p^n)}>\ve^{1/8}\right)\ll\sqrt{\ve}
  \]
  where
  \[
    G(x,y,h,z):=\Delta_{h}B(y+z)\Delta_{2h}B(y+2z)\Delta_{h}C(x+y+2z)\Delta_{-h}D(2x+y+z),
  \]
  and applying Lemma~\ref{intersection} yields
  \[
  \PP\left((x,x-h)\in A\times A:\codim\left\{z\in\F_p^n:\Delta_{(-h,h)}\Phi(x,y+z)=1\right\}\neq 2d\right)\ll\f{\ve^{\Omega(1)}}{\rho^{O(1)}}
\]
for all $y\in\F_p^n$. Thus, by Lemma~\ref{subspaceavg},
  \[
  \E_{y}G(x,y,h,z)\Delta_{(-h,h)}\Phi(x,y+z)=\beta^4\gamma^2\delta^2\rho^2+O(\ve^{1/8})
  \]
  for all but a $O(\ve^{\Omega(1)}/\rho^{O(1)})$-proportion of $(x,x+h,y)\in A\times A\times\F_p^n$, so that
  \[
  \E_{x,y}\mu(x,y)^2=\beta^4\gamma^2\delta^2\rho^2\E_{x,y,h}\Delta_{-h}C(x+y)\Delta_{-2h}D(2x+y)\Delta_{(-h,0)}\Phi(x,y)+O\left(\f{\ve^{\Omega(1)}}{\rho^{O(1)}}\right).
\]
By Lemmas~\ref{cs2} and~\ref{intersection} again, we have
\[
  \PP\left((x,h)\in\F_p^n\times\F_p^n:\|H(x,h,\cdot)-\gamma^2\delta^2\|_{U^2(\F_p^n)}>\ve^{1/8}\right)\ll\sqrt{\ve}
\]
and
\[
  \PP\left((x,x-h)\in A\times A:\codim\left\{y\in\F_p^n:\Delta_{(-h,0)}\Phi(x,y)=1\right\}\neq 2d\right)\ll\f{\ve^{\Omega(1)}}{\rho^{O(1)}},
\]
where
\[
  H(x,h,y):=\Delta_{-h}C(x+y)\Delta_{-2h}D(2x+y),
\]
so that
\[
  \E_{x,y,h}\Delta_{-h}C(x+y)\Delta_{-2h}D(2x+y)\Delta_{(-h,0)}\Phi(x,y)=\gamma^2\delta^2\rho^2\E_{x,h}\Delta_{-h}A(x)+O\left(\f{\ve^{\Omega(1)}}{\rho^{O(1)}}\right)
\]
by Lemma~\ref{subspaceavg}. Using Corollary~\ref{Usuniformity} to estimate $\E_{x,h}\Delta_{-h}A(x)$, we thus conclude that
\[
  \E_{x,y}\mu(x,y)^2 = \alpha^2\beta^4\gamma^4\delta^4\rho^4+O\left(\f{\ve^{\Omega(1)}}{\rho^{O(1)}}\right)
\]

  Our estimates for the first and second moments of $\mu$ imply that $\mu$ has variance $\E_{x,y}|\mu(x,y)-\alpha\beta^2\gamma^2\delta^2\rho|^2\ll\ve^{\Omega(1)}/\rho^{O(1)}$. It follows that
  \begin{align*}
    \E_{x,y}S(x,y)\mu(x,y) &= \sigma\alpha^2\beta^3\gamma^3\delta^3\rho^3+O\left(\E_{x,y}|\mu(x,y)-\alpha\beta^2\gamma^2\delta^2\rho^2|\right) \\
                           &= \sigma\alpha^2\beta^3\gamma^3\delta^3\rho^3+O\left([\E_{x,y}|\mu(x,y)-\alpha\beta^2\gamma^2\delta^2\rho^2|^2]^{1/2}\right) \\
                           &= \sigma\alpha^2\beta^3\gamma^3\delta^3\rho^3+O\left(\f{\ve^{\Omega(1)}}{\rho^{O(1)}}\right).
  \end{align*}
  When $c_1$ is sufficiently small and $c_2$ is sufficiently large, this gives the desired lower bound for $\Lambda(T,T,T,S)$.
\end{proof}

To finish this section, we prove Lemma~\ref{control}.

\begin{proof}[Proof of Lemma~\ref{control}]
  We prove~\eqref{gvn1},~\eqref{gvn2}, and then~\eqref{gvn3}, proceeding in decreasing order of the number of applications of Cauchy--Schwarz required. For~\eqref{gvn1}, we make the change of variables $z\mapsto z-x-y$ to write $\Lambda(f_0,f_1,f_2,f_3)$ as
  \[
    \E_{x,y,z}f_0(x,y)f_1(x,z-x)f_2(x,2z-2x-y)f_3(z-y,y),
  \]
  which, by applying the Cauchy--Schwarz inequality in the $x$ and $z$ variables, has modulus squared bounded by
  \[
    \mu_{\F_p^n\times\F_p^n}(T)\cdot\E_{x,z}B(z-x)D(x+z)\left|\E_yf_0(x,y)f_2(x,2z-2x-y)f_3(z-y,y)\right|^2.
  \]
  The first factor equals $\alpha\beta\gamma\delta\rho+O(\ve^{1/8})$ by Lemma~\ref{Tdensity}.   Expanding the square and making a change of variables, the second factor equals
  \[
    \E_{x,y,z,h_1}B(z-x)D(x+z)\Delta_{(0,h_1)}f_0(x,y)\Delta_{(0,-h_1)}f_2(x,2z-2x-y)\Delta_{(-h_1,h_1)}f_3(z-y,y).
  \]
  By another application of the Cauchy--Schwarz inequality in the $y,z,$ and $h_1$
  variables, the modulus squared of this is at most
  \begin{equation}\label{2.5b}
    \E_{y,z,h_1}\Delta_{h_1}B(y)C(z+y)\Delta_{-h_1}D(2z+y)\Delta_{(-h_1,h_1)}\Phi(z,y)
  \end{equation}
  times
  \begin{align*}
    \E_{y,z,h_1} \bigg(&C(z)\Delta_{(-h_1,h_1)}\Phi(z-y,y)\left|\E_{x}B(z-x)D(x+z)\Delta_{(0,h_1)}f_0(x,y)\Delta_{(0,-h_1)}f_2(x,2z-2x-y)\right|^2\bigg).
  \end{align*}
  The first factor~\eqref{2.5b} can be estimated in the same manner as the averages appearing in the proof of Lemma~\ref{lower}, and equals $\alpha^2\beta^2\gamma\delta^2\rho^2+O(\ve^{\Omega(1)}/\rho^{O(1)})$.   Expanding the square and making a change of variables, we get that the second factor equals
  \begin{align*}
    \E_{x,y,z,h_1,h_2}\bigg(&\Delta_{-h_2}B(y-z)C(x+y-z)\Delta_{h_2}D(2x+y-z)\Delta_{(-h_1,h_1)}\Phi(x+z,y-2z) \\
                  &\Delta_{(0,h_1),(h_2,0)}f_0(x,y-2z)\Delta_{(0,-h_1),(h_2,-2h_2)}f_2(x,y)\bigg).
  \end{align*}
  A final application of the Cauchy--Schwarz inequality in the $x,y,h_1,$ and $h_2$ variables bounds the modulus squared of this by
  \begin{equation}\label{2.5c}
    \E_{x,y,h_1,h_2}\Delta_{-h_1,-2h_2}B(y)\Delta_{-h_1,-h_2}C(x+y)\Delta_{-h_1}D(2x+y)\Delta_{(0,-h_1),(h_2,-2h_2)}\Phi(x,y)
  \end{equation}
  times
  \begin{align*}
    \E_{x,y,h_1,h_2}&\Delta_{-h_1,-2h_2}B(y)\Delta_{-h_1,-h_2}C(x+y)\Delta_{-h_1}D(2x+y)\Delta_{(h_2,-2h_2)}\Phi(x,y)\\
    &\cdot|\E_{z}\Delta_{-h_2}B(y-z)C(x+y-z)\Delta_{h_2}D(2x+y-z)\Delta_{(-h_1,h_1)}\Phi(x+z,y-2z)\\
    &\ \ \ \ \ \ \ \Delta_{(0,h_1),(h_2,0)}f_0(x,y-2z)|^2
  \end{align*}
  By Lemmas~\ref{cs2} and~\ref{intersection}, we have
  \[
    \PP\left((x,y,h_2)\in\F_p^n\times\F_p^n\times\F_p^n:\|F(x,y,h_2,\cdot)-\beta^2\gamma^2\delta\|_{U^2(\F_p^n)}>\ve^{1/8}\right)\ll\sqrt{\ve},
  \]
  \[
    \PP\left((x,h_2)\in\F_p^n\times\F_p^n:\|G(x,h_2,\cdot)-\beta^2\gamma^2\delta\|_{U^2(\F_p^n)}>\ve^{1/8}\right)\ll\sqrt{\ve},
  \]
  \[
    \PP\left((x,x+h_2)\in A\times A:\codim\left\{h_1\in\F_p^n:\Delta_{(h_2,-2h_2)}\Phi(x,y-h_1)=1\right\}\neq 2d\right)\ll\f{\ve^{\Omega(1)}}{\rho^{O(1)}}
  \]
  for all $y\in\F_p^n$, and
  \[
    \PP\left((x,x+h_2)\in A\times A:\codim\left\{y\in\F_p^n:\Delta_{(h_2,-2h_2)}\Phi(x,y)=1\right\}\neq 2d\right)\ll\f{\ve^{\Omega(1)}}{\rho^{O(1)}},
  \]
  where
  \[
    F(x,y,h_2,h_1):=\Delta_{-2h_2}B(y-h_1)\Delta_{-h_2}C(x+y-h_1)D(2x+y-h_1)
  \]
  and
  \[
    G(x,h_2,y):=\Delta_{-2h_2}B(y)\Delta_{-h_2}C(x+y)D(2x+y).
  \]
  It then follows from Lemma~\ref{subspaceavg} that~\eqref{2.5c} equals
  \[
    \beta^2\gamma^2\delta\rho^2\E_{x,y,h_1,h_2}\Delta_{-2h_2}B(y)\Delta_{-h_2}C(x+y)D(2x+y)\Delta_{(h_2,-2h_2)}\Phi(x,y)+O\left(\f{\ve^{\Omega(1)}}{\rho^{O(1)}}\right),
  \]
  which equals
  \[
    \beta^4\gamma^4\delta^2\rho^4\E_{x,h_2}\Delta_{h_2}A(x)+O\left(\f{\ve^{\Omega(1)}}{\rho^{O(1)}}\right)= \alpha^2\beta^4\gamma^4\delta^2\rho^4+O\left(\f{\ve^{\Omega(1)}}{\rho^{O(1)}}\right).
  \]
  
  It remains to relate
\begin{align*}
    \E_{x,y,h_1,h_2}&\Delta_{-h_1,-2h_2}B(y)\Delta_{-h_1,-h_2}C(x+y)\Delta_{-h_1}D(2x+y)\Delta_{(h_2,-2h_2)}\Phi(x,y)\\
    &\cdot|\E_{z}\Delta_{-h_2}B(y-z)C(x+y-z)\Delta_{h_2}D(2x+y-z)\Delta_{(-h_1,h_1)}\Phi(x+z,y-2z)\\
    &\ \ \ \ \ \ \ \Delta_{(0,h_1),(h_2,0)}f_0(x,y-2z)|^2
  \end{align*}
  to $\|f_0\|_{\star_1}$. Expanding the square and making a change of variables yields
  \begin{align*}
    \E_{x,y,z,h_1,h_2,h_3}\big(&\Delta_{-h_1,-2h_2}B(y+2z)\Delta_{-h_2,-h_3}B(y+z) \\
                          &\Delta_{-h_1,-h_2}C(x+y+2z)\Delta_{-h_3}C(x+y+z) \\
    &\Delta_{-h_1}D(2x+y+2z)\Delta_{h_2,-h_3}D(2x+y+z)\\
                          &\Delta_{(h_2,-2h_2)}\Phi(x,y+2z)\Delta_{(-h_1,h_1),(h_3,-2h_3)}\Phi(x+z,y) \\
                          &\Delta_{(0,h_1),(h_2,0),(0,-2h_3)}f_0(x,y)\big),
  \end{align*}
  which can be written as
  \[
    \E_{x,y,h_1,h_2,h_3}\Delta_{(0,h_1),(h_2,0),(0,-2h_3)}f_0(x,y)\mu(x,y,h_1,h_2,h_3),
  \]
  where
  \begin{align*}
    \mu(x,y,h_1,h_2,h_3):= \E_{z}\big(&\Delta_{-h_1,-2h_2}B(y+2z)\Delta_{-h_2,-h_3}B(y+z) \\
                          &\Delta_{-h_1,-h_2}C(x+y+2z)\Delta_{-h_3}C(x+y+z) \\
    &\Delta_{-h_1}D(2x+y+2z)\Delta_{h_2,-h_3}D(2x+y+z)\\
                          &\Delta_{(h_2,-2h_2)}\Phi(x,y+2z)\Delta_{(-h_1,h_1),(h_3,-2h_3)}\Phi(x+z,y)\big).
  \end{align*}
We will show that, for almost every $5$-tuple $(x,y,h_1,h_2,h_3)\in (\F_p^n)^5$ for which $x,x+h_2\in A$, $\mu(x,y,h_1,h_2,h_3)$ is very close to the constant value $\alpha^4\beta^8\gamma^6\delta^6\rho^{6}$. Indeed, the first moment $\E_{\substack{x,y,h_1,h_2,h_3 \\ x,x+h_2\in A}}\mu(x,y,h_1,h_2,h_3)$ is
\begin{align*}
  \f{1}{\alpha^2+O(\ve)}\E_{x,y,z,h_1,h_2,h_3}\big(&\Delta_{-h_1,-2h_2}B(y+2z)\Delta_{-h_2,-h_3}B(y+z) \\
                          &\Delta_{-h_1,-h_2}C(x+y+2z)\Delta_{-h_3}C(x+y+z) \\
    &\Delta_{-h_1}D(2x+y+2z)\Delta_{h_2,-h_3}D(2x+y+z)\\
                          &\Delta_{(h_2,-2h_2)}\Phi(x,y+2z)\Delta_{(-h_1,h_1),(h_3,-2h_3)}\Phi(x+z,y)\big).
\end{align*}
  Lemmas~\ref{cs2} and~\ref{intersection} tell us that
  \[
    \PP\left((x,z,h_1,h_2,h_3)\in(\F_p^n)^5:\|H(x,z,h_1,h_2,h_3,\cdot)-\beta^8\gamma^6\delta^6\|_{U^2(\F_p^n)}>\ve^{1/8}\right)\ll\sqrt{\ve}
  \]
  and
  \begin{align*}
    \PP\big(&(x,x+h_2,x+z,x+z-h_1,x+z+h_3,x+z-h_1+h_3)\in A^6\\
      &:\codim\left\{y\in\F_p^n:\Delta_{(h_2,-2h_2)}\Phi(x,y+2z)\Delta_{(-h_1,h_1),(h_3,-2h_3)}\Phi(x+z,y)=1\right\}\neq 6d\big)\ll\f{\ve^{\Omega(1)}}{\rho^{O(1)}},
  \end{align*}
  where
  \begin{align*}
    H(x,y,h_1,h_2,h_3,z):=\big(&\Delta_{-h_1,-2h_2}B(y+2z)\Delta_{-h_2,-h_3}B(y+z)\Delta_{-h_1,-h_2}C(x+y+2z) \\
    &\Delta_{-h_3}C(x+y+z)\Delta_{-h_1}D(2x+y+2z)\Delta_{h_2,-h_3}D(2x+y+z)\big)
  \end{align*}
  so it follows from Lemma~\ref{subspaceavg} that the first moment equals
  \[
    \f{1}{\alpha^2+O(\ve)}\beta^8\gamma^6\delta^6\rho^6\E_{x,z,h_1,h_2,h_3}\Delta_{h_2}A(x)\Delta_{-h_1,h_3}A(x+z)+O\left(\f{\ve^{\Omega(1)}}{\rho^{O(1)}}\right),
  \]
  which equals
  \[
    \alpha^4\beta^8\gamma^6\delta^6\rho^6+O\left(\f{\ve^{\Omega(1)}}{\rho^{O(1)}}\right)
  \]
  by Corollary~\ref{Usuniformity}. The second moment $\E_{\substack{x,y,h_1,h_2,h_3 \\ x,x+h_2\in A}}\mu(x,y,h_1,h_2,h_3)^2$ is
  \begin{align*}
    \f{1}{\alpha^2+O(\ve)}\E_{x,y,z,h_1,h_2,h_3,k}\big(&\Delta_{-h_1,-2h_2,2k}B(y+2z)\Delta_{-h_2,-h_3,k}B(y+z)\\
                            &\Delta_{-h_1,-h_2,2k}C(x+y+2z)\Delta_{-h_3,k}C(x+y+z)\\
                                &\Delta_{-h_1,2k}D(2x+y+2z)\Delta_{h_2,-h_3,k}D(2x+y+z) \\
                                &\Delta_{(h_2,-2h_2),(0,2k)}\Phi(x,y+2z)\Delta_{(-h_1,h_1),(h_3,-2h_3),(k,0)}\Phi(x+z,y)\big),
  \end{align*}
  which, noting that
  \[
    \Delta_{(h_2,-2h_2),(0,2k)}\Phi(x,y+2z)=\Delta_{(h_2,-2h_2)}\Phi(x,y+2z)\Delta_{(h_2,-2h_2)}\Phi(x,2k-u),
  \]
  we can write as
  \begin{align*}
    \E_{x,y,z,h_1,h_2,h_3,k}\big(&\Delta_{-h_1,-2h_2,2k}B(y+2z)\Delta_{-h_2,-h_3,k}B(y+z)\\
                                 &\Delta_{-h_1,-h_2,2k}C(x+y+2z)\Delta_{-h_3,k}C(x+y+z)\\
                                 &\Delta_{-h_1,2k}D(2x+y+2z)\Delta_{h_2,-h_3,k}D(2x+y+z) \\
                                 &\Delta_{(h_2,-2h_2)}\Phi(x,y+2z)\Delta_{(-h_1,h_1),(h_3,-2h_3),(k,0)}\Phi(x+z,y) \\
    &\Delta_{(h_2,-2h_2)}\Phi(x,2k-u)\big),
  \end{align*}  
  Lemmas~\ref{cs2} and~\ref{intersection}, analogously to the case of the first moment, tell us that
  \[
    \PP\left((x,z,h_1,h_2,h_3,k)\in(\F_p^n)^6:\|I(x,z,h_1,h_2,h_3,k,\cdot)-\beta^{16}\gamma^{12}\delta^{12}\|_{U^2(\F_p^n)}>\ve^{1/8}\right)\ll\sqrt{\ve}
  \]
  and
  \begin{align*}
    \PP\big(&x\in A\cap (A-h_2)\text{ and }x+z\in \bigcap_{\omega\in\{0,1\}^8}(A-\omega\cdot(-h_1,h_3,k))\\
            &:\codim\left\{y\in\F_p^n:\Delta_{(h_2,-2h_2)}\Phi(x,y+2z)\Delta_{(-h_1,h_1),(h_3,-2h_3),(k,0)}\Phi(x+z,y)=1\right\}\neq 10d\big)\ll\f{\ve^{\Omega(1)}}{\rho^{O(1)}},
  \end{align*}
  where
\begin{align*}
    I(x,z,h_1,h_2,h_3,k,y):=\big(&\Delta_{-h_1,-2h_2,2k}B(y+2z)\Delta_{-h_2,-h_3,k}B(y+z)\\
                                 &\Delta_{-h_1,-h_2,2k}C(x+y+2z)\Delta_{-h_3,k}C(x+y+z)\\
                                 &\Delta_{-h_1,2k}D(2x+y+2z)\Delta_{h_2,-h_3,k}D(2x+y+z)\big),
  \end{align*}  
  so it follows from Lemma~\ref{subspaceavg} that the second moment equals $(\alpha^2+O(\ve))\1$ times
  \begin{equation}\label{APhiavg}
    \beta^{16}\gamma^{12}\delta^{12}\rho^{10}\E_{x,z,h_1,h_2,h_3,k}\Delta_{-h_1,h_3,k}A(x+z)\Delta_{(h_2,-2h_2)}\Phi(x,2k-u)+O\left(\f{\ve^{\Omega(1)}}{\rho^{O(1)}}\right).
  \end{equation}
  To estimate the main term of~\eqref{APhiavg}, we note that
  \[
    \|\Phi(\cdot,2\cdot-u)-\alpha\rho\|_{U^6(\F_p^n\times\F_p^n)}=\|\Phi-\alpha\rho\|_{U^6(\F_p^n\times\F_p^n)}
  \]
  and apply Lemmas~\ref{cs2} and~\ref{intersection} again to get that
  \[
    \PP\left((x,z,h_1,h_3)\in(\F_p^n)^4:\|J(x,z,h_1,h_3,\cdot)-\alpha^4\|_{U^2(\F_p^n)}>\ve^{1/8}\right)\ll\sqrt{\ve}
  \]
  and
  \[
    \PP\left((x,x+h_2)\in A\times A:\codim\left\{k\in\F_p^n:\Delta_{(h_2,-2h_2)}\Phi(x,2k-u)=1\right\}\neq 2d\right)\ll\f{\ve^{\Omega(1)}}{\rho^{O(1)}},
  \]
  where
  \[
    J(x,z,h_1,h_3,k):=\Delta_{-h_1,h_3}A(x+z+k).
  \]
  Thus, by Lemma~\ref{subspaceavg}, we have that
  \[
    \E_{x,z,h_1,h_2,h_3,k}\Delta_{-h_1,h_3,k}A(x+z)\Delta_{(h_2,-2h_2)}\Phi(x,2k-u)
  \]
  equals
  \[
    \alpha^4\rho^2\E_{x,z,h_1,h_2,h_3}\Delta_{h_2}A(x)\Delta_{-h_1,h_3}A(x+z)+O\left(\f{\ve^{\Omega(1)}}{\rho^{O(1)}}\right),
  \]
  which equals
  \[
    \alpha^{10}\rho^2+O\left(\f{\ve^{\Omega(1)}}{\rho^{O(1)}}\right)
  \]
  by Corollary~\ref{Usuniformity}. It therefore follows that the second moment is
  \[
    \alpha^{8}\beta^{16}\gamma^{12}\delta^{12}\rho^{12}+O\left(\f{\ve^{\Omega(1)}}{\rho^{O(1)}}\right).
  \]

  Thus,
  \[
    \E_{\substack{x,y,h_1,h_2,h_3 \\ x,x+h_2\in A}}\left|\mu(x,y,h_1,h_2,h_3)-\alpha^4\beta^8\gamma^6\delta^6\rho^6\right|^2\ll\f{\ve^{\Omega(1)}}{\rho^{O(1)}},
  \]
  and we conclude that
  \begin{align*}
        \E_{x,y,h_1,h_2,h_3}\Delta_{(0,h_1),(h_2,0),(0,-2h_3)}f_0(x,y)\mu(x,y,h_1,h_2,h_3) = \alpha^4\beta^8\gamma^6\delta^6\rho^6\|f_0\|_{\star_1}^8+O\left(\f{\ve^{\Omega(1)}}{\rho^{O(1)}}\right).
  \end{align*}
  Putting everything together gives
  \[
    \left|\Lambda(f_0,f_1,f_2,f_3)\right|^8\leq\alpha^{14}\beta^{20}\gamma^{16}\delta^{16}\rho^{18}\|f_0\|_{\star_1}^8+O\left(\f{\ve^{\Omega(1)}}{\rho^{O(1)}}\right),
  \]
  as desired.
  
  For~\eqref{gvn2}, we make a change of variables and apply the Cauchy--Schwarz inequality to bound $|\Lambda(T,f_1,f_2,f_3)|^2$ by
  \[
    \mu_{\F_p^n\times\F_p^n}(T)\cdot\E_{x,y}B(y)C(x+y)\Phi(x,y)|\E_zC(x+y-2z)D(2x+y-2z)f_1(x,y-z)f_3(x+z,y-2z)|^2.
  \]
  Expanding the square in the second quantity and making a change of variables yields
  \begin{align*}
    \E_{x,y,z,h_1}&B(y+2z)C(x+y+z)\Delta_{-2h_1}C(x+y-z)\Delta_{-2h_1}D(2x+y-2z)\Phi(x-z,y+2z) \\
                  &\Delta_{(0,-h_1)}f_1(x-z,y+z)\Delta_{(h_1,-2h_1)}f_3(x,y).
  \end{align*}
  By applying the Cauchy--Schwarz inequality again in the $x,y,$ and $h_1$ variables, the modulus squared of this is bounded above by
  \begin{equation}\label{2.6b}
    \E_{x,y,h_1}\Delta_{-2h_1}B(y)\Delta_{-h_1}C(x+y)D(2x+y)\Delta_{(h_1,-2h_1)}\Phi(x,y)
  \end{equation}
  times
  \begin{align*}
    \E_{x,y,h_1}\big(&\Delta_{-2h_1}B(y)D(2x+y)\Delta_{(h_1,-2h_1)}\Phi(x,y)\\
                     &|\E_{z}[B(y+2z)C(x+y+z)\Delta_{-2h_1}C(x+y-z)\Delta_{-2h_1}D(2x+y-2z)\Phi(x-z,y+2z)\\
    &\ \ \ \ \ \Delta_{(0,-h_1)}f_1(x-z,y+z)]|^2\big).
  \end{align*}
  Expanding the square and making a change of variables, the second factor equals
  \begin{align*}
    \E_{x,y,z,h_1,h_2}&\Delta_{2h_2}B(y+z)\Delta_{-2h_1}B(y-z)\Delta_{-2h_1,-h_2}C(x+y-z)\Delta_{h_2}C(x+y+z)\\
                      &\Delta_{-2h_1,-2h_2}D(2x+y-z)D(2x+y+z)\Delta_{(-h_2,h_2)}\Phi(x,y+z)\Delta_{(h_1,-2h_1)}\Phi(x+z,y-z) \\
    &\Delta_{(0,-h_1),(-h_2,h_2)}f_1(x,y),
  \end{align*}
  which we can write as
  \[
    \E_{x,y,h_1,h_2}\Delta_{(0,-h_1),(-h_2,h_2)}f_1(x,y)\mu'(x,y,h_1,h_2),
  \]
  where
  \begin{align*}
    \mu'(x,y,h_1,h_2):=\E_{z}&\Delta_{2h_2}B(y+z)\Delta_{-2h_1}B(y-z)\Delta_{-2h_1,-h_2}C(x+y-z)\Delta_{h_2}C(x+y+z)\\
                      &\Delta_{-2h_1,-2h_2}D(2x+y-z)D(2x+y+z)\Delta_{(-h_2,h_2)}\Phi(x,y+z)\Delta_{(h_1,-2h_1)}\Phi(x+z,y-z).
  \end{align*}
  The quantity~\eqref{2.6b} and the weight $\mu'$ can be analyzed in the same manner as the corresponding quantity and weight in the proof of~\eqref{gvn1}, so that
  \[
    |\Lambda(T,f_1,f_2,f_3)|^4\leq \alpha^{8}\beta^{8}\gamma^{10}\delta^{8}\rho^{8}\|f_1\|_{\star_2}^4+O\left(\f{\ve^{\Omega(1)}}{\rho^{O(1)}}\right).
  \]

  Finally, for~\eqref{gvn3}, we make a change of variables and apply the Cauchy--Schwarz inequality to bound $|\Lambda(T,T,f_2,f_3)|^2$ by $\mu_{\F_p^n\times\F_p^n}(T)$ times
  \begin{align*}
    \E_{x,y}B(y)C(x+y)\Phi(x,y)|\E_z &B(y+z)C(x+y-z)D(2x+y-2z)D(2x+y-z)\\
    &\Phi(x-z,y+z)f_2(x-z,y+2z)|^2
  \end{align*}
  Expanding the square in the second quantity and making a change of variables yields
    \begin{align*}
      \E_{x,y,z,h_1}&B(y)C(x+y-z)\Phi(x+z,y-2z)\Delta_{h_1}B(y-z)\Delta_{-h_1}C(x+y-2z)\\
                    &\Delta_{-2h_1}D(2x+y-2z)\Delta_{-h_1}D(2x+y-z)\Delta_{(-h_1,h_1)}\Phi(x,y-z)\Delta_{(-h_1,2h_1)}f_2(x,y),
    \end{align*}
    which can be written as
    \[
      \E_{x,y,h_1}\Delta_{(-h_1,2h_1)}f_2(x,y)\mu''(x,y,h_1),
    \]
    where
    \begin{align*}
      \mu''(x,y,h_1):=\E_{z}&B(y)C(x+y-z)\Phi(x+z,y-2z)\Delta_{h_1}B(y-z)\Delta_{-h_1}C(x+y-2z)\\
            &\Delta_{-2h_1}D(2x+y-2z)\Delta_{-h_1}D(2x+y-z)\Delta_{(-h_1,h_1)}\Phi(x,y-z).
    \end{align*}
    This weight can again be analyzed in the same manner as the weights appearing in the proof of~\eqref{gvn1}, giving
    \[
      |\Lambda(T,T,f_2,f_3)|^2\leq \alpha\beta^3\gamma^3\delta^4\rho^3\|f_2\|_{\star_3}^2+O\left(\f{\ve^{\Omega(1)}}{\rho^{O(1)}}\right),
    \]
    and completing the proof of the lemma.
\end{proof}

\section{Obtaining a density-increment}\label{inverse}

As was mentioned in Section~\ref{outline}, one ingredient in the proof of
Theorem~\ref{starinverse} is an inverse theorem for the $U^2(\Phi(x,\cdot))$-norm
localized to pseudorandom sets. Applying the standard, nonlocalized inverse theorem for
the $U^2(\Phi(x,\cdot))$-norm would yield a density-increment that gets weaker as $T$
becomes sparser, which is inadequate to close the density-increment iteration. To get a
strong enough localized version of this inverse theorem, we will need to use the
transference principle. The particular instance of it required is an immediate consequence
of the dense model lemma from~\cite{Zhao2014}, which appears as Lemma~3.3 in that paper.

\begin{lemma}[Dense model lemma for the $U^s$-norm on subspaces]\label{densemodel}
  For every natural number $s$, there exists a constant $c_s>0$ such that the following
  holds. Let $\ve>0$, $V\leq\F_p^n$ be a subspace, and $f,\nu:V\to[0,\infty)$ be functions
  satisfying
\begin{enumerate}
\item $0\leq f\leq\nu$,
\item $\E_{x\in V}f(x)\leq 1$, and
\item $\|\nu-1\|_{U^s(V)}\leq \exp(-\ve^{-c_s})$.
\end{enumerate}
Then there exists a $\tilde{f}:V\to[0,1]$ such that $\E_{x\in V}f(x)=\E_{x\in V}\tilde{f}(x)$ and $\|f-\tilde{f}\|_{U^s(V)}\leq\ve$.
\end{lemma}
One can see that this lemma is a consequence of Zhao's lemma by using the Gowers--Cauchy--Schwarz inequality to translate between his $(s,\ve)$-discrepancy pair condition and our $U^s$-uniformity condition.

In the course of the proof of Theorem~\ref{starinverse}, we will encounter various
averages of linear forms that turn out to be controlled by certain degree $1$ and $2$
directional uniformity norms. Because of this, we will also need to obtain a density
increment when these norms of the balanced function $g_S=S-\sigma$ are large. The first
two subsections of this section are devoted to proving that this is possible.

\subsection{Results on degree $1$ norms}\label{deg1ss}

We first show that the relevant fibers of any set of the form~\eqref{Tform} typically have
close to their average density, provided that $A,B,C,D,$ and $\Phi$ are sufficiently pseduorandom.

\begin{lemma}\label{fibersize}
   Let $d$ be a nonnegative integer, and set $\rho:=p^{-d}$. Suppose that $A,B,C,D\subset\F_p^n$ have densities $\alpha,\beta,\gamma,\delta$, respectively, and that $\Phi\subset\F_p^n\times\F_p^n$ has density $\alpha\rho$ in $\F_p^n\times\F_p^n$ and takes the form
  \[
    \Phi=\{(x,y)\in A\times\F_p^n:y\in u+V_x\},
  \]
  where each $V_x$ is a subspace of $\F_p^n$ of codimension $d$. Let $\ve>0$ and assume that
  \[
    \|A-\alpha\|_{U^{5}(\F_p^n)},\|B-\beta\|_{U^{5}(\F_p^n)},\|C-\gamma\|_{U^{5}(\F_p^n)},\|D-\delta\|_{U^{5}(\F_p^n)},\|\Phi-\alpha\rho\|_{U^{2}(\F_p^n\times\F_p^n)}<\ve.
  \]
  Define $T$ by~\eqref{Tform}. Then
  \[
    \PP\left(x\in A:|\mu_{\F_p^n}(T(x,\cdot))-\beta\gamma\delta\rho|>\ve'\right)\ll\f{\ve^{1/8}}{(\ve')^2\alpha},
  \]
  \[
    \PP\left(y\in B:|\mu_{\F_p^n}(T(\cdot,y))-\alpha\gamma\delta\rho|>\ve'\right)\ll\f{\ve^{1/8}}{(\ve')^2\beta},
  \]
    \[
    \PP\left(z\in C:|\mu_{\F_p^n}(T(\cdot,z-\cdot))-\alpha\beta\delta\rho|>\ve'\right)\ll\f{\ve^{1/8}}{(\ve')^2\gamma}
  \]
  and
  \[
    \PP\left(w\in D:|\mu_{\F_p^n}(T(\cdot,w-2\cdot))-\alpha\beta\gamma\rho|>\ve'\right)\ll\f{\ve^{1/8}}{(\ve')^2\delta}
  \]
  for any $\ve'>0$.
\end{lemma}
\begin{proof}
  From Lemma~\ref{Tdensity}, we have $\E_{x\in A}\mu_{\F_p^n}(T(x,\cdot))=\beta\gamma\delta\rho+O\left(\ve^{1/8}/\alpha\right)$. For the second moment, we argue as in Section~\ref{gvn} to estimate
  \[
    \E_{x}A(x)\left|\E_{y}B(y)C(x+y)D(2x+y)\Phi(x,y)\right|^2=\alpha\beta^2\gamma^2\delta^2\rho^2+O(\ve^{1/8})
  \]
  using Lemmas~\ref{subspaceavg} and~\ref{cs2}, giving $\E_{x\in A}\mu_{\F_p^n}(T(x,\cdot))^2=\beta^2\gamma^2\delta^2\rho^2+O\left(\ve^{1/8}/\alpha\right)$. It now follows from Markov's inequality that
  \[
    \PP\left(x\in A:|\mu_{\F_p^n}(T(x,\cdot))-\beta\gamma\delta\rho|>\ve'\right)\ll\f{\ve^{1/8}}{(\ve')^2\alpha}
  \]
  for all $\ve'>0$. The three other estimates are proved analogously.
\end{proof}

Now we can obtain a density-increment when the degree $1$ uniformity norms controlled by $\|\cdot\|_{\star_1}$ and $\|\cdot\|_{\star_2}$ are large. The proof is essentially an averaging argument, like the proof of the analogous Lemma~3.1 in~\cite{Green05note}. The most substantial new feature, which will arise many times in this section, is that we must now verify that the set on which we claim to have found a density-increment actually has close to the correct density in $\F_p^n\times\F_p^n$. In the case of corners, any product set $A\times B$ trivially has density equal to the product of the densities of $A$ and $B$. This is not, in general, the case for sets of the form~\eqref{Tform} unless further assumptions are made about $A,B,C,D,$ and $\Phi$.

\begin{lemma}\label{deg1inverse}
  There exist absolute constants $0<c_1<1<c_2$ such that the following holds. Let $d$ be a nonnegative integer, and set $\rho:=p^{-d}$. Suppose that $A,B,C,D\subset\F_p^n$ have densities $\alpha,\beta,\gamma,\delta$, respectively, and that $\Phi\subset\F_p^n\times\F_p^n$ takes the form
  \[
    \Phi=\{(x,y)\in A\times\F_p^n:y\in u+V_x\},
  \]
  where each $V_x$ is a subspace of $\F_p^n$ of codimension $d$. Let $\tau>0$ and $\ve\leq c_1(\tau\alpha\beta\gamma\delta\rho)^{c_2}$, and assume that
  \[
    \|A-\alpha\|_{U^{5}(\F_p^n)},\|B-\beta\|_{U^{5}(\F_p^n)},\|C-\gamma\|_{U^{5}(\F_p^n)},\|D-\delta\|_{U^{5}(\F_p^n)},\|\Phi_0-\rho\|_{U^{2}(\F_p^n\times\F_p^n)}<\ve.
  \]
  Define $T$ by~\eqref{Tform}, and let $S\subset T$ have density $\sigma$ in $T$. Suppose that
  \begin{equation}\label{deg11}
    \E_{x,y,h}\Delta_{(0,h)}g_S(x,y)\geq\tau\alpha\beta^2\gamma^2\delta^2\rho^2,
  \end{equation}
  \begin{equation}\label{deg12}
    \E_{x,y,h}\Delta_{(h,0)}g_S(x,y)\geq\tau\alpha^2\beta\gamma^2\delta^2\rho^2,
  \end{equation}
  or
  \begin{equation}\label{deg13}
    \E_{x,y,h}\Delta_{(-h,h)}g_S(x,y)\geq\tau\alpha^2\beta^2\gamma\delta^2\rho^2.
  \end{equation}
  Then $S$ has density at least $\sigma+\Omega(\tau^{O(1)})$ on some subset $T'$ of $T$ of the form
  \[
    T'=\left\{(x,y)\in\F_p^n\times\F_p^n:A'(x)B'(y)C'(x+y)D(2x+y)\Phi(x,y)=1\right\},
  \]
  where the densities of $A',B',C'\subset\F_p^n$ are $\Omega(\tau^{O(1)}\alpha)$, $\Omega(\tau^{O(1)}\beta)$, and $\Omega(\tau^{O(1)}\gamma)$, respectively.
\end{lemma}
\begin{proof}
  The assumption~\eqref{deg11} can be written as
  \[
    \E_{x\in A}|\E_{y}g_S(x,y)|^2\geq\tau\beta^2\gamma^2\delta^2\rho^2.
  \]
  Since $\mu_{\F_p^n}(T(x,\cdot))>2\beta\gamma\delta\rho$ for at most a
  $O(\ve^{1/8}/\alpha^3\beta^2\gamma^2\delta^2\rho^2)$-proportion of $x\in A$ by
  Lemma~\ref{fibersize}, it follows that, as long as $c_2$ is sufficiently large, there
  exists a subset $A_0\subset A$ of relative density at least $\tau/4$ for which
  \[
    \left|\E_{y}g_S(x,y)\right|\geq\f{\sqrt{\tau}\beta\gamma\delta\rho}{4}
  \]
  for all $x\in A_0$, provided that $c_1$ is small enough and $c_2$ is large enough. Note that $\E_{y}g_S(x,y)$ is a real number, and thus is either positive or negative. There must therefore exist a subset $A_1\subset A_0$ of density at least $1/2$ in $A_0$ such that either
  \[
    \E_{y}g_S(x,y)\geq\f{\sqrt{\tau}\beta\gamma\delta\rho}{4}
  \]
  for every $x\in A_1$ or
  \[
    \E_{y}g_S(x,y)\leq-\f{\sqrt{\tau}\beta\gamma\delta\rho}{4}
  \]
  for every $x\in A_1$.

  In the first case, setting $A'=A_1$ and $\alpha'=\mu_{\F_p^n}(A')$, we have
  \[
    \left|\E_{x\in A'}\E_{y}g_S(x,y)-\left(\E_{x\in A'}\E_{y}S(x,y)-\sigma\beta\gamma\delta\rho\right)\right|<\f{\sqrt{\tau}\beta\gamma\delta\rho}{8}
  \]
  by Lemma~\ref{fibersize} whenever $c_1$ is small enough and $c_2$ is large enough, so that
  \[
    \E_{x,y}A'(x)S(x,y)>\left(\sigma+\f{\sqrt{\tau}}{8}\right)\alpha'\beta\gamma\delta\rho.
  \]
  Since $\E_{x,y}A'(x)B(y)C(x+y)D(2x+y)\Phi(x,y)=\alpha'\beta\gamma\delta\rho+O(\ve^{1/16}/\alpha)$ by Lemma~\ref{fibersize}, the conclusion of the lemma follows when $c_1$ is small enough and $c_2$ is large enough by taking $B'=B$.
  
  In the second case, setting $A'=A\setminus A_1$ and $\alpha'=\mu_{\F_p^n}(A')$, we use the fact that $\E_{x,y}g_S(x,y)=0$ to deduce that
  \[
    \E_{x\in A'}\E_{y}g_S(x,y)\geq\f{\tau^{3/2}\beta\gamma\delta\rho}{16}
  \]
  whenever $c_1$ is small enough and $c_2$ is large enough. Note that $\E_{x,y}A'(x)B(y)C(x+y)D(2x+y)\Phi(x,y)=\alpha'\beta\gamma\delta\rho+O(\ve^{1/16}/\alpha)$ in this case as well by Lemma~\ref{fibersize}, so the conclusion of the lemma will follow as long as $\alpha'=\Omega(\tau^{O(1)}\alpha)$. But this also follows from Lemma~\ref{fibersize}, for we have
  \[
    \tau\alpha(\beta\gamma\delta\rho)+O(\ve^{1/16})\ll\left|\E_{x}A_1(x)g_S(x,y)\right|=\left|\E_{x}A'(x)g_S(x,y)\right|=\alpha'(\beta\gamma\delta\rho)+O(\alpha'\ve^{1/16}/\alpha).
  \]

  The proof of the lemma starting from the assumptions~\eqref{deg12} and~\eqref{deg13} is essentially identical, but using the second and third probability estimates from Lemma~\ref{fibersize}, respectively.
\end{proof}

\subsection{Results on degree $2$ norms}\label{deg2ss}

To obtain a density-increment when the relevant degree $2$ directional uniformity norms of $g_S$ are large, we first need to show that certain degree $2$ ``inner products'' are controlled by the degree $1$ directional uniformity norms studied in the previous subsection.

\begin{lemma}\label{deg2to1}
Let $d$ be a nonnegative integer, and set $\rho:=p^{-d}$. Suppose that $A,B,C,D\subset\F_p^n$ have densities $\alpha,\beta,\gamma,\delta$, respectively, and that $\Phi\subset\F_p^n\times\F_p^n$ takes the form
  \[
    \Phi=\left\{(x,y)\in A\times\F_p^n:y\in u+V_x\right\},
  \]
  where each $V_x$ is a subspace of $\F_p^n$ of codimension $d$. Let $\ve>0$ and assume that
  \[
    \|A-\alpha\|_{U^{8}(\F_p^n)},\|B-\beta\|_{U^{8}(\F_p^n)},\|C-\gamma\|_{U^{8}(\F_p^n)},\|D-\delta\|_{U^{8}(\F_p^n)},\|\Phi-\alpha\rho\|_{U^{4}(\F_p^n\times\F_p^n)}<\ve.
  \]
  Define $T$ by~\eqref{Tform}, and let $S\subset T$ and $\tau>0$.

  If
  \begin{equation}\label{deg2to11}
    \left|\E_{x,y,h,k}\prod_{\omega\in\{0,1\}^2}g_\omega((x,y)+\omega\cdot((0,h),(0,k)))\right|\geq\tau\alpha\beta^4\gamma^4\delta^4\rho^3,
  \end{equation}
  where $g_\omega$ equals $T$ or $g_S$ for all $\omega\in\{0,1\}^2$, at least one
  $g_\omega$ equals $g_S$, and at least one $g_\omega$ equals $T$, then
  \[
    \E_{x,y,h}\Delta_{(0,h)}g_S(x,y)\geq\tau^2\alpha\beta^2\gamma^2\delta^2\rho^2+O\left(\f{\ve^{\Omega(1)}}{(\alpha\beta\gamma\delta\rho)^{O(1)}}\right).
  \]
 If
  \begin{equation}\label{deg2to12}
    \left|\E_{x,y,h,k}\prod_{\omega\in\{0,1\}^2}g_\omega((x,y)+\omega\cdot((h,0),(0,k)))\right|\geq\tau\alpha^2\beta^2\gamma^4\delta^4\rho^4,
  \end{equation}
  where $g_\omega$ equals $T$ or $g_S$ for all $\omega\in\{0,1\}^2$, at least one
  $g_\omega$ equals $g_S$, and at least one $g_\omega$ equals $T$, then
  \[
    \E_{x,y,h}\Delta_{(0,h)}g_S(x,y)\geq\tau^2\alpha\beta^2\gamma^2\delta^2\rho^2+O\left(\f{\ve^{\Omega(1)}}{(\alpha\beta\gamma\delta\rho)^{O(1)}}\right)
  \]
  or
  \[
    \E_{x,y,h}\Delta_{(h,0)}g_S(x,y)\geq\tau^2\alpha^2\beta\gamma^2\delta^2\rho^2+O\left(\f{\ve^{\Omega(1)}}{(\alpha\beta\gamma\delta\rho)^{O(1)}}\right).
  \]
  If
  \begin{equation}\label{deg2to13}
    \left|\E_{x,y,h,k}\prod_{\omega\in\{0,1\}^2}g_\omega((x,y)+\omega\cdot((0,h),(-k,k)))\right|\geq\tau\alpha^2\beta^4\gamma^2\delta^4\rho^4,
  \end{equation}
  where $g_\omega$ equals $T$ or $g_S$ for all $\omega\in\{0,1\}^2$, at least one
  $g_\omega$ equals $g_S$, and at least one $g_\omega$ equals $T$, then
  \[
    \E_{x,y,h}\Delta_{(0,h)}g_S(x,y)\geq\tau^2\alpha\beta^2\gamma^2\delta^2\rho^2+O\left(\f{\ve^{\Omega(1)}}{(\alpha\beta\gamma\delta\rho)^{O(1)}}\right),
  \]
  or
  \[
    \E_{x,y,h}\Delta_{(-h,h)}g_S(x,y)\geq\tau^2\alpha^2\beta^2\gamma\delta^2\rho^2+O\left(\f{\ve^{\Omega(1)}}{(\alpha\beta\gamma\delta\rho)^{O(1)}}\right).
  \]
\end{lemma}
\begin{proof}
  We rewrite the various assumptions that~\eqref{deg2to11},~\eqref{deg2to12}, and~\eqref{deg2to13} hold when at least two of the $g_\omega$'s equal $T$ as
  \[
    \left|\E_{x,y,h}g_S(x,y)g_S(x,y+h)\mu_1(x,y,h)\right|\geq\tau\alpha\beta^4\gamma^4\delta^4\rho^3,
  \]
  \[
    \left|\E_{x,y}g_S(x,y)\mu_2(x,y)\right|\geq\tau\alpha\beta^4\gamma^4\delta^4\rho^3,
  \]
  \[
    \left|\E_{x,y,k}g_S(x,y)g_S(x,y+k)\mu_3(x,y,k)\right|\geq\tau\alpha^2\beta^2\gamma^4\delta^4\rho^4,
  \]
  \[
    \left|\E_{x,y,h}g_S(x,y)g_S(x+h,y)\mu_4(x,y,h)\right|\geq\tau\alpha^2\beta^2\gamma^4\delta^4\rho^4,
  \]
  \[
    \left|\E_{x,y}g_S(x,y)\mu_5(x,y)\right|\geq\tau\alpha^2\beta^2\gamma^4\delta^4\rho^4,
  \]
  \[
    \left|\E_{x,y,h}g_S(x,y)g_S(x,y+h)\mu_6(x,y,h)\right|\geq\tau\alpha^2\beta^4\gamma^2\delta^4\rho^4,
  \]
  \[
    \left|\E_{x,y,k}g_S(x,y)g_S(x-k,y+k)\mu_7(x,y,k)\right|\geq\tau\alpha^2\beta^4\gamma^2\delta^4\rho^4,
  \]
  and
  \[
    \left|\E_{x,y}g_S(x,y)\mu_8(x,y)\right|\geq\tau\alpha^2\beta^4\gamma^2\delta^4\rho^4,
  \]
  where
  \[
    \mu_1(x,y,h)=\E_{k}T(x,y+k)T(x,y+h+k),
  \]
  \[
    \mu_2(x,y)=\E_{h,k}T(x,y+h)T(x,y+k)T(x,y+h+k),
  \]
  \[
    \mu_3(x,y,k)=\E_{h}T(x+h,y)T(x+h,y+k),
  \]
  \[
    \mu_4(x,y,h)=\E_{k}T(x,y+k)T(x+h,y+k),
  \]
  \[
    \mu_5(x,y)=\E_{h,k}T(x+h,y)T(x,y+k)T(x+h,y+k),
  \]
    \[
    \mu_6(x,y,h)=\E_{k}T(x-k,y+k)T(x-k,y+h+k),
  \]
  \[
    \mu_7(x,y,k)=\E_{h}T(x,y+h)T(x-k,y+h+k),
  \]
  and
  \[
    \mu_8(x,y)=\E_{h,k}T(x,y+h)T(x-k,y+k)T(x-k,y+h+k).
  \]
  Using the definition~\eqref{Tform} of $T$ and arguing as in Section~\ref{gvn} using Lemmas~\ref{subspaceavg},~\ref{cs2}, and~\ref{intersection} gives that each of $\mu_1,\dots,\mu_8$ is typically very close to its average value on its support. Precisely, we have the estimates
  \[
    \E_{x,y,h}A(x)\Phi(x,h-u)|\mu_1(x,y,h)-\beta^2\gamma^2\delta^2\rho|^2\ll\f{\ve^{\Omega(1)}}{\rho^{O(1)}},
  \]
  \[
    \E_{x,y}A(x)\Phi(x,y)|\mu_2(x,y)-\beta^3\gamma^3\delta^3\rho^2|^2\ll\f{\ve^{\Omega(1)}}{\rho^{O(1)}},
  \]
  \[
    \E_{x,y,k}B(y)B(y+k)|\mu_3(x,y,k)-\alpha\gamma^2\delta^2\rho^2|^2\ll\f{\ve^{\Omega(1)}}{\rho^{O(1)}},
  \]
  \[
    \E_{x,y,h}A(x)A(x+h)|\mu_4(x,y,h)-\beta\gamma^2\delta^2\rho^2|^2\ll\f{\ve^{\Omega(1)}}{\rho^{O(1)}},
  \]
  \[
    \E_{x,y}A(x)B(y)|\mu_5(x,y)-\beta^2\gamma^3\delta^3\rho^3|^2\ll\f{\ve^{\Omega(1)}}{\rho^{O(1)}},
  \]
  \[
    \E_{x,y,h}C(x+y)C(x+y+h)|\mu_6(x,y,h)-\alpha\beta^2\delta^2\rho^2|^2\ll\f{\ve^{\Omega(1)}}{\rho^{O(1)}},
  \]
  \[
    \E_{x,y,k}A(x)A(x-k)|\mu_7(x,y,k)-\beta^2\gamma\delta^2\rho^2|^2\ll\f{\ve^{\Omega(1)}}{\rho^{O(1)}},
  \]
  and
  \[
    \E_{x,y}A(x)C(x+y)|\mu_8(x,y)-\alpha\beta^3\gamma\delta^3\rho^3|^2\ll\f{\ve^{\Omega(1)}}{\rho^{O(1)}},
  \]
  which imply that
  \[
    \left|\E_{x,y,h}g_S(x,y)g_S(x,y+h)\mu_1(x,y,h)\right|=\beta^2\gamma^2\delta^2\rho\E_{x,y,h}g_S(x,y)g_S(x,y+h)+O\left(\f{\ve^{\Omega(1)}}{\rho^{O(1)}}\right),
  \]
  \[
    \left|\E_{x,y}g_S(x,y)\mu_2(x,y)\right|=\beta^3\gamma^3\delta^3\rho^2|\E_{x,y}g_S(x,y)|+O\left(\f{\ve^{\Omega(1)}}{\rho^{O(1)}}\right),
  \]
  \[
    \left|\E_{x,y,k}g_S(x,y)g_S(x,y+k)\mu_3(x,y,k)\right|=\alpha\gamma^2\delta^2\rho^2\E_{x,y,k}g_S(x,y)g_S(x,y+k)+O\left(\f{\ve^{\Omega(1)}}{\rho^{O(1)}}\right),
  \]
  \[
    \left|\E_{x,y,h}g_S(x,y)g_S(x+h,y)\mu_4(x,y,h)\right|=\beta\gamma^2\delta^2\rho^2\E_{x,y,h}g_S(x,y)g_S(x+h,y)+O\left(\f{\ve^{\Omega(1)}}{\rho^{O(1)}}\right),
  \]
  \[
    \left|\E_{x,y}g_S(x,y)\mu_5(x,y)\right|=\beta^2\gamma^3\delta^3\rho^3|\E_{x,y}g_S(x,y)|+O\left(\f{\ve^{\Omega(1)}}{\rho^{O(1)}}\right),
  \]
    \[
    \left|\E_{x,y,h}g_S(x,y)g_S(x,y+h)\mu_6(x,y,h)\right|=\alpha\beta^2\delta^2\rho^2\E_{x,y,h}g_S(x,y)g_S(x,y+h)+O\left(\f{\ve^{\Omega(1)}}{\rho^{O(1)}}\right),
  \]
  \[
    \left|\E_{x,y,k}g_S(x,y)g_S(x-k,y+k)\mu_7(x,y,k)\right|=\beta^2\gamma\delta^2\rho^2\E_{x,y,k}g_S(x,y)g_S(x-k,y+k)+O\left(\f{\ve^{\Omega(1)}}{\rho^{O(1)}}\right),
  \]
  and
  \[
    \left|\E_{x,y}g_S(x,y)\mu_8(x,y)\right|=\alpha\beta^3\gamma\delta^3\rho^3\left|\E_{x,y}g_S(x,y)\right|+O\left(\f{\ve^{\Omega(1)}}{\rho^{O(1)}}\right).
  \]
  Since $|\E_{x,y}g_S(x,y)|^2\leq \alpha\E_{x,y,h}g_S(x,y)g_S(x,y+h)$ by an application of the Cauchy--Schwarz inequality, the conclusion of the lemma easily follows starting from any one of the assumptions~\eqref{deg2to11}, \eqref{deg2to12}, or~\eqref{deg2to13} when at least two of the $g_\omega$'s equal $T$.
  
  To prove the lemma when only one $g_\omega$ in~\eqref{deg2to11},~\eqref{deg2to12}, or~\eqref{deg2to13} equals $T$, we will apply the Cauchy--Schwarz inequality once, and then argue analogously. By making a change of variables, we may start from the assumption that
  \[
    \left|\E_{x,y,h,k}g_S(x,y)g_S(x,y+h)g_S(x,y+k)T(x,y+h+k)\right|\geq\tau\alpha\beta^4\gamma^4\delta^4\rho^3,
  \]
  \[
    \left|\E_{x,y,h,k}g_S(x,y)g_S(x+h,y)g_S(x,y+k)T(x+h,y+k)\right|\geq\tau\alpha^2\beta^2\gamma^4\delta^4\rho^4,
  \]
  or
  \[
    \left|\E_{x,y,h,k}g_S(x,y)g_S(x,y+h)g_S(x-k,y+k)T(x-k,y+h+k)\right|\geq\tau\alpha^2\beta^4\gamma^2\delta^4\rho^4.
  \]
  We apply the Cauchy--Schwarz inequality in each of these three cases to get that
  \begin{equation}\label{deg2to11a}
    \E_{x,y,h}\left|\E_{k}g_S(x,y+k)T(x,y+h+k)\right|^2\Delta_hB(y)\Delta_hC(x+y)\Delta_hD(2x+y)\Phi(x,y)
  \end{equation}
  times $\E_{x,y,h}\Delta_{(0,h)}T(x,y)$ is at least $\tau^2\alpha^2\beta^8\gamma^8\delta^8\rho^6$,
  \begin{equation}\label{deg2to13b}
    \E_{x,y,h}\left|\E_kg_S(x,y+k)T(x+h,y+k)\right|^2B(y)\Delta_hC(x+y)\Delta_{2h}D(2x+y)\Delta_{(h,0)}\Phi(x,y)
  \end{equation}
  times $\E_{x,y,h}\Delta_{(h,0)}T(x,y)$ is at least $\tau^2\alpha^4\beta^4\gamma^8\delta^8\rho^8$, or 
  \begin{equation}\label{deg2to13c}
    \E_{x,y,h}\left|\E_kg_S(x-k,y+k)T(x-k,y+h+k)\right|^2A(x)\Delta_hB(y)\Delta_{h}D(2x+y)\Delta_{(0,h)}\Phi(x,y)
  \end{equation}
  times $\E_{x,y,h}\Delta_{(0,h)}T(x,y)$ is at least $\tau^2\alpha^4\beta^8\gamma^4\delta^8\rho^8$.
  
  We have
  \[
    \E_{x,y,h}\Delta_{(0,h)}T(x,y)=\alpha\beta^2\gamma^2\delta^2\rho^2+O\left(\f{\ve^{\Omega(1)}}{\rho^{O(1)}}\right)
  \]
  and
  \[
    \E_{x,y,h}\Delta_{(h,0)}T(x,y)=\alpha^2\beta\gamma^2\delta^2\rho^2+O\left(\f{\ve^{\Omega(1)}}{\rho^{O(1)}}\right)
  \]
  by Lemmas~\ref{subspaceavg},~\ref{cs2}, and~\ref{intersection}. Expanding the square and making a change of variables,~\eqref{deg2to11a} equals
  \[
    \E_{x,y,l}\Delta_{(0,\ell)}g_S(x,y)\mu_9(x,y,\ell),
  \]
  where
  \begin{align*}
    \mu_9(x,y,\ell):=\E_{h,k}&\Delta_hB(y-k)\Delta_\ell B(y+h)\Delta_hC(x+y-k)\Delta_\ell
                               C(x+y+h)\\
    &\Delta_hD(2x+y-k)\Delta_\ell D(2x+y+h)\Phi(x,y-k)\Phi(x,y+h)
  \end{align*}
\eqref{deg2to13b} equals
  \[
    \E_{x,y,\ell}\Delta_{(0,\ell)}g_S(x,y)\mu_{10}(x,y,\ell),
  \]
  where
  \begin{align*}
    \mu_{10}(x,y,\ell):=\E_{h,k}&B(y-k)\Delta_hC(x+y-k)\Delta_{\ell}C(x+y+h)\\
    &\Delta_{2h}D(2x+y-k)\Delta_{\ell}D(2x+y+2h)\Delta_{(h,0)}\Phi(x,y-k)\Delta_{(0,\ell)}\Phi(x+h,y)
  \end{align*}
  and~\eqref{deg2to13c} equals
  \[
    \E_{x,y,\ell}\Delta_{(-\ell,\ell)}g_S(x,y)\mu_{11}(x,y,\ell),
  \]
  where
  \begin{align*}
    \mu_{11}(x,y,\ell):=\E_{h,k}&\Delta_hB(y-k)\Delta_{\ell}B(y+h)C(x+y+h) \\
                                &\Delta_{h}D(2x+y+k)\Delta_{-\ell}D(2x+y+h)\Delta_{(0,h)}\Phi(x+k,y-k)\Delta_{(-\ell,\ell)}\Phi(x,y+h)
  \end{align*}
  
  Analogously to the weights $\mu_1,\dots,\mu_8$, Lemmas~\ref{subspaceavg},~\ref{cs2}, and~\ref{intersection} give the estimates
  \[
    \E_{x,y,\ell}\Delta_{(0,\ell)}\Phi(x,y)\left|\mu_9(x,y,\ell)-\beta^4\gamma^4\delta^4\rho^2\right|^2\ll\f{\ve^{\Omega(1)}}{\rho^{O(1)}},
  \]
  \[
    \E_{x,y,\ell}A(x)\Delta_{\ell}B(y)\left|\mu_{10}(x,y,\ell)-\alpha\beta\gamma^4\delta^4\rho^4\right|^2\ll\f{\ve^{\Omega(1)}}{\rho^{O(1)}},
  \]
  and
  \[
    \E_{x,y,\ell}\Delta_{-\ell}A(x)\left|\mu_{11}(x,y,\ell)-\alpha\beta^4\gamma\delta^4\rho^4\right|^2\ll\f{\ve^{\Omega(1)}}{\rho^{O(1)}},   
  \]
  from which it follows that
  \[
    \left|\E_{x,y,\ell}\Delta_{(0,\ell)}g_S(x,y)\mu_9(x,y,\ell)\right|=\beta^4\gamma^4\delta^4\rho^2\E_{x,y,\ell}\Delta_{(0,\ell)}g_S(x,y)+O\left(\f{\ve^{\Omega(1)}}{\rho^{O(1)}}\right),
  \]
  \[
    \left|\E_{x,y,\ell}\Delta_{(0,\ell)}g_S(x,y)\mu_{10}(x,y,\ell)\right|=\alpha\beta\gamma^4\delta^4\rho^4\E_{x,y,\ell}\Delta_{(0,\ell)}g_S(x,y)+O\left(\f{\ve^{\Omega(1)}}{\rho^{O(1)}}\right),
  \]
  and
  \[
    \left|\E_{x,y,\ell}\Delta_{(-\ell,\ell)}g_S(x,y)\mu_{11}(x,y,\ell)\right|=\alpha\beta^4\gamma\delta^4\rho^4\E_{x,y,\ell}\Delta_{(-\ell,\ell)}g_S(x,y)+O\left(\f{\ve^{\Omega(1)}}{\rho^{O(1)}}\right).
  \]
  This completes the proof of the lemma.
\end{proof}

Now we are almost ready to prove our desired density-increment result for the localized degree $2$ directional uniformity norms controlled by $\|\cdot\|_{\star_1}$:

\begin{lemma}\label{deg2inverse}
 There exist absolute constants $0<c_1<1<c_2,c_3$ such that the following holds. Let $d$ be a nonnegative integer, and set $\rho:=p^{-d}$. Suppose that $A,B,C,D\subset\F_p^n$ have densities $\alpha,\beta,\gamma,\delta$, respectively, and that $\Phi\subset\F_p^n\times\F_p^n$ takes the form
  \[
    \Phi=\{(x,y)\in A\times\F_p^n:y\in u+V_x\},
  \]
  where each $V_x$ is a subspace of $\F_p^n$ of codimension $d$. Let $\tau>0$ and $\ve\leq c_1(\tau\alpha\beta\gamma\delta\rho)^{c_2}\exp(-(32/\tau)^{c_3})$, and assume that
  \[
    \|A-\alpha\|_{U^{8}(\F_p^n)},\|B-\beta\|_{U^{8}(\F_p^n)},\|C-\gamma\|_{U^{8}(\F_p^n)},\|D-\delta\|_{U^{8}(\F_p^n)},\|\Phi-\alpha\rho\|_{U^{4}(\F_p^n\times\F_p^n)}<\ve.
  \]
  Define $T$ by~\eqref{Tform}, and let $S\subset T$ have density $\sigma$ in $T$. Suppose that
  \begin{equation}\label{deg21}
    \E_{x,y,h,k}\Delta_{(0,h),(0,k)}g_S(x,y)\geq\tau\alpha\beta^4\gamma^4\delta^4\rho^3
  \end{equation}
  or
  \begin{equation}\label{deg22}
    \E_{x,y,h,k}\Delta_{(h,0),(0,k)}g_S(x,y)\geq\tau\alpha^2\beta^2\gamma^4\delta^4\rho^4.
  \end{equation}
  Then $S$ has density at least $\sigma+\Omega(\tau^{O(1)})$ on some subset $T'$ of $T$ of the form
  \[
    T'=\left\{(x,y)\in\F_p^n\times\F_p^n:A'(x)B'(y)C(x+y)D(2x+y)\Phi'(x,y)=1\right\},
  \]
  where the densities of $A',B'\subset\F_p^n$ are both $\Omega((\sigma\tau\alpha\beta\gamma\delta\rho)^{O(1)})$, and $\Phi'$ is of the form
  \[
    \Phi'=\left\{(x,y)\in A'\times\F_p^n:y\in u'+V'_x\right\},
  \]
  where each $V_x'$ is a subspace of $\F_p^n$ of codimension $d+1$.
\end{lemma}

To prove this lemma starting from the assumption~\eqref{deg21}, we will apply a localized
$U^2(\Phi(x,\cdot))$-norm inverse theorem for many fixed $x$, which will produce a
density-increment on a set of the form
\begin{equation}\label{Psiform}
  \left\{(x,y)\in\F_p^n\times\F_p^n:A'(x)B(y)C(x+y)D(2x+y)\Psi(x,y)=1\right\},
\end{equation}
where
\[
  \Psi=\{(x,y)\in A'\times\F_p^n:y\in u_x'+V_x'\}.
\]
This is not yet what the conclusion of the lemma promises, since $u'_x$ varies with
$x$. To show that we can select a fixed $u'$ (at the cost of shrinking the size of $A'$ a
bit) we will need Lemma~\ref{uxtou} below.

The reader may wonder why we cannot just run the density-increment argument on sets of the
form~\eqref{Psiform} and skip having to prove Lemma~\ref{uxtou}. The issue with this
hypothetical proof is that $U^s(\F_p^n\times\F_p^n)$-uniformity of a set of the form
$\Psi$ is not strong enough to guarantee that the analogue of Lemma~\ref{intersection} is
true, regardless of how large $s$ is taken to be. Thus, such sets are not as amenable to a
Shkredov-like pseudorandomization procedure.

\begin{lemma}\label{uxtou}
  There exist absolute constants $0<c_1<1<c_2$ such that the following holds. Let $d$ be a
  nonnegative integer, and set $\rho:=p^{-d}$.  Let $A\subset\F_p^n$ have density
  $\alpha$, $\Psi\subset\F_p^n\times\F_p^n$ take the form
  \[
    \Psi=\{(x,y)\in A\times\F_p^n:y\in u_{x}+V_{x}\},
  \]
  where $u_{x}\in\F_p^n$ and each $V_{x}$ is a subspace of $\F_p^n$ of codimension $d$,
  and $K\subset A\times\F_p^n$ satisfy
  \begin{equation*}
    \left|\E_{y}K(x,y)-\kappa\right|,\left|\E_{y}K(x,y)H(y)-\kappa\rho\right|<\ve
  \end{equation*}
  for every $x\in A$ and every affine subspace $H$ of $\F_p^n$ of codimension $d$. Let
  $\tau>0$, assume that $\ve\leq c_1(\tau\alpha\beta\gamma\delta\rho)^{c_2}$, and suppose
  that $S\subset \F_p^n\times\F_p^n$ has density at least $\sigma+\tau$ in $K\cap\Psi$,
  where $\sigma>0$. Then there exists a $u\in\F_p^n$ and a subset $A'\subset A$ such that
  $S\cap T'$ has density at least $\sigma+\tau/4$ in $T'$, where
  \[
    T':=A'\times\F_p^n\cap K\cap\Phi',
  \]
  \[
    \Phi':=\{(x,y)\in A'\times\F_p^n:y\in u+V_{x}\},
  \]
  and $\mu_{\F_p^n\times\F_p^n}(A')\geq\alpha\rho\tau/2$.
\end{lemma}
\begin{proof}
  Define
  \[
    \Psi_{u}:=\{(x,y)\in A\times\F_p^n:y\in u+V_{x}\}
  \]
  and
  \[
    A_{u}:=\{x\in A: u_{x}-u\in V_{x}\}
  \]
  for every $u\in\F_p^n$, and set $\alpha_{u}:=\mu_{\F_p^n}(A_{u})$, so that
  $\E_u\alpha_u=\alpha\rho$. By our assumption on $K$, we have
  \[
    \mu_{\F_p^n\times\F_p^n}(K)=\alpha\kappa+O\left(\ve\right)
  \]
  and
  \[
    \mu_{\F_p^n\times\F_p^n}(K\cap\Psi_{u})=\alpha_u\kappa\rho+O\left(\ve\right)
  \]
  for all $u\in\F_p^n$.

  Note that
  \[
    \E_{u}\mu_{\F_p^n\times\F_p^n}(S\cap K\cap\Psi_{u})\geq\rho(\sigma+\tau)\mu_{\F_p^n\times\F_p^n}(K)=(\sigma+\tau)\alpha\kappa\rho^2+O\left(\ve\right),
  \]
  and set
  \[
    G(u):=\f{\mu_{\F_p^n\times\F_p^n}(S\cap K\cap\Psi_{u})}{\mu_{\F_p^n\times\F_p^n}(K\cap\Psi_{u})}
  \]
  for each $u\in \F_p^n$. Then we have
  \begin{align*}
    (\sigma+\tau)\alpha\kappa\rho^2+O\left(\ve\right)&\leq\kappa\rho\E_{u}\alpha_{u}G(u) \\                                                                    &<\f{\kappa\rho}{p^n}\left(\left(\sigma+\f{\tau}{4}\right)\sum_{\substack{ u\in\F_p^n \\ G(u)<\sigma+\tau/4}}\alpha_{u}+\sum_{\substack{ u\in\F_p^n \\ G(u)\geq\sigma+\tau/4}}\alpha_u\right) \\                                                                      &<\kappa\rho\left(\left(\sigma+\f{\tau}{4}\right)\left(\alpha\rho-\eta\right)+\eta\right),
  \end{align*}
  where
  \[
    \eta:=\f{1}{p^n}\sum_{\substack{ u\in\F_p^n \\ G(u)\geq\sigma+\tau/4}}\alpha_u,
  \]
  so that
  \[
    \eta>\f{5\tau\alpha\rho}{8}
  \]
  when $c_1$ is small enough and $c_2$ is large enough. The contribution to $\eta$ coming from $u$ for which $\alpha_u<\tau\alpha\rho/2$ is obviously at most $\tau\alpha\rho/2$, which implies that
  \[
    \f{1}{p^n}\sum_{\substack{ u\in\F_p^n \\ G(u)\geq\sigma+\tau/4 \\ \alpha_u\geq\tau\alpha\rho/2}}\alpha_u>\f{\tau\alpha\rho}{8},
  \]
  which is clearly positive. We thus conclude that there must exist a $u\in\F_p^n$ for which $\alpha_u=\mu_{\F_p^n}(A_{u})\geq\tau\alpha\rho/2$ and
  \[
    \mu_{\F_p^n\times\F_p^n}(S\cap K\cap\Psi_{u})\geq\left(\sigma+\f{\tau}{4}\right)\mu_{\F_p^n\times\F_p^n}(K\cap\Psi_{u}).
  \]
  The conclusion of the lemma now follows by taking $A'=A_{u}$ and $\Phi'=\Psi_{u}\cap (A'\times\F_p^n)$.
\end{proof}

Now we can prove Lemma~\ref{deg2inverse}.
\begin{proof}[Proof of Lemma~\ref{deg2inverse}]
  First assume that~\eqref{deg21} holds. By writing $S=g_S+\sigma T$, we see that $\E_{x,y,h,k}\Delta_{(0,h),(0,k)}S(x,y)$ equals
  \[
    \sigma^4\E_{x,y,h,k}\Delta_{(0,h),(0,k)}T(x,y)+\E_{x,y,h,k}\Delta_{(0,h),(0,k)}g_S(x,y)
  \]
  plus $14$ terms of the form
  \begin{equation}\label{14terms1}
    \E_{x,y,h,k}\prod_{\omega\in\{0,1\}^2}g_\omega((x,y)+\omega\cdot((0,h),(0,k))),
  \end{equation}
  where at least one $g_\omega$ equals $g_S$ and at least one other equals $\sigma T$, and
  \[
    \sigma^4\E_{x,y,h,k}\Delta_{(0,h),(0,k)}T(x,y)+\E_{x,y,h,k}\Delta_{(0,h),(0,k)}g_S(x,y)\geq (\sigma^4+\tau)\alpha\beta^4\gamma^4\delta^4\rho^3+O\left(\f{\ve^{\Omega(1)}}{\rho^{O(1)}}\right).
  \]
  If one of the terms~\eqref{14terms1} has absolute value larger than $\tau\alpha\beta^4\gamma^4\delta^4\rho^3/32$, then combining Lemmas~\ref{deg2to1} and~\ref{deg1inverse} produces the desired density increment. Thus, we may proceed under the assumption that all have size at most $\tau\alpha\beta^4\gamma^4\delta^4\rho^3/32$, so that
  \[
    \E_{x,y,h,k}\Delta_{(0,h),(0,k)}S(x,y)\geq\left(\sigma^4+\f{\tau}{2}\right)\alpha\beta^4\gamma^4\delta^4\rho^3.
  \]

  For each $x\in\F_p^n$, set $\Phi_x(y):=\Phi(x,y)$, $T_x(y):=T(x,y)$, and
  $S_x(y)=S(x,y)$. Lemmas~\ref{subspaceavg},~\ref{cs2}, and~\ref{intersection} tell us
  that
  \[
    \E_xA(x)\left|\E_{y,h,k}\Delta_{h,k}T_x(y)-\beta^4\gamma^4\delta^4\rho^3\right|\ll\ve^{\Omega(1)},
  \]
  \[
    \E_{x}A(x)\|T_x-\beta\gamma\delta\|_{U^2(\Phi_x)}^4\ll\ve^{\Omega(1)},
  \]
  and that
  \[
    \E_xA(x)\left|\E_{y\in\Phi_x}S_x(y)-\sigma\beta\gamma\delta\right|^2<\frac{\tau^4}{64}\alpha\beta^2\gamma^2\delta^2\rho^2,
  \]
  or else Lemma~\ref{deg1inverse} will again give the desired density increment. It
  follows that there exists a subset $A_0\subset A$ of density $\gg\tau$ in $A$ such that
  \[
    \|S_x\|_{U^2(\Phi_x)}^4\geq\left(\sigma^4+\f{\tau}{4}\right)\beta^4\gamma^4\delta^4, \|T_x-\beta\gamma\delta\|_{U^2(\Phi_x)}\ll\ve^{\Omega(1)},\text{ and } \left|\E_{y\in\Phi_x}S_x(y)-\sigma\beta\gamma\delta\right|<\f{\tau\beta\gamma\delta}{64}
  \]
  for every $x\in A_0$. Setting
  \[
    f_x(y):=\f{1}{\beta\gamma\delta}S_x(y)\qquad\text{ and }\qquad\nu_x:=\f{1}{\beta\gamma\delta}T_x,
  \]
  for all $x\in A_0$ we then have $\|f_x\|^4_{U^2(\Phi_x)}\geq\sigma^4+\tau/4$, $0\leq f_x\leq\nu_x$, $\E_{y}f_x(y)\leq 1$, and
  \[
    \|\nu_x-1\|_{U^2(\Phi_x)}\ll \f{\ve^{\Omega(1)}}{(\alpha\beta\gamma\delta\rho)^{O(1)}}\leq\exp\left(-(32/\tau)^{c_3}\right),
  \]
  provided that $c_1$ is small enough and $c_2$ is large enough. Thus, as long as $c_3$ is sufficiently large, Lemma~\ref{densemodel} tells us that there exists a function $\tilde{f}_x:\Phi_x\to[0,1]$ such that $\E_{y\in\Phi_x}f_x(y)=\E_{y\in\Phi_x}\tilde{f}_x(y)$ and $\|f_x-\tilde{f}_x\|_{U^2(\Phi_x)}^4\leq\tau/32$. As a consequence, since $\|f_x\|^4_{U^2(\Phi_x)}\geq\sigma^4+\f{\tau}{4}$, we must have
  \[
    \|\tilde{f}_x\|_{U^2(\Phi_x)}^4\geq\sigma^4+\f{\tau}{8}
  \]
  as well. Set $\tilde{\sigma}_x:=\E_{y}f_x(y)=\E_{y}\tilde{f}_x(y)$ and let $v_x\in\Phi_x$, so that
  \[
    \|\tilde{f}_x\|_{U^2(\Phi_x)}^4=\tilde{\sigma}_x^4+\sum_{0\neq\xi\in\widehat{\Phi_x-v_x}}\left|(\widehat{\tilde{f}_x-\tilde{\sigma}})(\xi)\right|^4.
  \]
  Since $|\tilde{\sigma}_x-\sigma|<\tau/64$, it follows that there exists a nonzero $\xi_x\in\widehat{\Phi_x-v_x}$ such that
  \[
    \left|\E_{y\in\Phi_x}(\tilde{f}_x-\sigma)(y)e_p(\xi_x\cdot y)\right|\geq\f{\tau}{16},
  \]
  where we have crucially used that $\tilde{f}_x$ is $1$-bounded. As $\|f_x-\tilde{f}_x\|_{U^2(\Phi_x)}^4\leq\tau/32$, it therefore follows that
  \[
    \left|\E_{y\in\Phi_x}(f_x-\sigma)(y)e_p(\xi_x\cdot y)\right|\geq\f{\tau}{32}
  \]
  for every $x\in A_0$.

  Extend $x\mapsto \xi_x$ from $A_0$ to $A$ by picking a nonzero $\xi_x\in\widehat{\Phi_x-v_x}$ arbitrarily for all $x\in A\setminus A_0$. We now split the average over $y\in\Phi_x$ above into an average of averages over cosets of $\langle\xi_x\rangle^\perp$ in $\Phi_x$ and average over all of $A$ to get that
  \[
    \E_{x\in A}\E_{t\in\F_p}\left|\E_{\substack{y\in\Phi_x \\ \xi_x\cdot y=t}}(f_x-\sigma)(y)\right|\gg\tau,
    \]
    and use the fact that
  \[
    \E_{x\in A}\E_{t\in\F_p}\E_{\substack{y\in\Phi_x \\ \xi_x\cdot y=t}}(f_x-\sigma)(y)=\E_{x\in A}\E_{y\in\Phi_x}(f_x-\sigma)(y)=0
  \]
  to deduce that
  \[
    \E_{x\in A}\E_{t\in\F_p}\max\left(0,\E_{\substack{y\in\Phi_x \\ \xi_x\cdot y=t}}(f_x-\sigma)(y)\right)\gg\tau.
  \]
  By applying the pigeonhole principle in the $x$ and $t$ variables, it follows that there exists a subset $A_1\subset A$ of density $\gg\tau$ in $A$ and, for each $x\in A_1$, an element $t_x\in \F_p$ for which
  \[
    \E_{\substack{y\in\Phi_x \\ \xi_x\cdot y=t_x}}(f_x-\sigma)(y)\gg\tau.
  \]
  Thus, recalling the definition of $f_x$, we have
  \begin{equation}\label{1almostinc}
    \E_{x\in A_1}\E_{\substack{y\in\Phi_x \\ \xi_x\cdot y=t_x}}S(x,y)\geq \left(\sigma+\Omega\left(\tau\right)\right)\beta\gamma\delta.
  \end{equation}
  Define $\phi:\F_p^n\to\F_p^n$ by taking $\phi(x)=\xi_x$ for all $x\in A$ and $\phi(x)$ to be an arbitrary element of $(\Phi_x-v_x)^\perp\setminus\{0\}$ for all $x\in \F_p^n\setminus A$, and similarly extend $x\mapsto t_x$ from $A_1$ to $A$ by taking $t_x$ to be an arbitrary element of $\F_p$ for which
  \[
    \E_{\substack{y\in\Phi_x \\ \phi(x)\cdot y=t_x}}S(x,y)\geq\E_{y\in\Phi_x}S(x,y)
  \]
  for all $x\in A\setminus A_1$. Such an element must exist by the pigeonhole principle. Set
  \[
    \Psi:=\{(x,y)\in\Phi:\phi(x)\cdot y=t_x\},
  \]
  so that $\codim\{y\in\F_p^n:\Psi(x,y)=1\}=d+1$ for every $x\in A$, and $\alpha_1:=|A_1|/p^n$. Then~\eqref{1almostinc} can be rewritten as
  \[
    \E_{x,y}S(x,y)A_1(x)\Psi(x,y)\geq\left(\sigma+\Omega\left(\tau\right)\right)\f{\alpha_1\beta\gamma\delta\rho}{p}.
  \]

  It remains to check that the density $\E_{x,y}A_1(x)B(y)C(x+y)D(2x+y)\Psi(x,y)$ is close to $\alpha_1\beta\gamma\delta\rho/p$, so that we indeed have the desired density-increment. But by Lemmas~\ref{subspaceavg} and~\ref{cs2}, we have
  \[
    \E_{y}B(y)C(x+y)D(2x+y)\Psi(x,y)=\f{\beta\gamma\delta\rho}{p}+O\left(\ve^{1/8}\right)
  \]
  for all but a $O(\sqrt{\ve})$-proportion of $x\in A$, from which it follows that
  \[
    \E_{x,y}A_1(x)B(y)C(x+y)D(2x+y)\Psi(x,y)=\f{\alpha_1\beta\gamma\delta\rho}{p}+O\left(\ve^{1/8}\right).
  \]
  Thus, $S$ has density at least $\sigma+\Omega(\tau)$ on
  \[
    Q:=\{(x,y)\in\F_p^n\times\F_p^n:A_1(x)B(y)C(x+y)D(2x+y)\Psi(x,y)=1\},
  \]
  provided that $c_2$ is large enough. The conclusion of the lemma now follows from Lemma~\ref{uxtou}

  Now suppose that~\eqref{deg22} holds. By writing $S=g_S+\sigma T$ and arguing as in the
  first case, we may proceed under the assumption that
  \[
    \E_{x,x',y,y'}S(x,y)S(x,y')S(x',y)S(x',y')\geq\left(\sigma^4+\f{\tau}{2}\right)\alpha^2\beta^2\gamma^4\delta^4\rho^4.
  \]
  We will first show that either
  \[
    \E_{x,y}S(x,y')S(x',y)C(x+y)D(2x+y)\Phi(x,y)=\left(\sigma^2+O(\tau^2)\right)\alpha\beta\gamma^3\delta^3\rho^3
  \]
  for almost every pair $(x',y')\in S$, or else we can deduce the desired density-increment using Lemma~\ref{deg1inverse}.

  Consider the average
  \[
    \E_{x,x',y,y'}S(x,y')S(x',y)S(x',y')C(x+y)D(2x+y)\Phi(x,y).
  \]
  Using that $S=g_S+\sigma T$, the above can be written as
  \begin{equation}\label{x'y'avg1}
    \sigma^3\E_{x,x',y,y'}T(x,y')T(x',y)T(x',y')C(x+y)D(2x+y)\Phi(x,y)
  \end{equation}
  plus seven other terms of the form
  \begin{equation}\label{x'y'avg2}
    \E_{x,x',y,y'}g_0(x,y')g_1(x',y)g_2(x',y')C(x+y)D(2x+y)\Phi(x,y),
  \end{equation}
  where $g_0,g_1,$ and $g_2$ all equal $g_S$ or $\sigma T$ and at least one $g_i$ equals
  $g_S$. By Lemmas~\ref{subspaceavg},~\ref{cs2}, and~\ref{intersection}, the
  quantity~\eqref{x'y'avg1} equals
  $\sigma^3\alpha^2\beta^2\gamma^4\delta^4\rho^4+O(\ve^{\Omega(1)}/\rho^{O(1)})$. Suppose
  that $k$ of the functions $g_0,g_1,$ and $g_2$ in~\eqref{x'y'avg2} equal $\sigma T$. By
  a similar argument to those used to prove Lemma~\ref{deg2to1}, if any term of the
  form~\eqref{x'y'avg2} has size at least
  $\tau^4\sigma^k\alpha^2\beta^2\gamma^4\delta^4\rho^4$, then
  \[
    \E_{x,y,h}\Delta_{(0,h)}g_S(x,y)\geq \tau^8\alpha\beta^2\gamma^2\delta^2\rho^2+O\left(\f{\ve^{\Omega(1)}}{\rho^{O(1)}}\right)
  \]
  or
  \[
    \E_{x,y,h}\Delta_{(h,0)}g_S(x,y)\geq \tau^8\alpha^2\beta\gamma^2\delta^2\rho^2+O\left(\f{\ve^{\Omega(1)}}{\rho^{O(1)}}\right),
  \]
  so that the desired density-increment follows from Lemma~\ref{deg1inverse}. Thus, we may proceed under the assumption that
  \[
    \E_{(x',y')\in S}\E_{x,y}S(x,y')S(x',y)C(x+y)D(2x+y)\Phi(x,y)=\sigma^2\alpha\beta\gamma^3\delta^3\rho^3+O(\tau^4\alpha\beta\gamma^3\delta^3\rho^3)
  \]

  Now consider the average
  \[
    \E_{x',y'}S(x',y')\left|\E_{x,y}S(x,y')S(x',y)C(x+y)D(2x+y)\Phi(x,y)\right|^2.
  \]
   Using that $S=g_S+\sigma T$, the above can be written as
  \begin{equation}\label{x'y'var1}
    \sigma^5\E_{x',y'}T(x',y')\left|\E_{x,y}T(x,y')T(x',y)C(x+y)D(2x+y)\Phi(x,y)\right|^2,
  \end{equation}
  plus 31 other terms of the form
  \begin{equation}\label{x'y'var2}
    \E_{x,y,x',y',w,z}g_0(x',y')g_1(x,y')g_2(x',y)g_3(z,y')g_4(x',w)C(x+y)C(z+w)D(2x+y)D(2z+w)\Phi(x,y)\Phi(z,w),
  \end{equation}
  where $g_0,g_1,g_2,g_3,$ and $g_4$ all equal $g_S$ or $\sigma T$ and at least one $g_i$
  equals $g_S$.

  By Lemmas~\ref{subspaceavg},~\ref{cs2}, and~\ref{intersection}, the
  quantity~\eqref{x'y'var1} equals
  $\sigma^5\alpha^3\beta^3\gamma^7\delta^7\rho^7+O(\ve^{\Omega(1)}/\rho^{O(1)})$. Suppose
  that $k$ of the functions $g_0,\dots,g_4$ in~\eqref{x'y'var2} equal $\sigma
  T$. Analogously to the situation for the first moment, if any of the terms of the
  form~\eqref{x'y'var2} has size at least
  $\tau^8\sigma^k\alpha^3\beta^3\gamma^7\delta^7\rho^7$, then we will be able to deduce
  the desired density-increment. The most involved case is when $g_0=\dots=g_4=g_S$. All other cases can be handled using a simpler version of the argument we are about to carry out.

  So, consider this most involved case, i.e., that
  \[
    \E_{x,y,x',y',w,z}g_S(x',y')g_S(x,y')g_S(x',y)g_S(z,y')g_S(x',w)C(x+y)C(z+w)D(2x+y)D(2z+w)\Phi(x,y)\Phi(z,w)
  \]
  has size at least $\tau^8\alpha^3\beta^3\gamma^7\delta^7\rho^7$. Applying the
  Cauchy--Schwarz inequality in the variables $x',y',y,$ and $w$ gives that
  \[
    \E_{x',y',y,w}T(x',y')T(x',y)T(x',w)
  \]
  times
  \begin{align*}
    \E_{x',y',y,w}\big(&\left|\E_{x,z}g_S(x,y')g_S(z,y')C(x+y)C(z+w)D(2x+w)D(2z+w)\Phi(x,y)\Phi(z,w)\right|^2 \\
                       &\cdot B(y)B(w)C(x'+y')C(x'+y)C(x'+w) \\
                       &\cdot D(2x'+y')D(2x'+y)D(2x'+w)\Phi(x',y')\Phi(x',y)\Phi(x',w)\big)
  \end{align*}
  is at least $\tau^8\alpha^{6}\beta^{6}\gamma^{14}\delta^{14}\rho^{14}$. By Lemmas~\ref{subspaceavg},~\ref{cs2}, and~\ref{intersection}, $\E_{x',y',y,w}T(x',y')T(x',y)T(x',w)=\alpha\beta^3\gamma^3\delta^3\rho^3+O(\ve^{\Omega(1)}/\rho^{O(1)})$. Expanding the square in the average above, this means that
  \begin{align*}
    \E_{x',y',y,w,x,z,u,v}\big(&g_S(x,y')g_S(z,y')g_S(u,y')g_S(v,y') \\
                               &B(y)B(w)C(x'+y')C(x'+y)C(x'+w) \\
                               &C(x+y)C(z+w)C(u+y)C(v+w) \\
                               &D(2x+w)D(2z+w)D(2u+w)D(2v+w) \\
                               &D(2x'+y')D(2x'+y)D(2x'+w) \\
                               &\Phi(x,y)\Phi(z,w)\Phi(u,y)\Phi(v,w)\Phi(x',y')\Phi(x',y)\Phi(x',w)\big)
  \end{align*}
  is $\gg\tau^8\alpha^5\beta^3\gamma^{11}\delta^{11}\rho^{11}$, provided that $c_1$ is small enough and $c_2$ is large enough. We can write the above as
  \[
    \E_{y',x,z,u,v}g_S(x,y')g_S(z,y')g_S(u,y')g_S(v,y') \mu(y',x,z,u,v),
  \]
  where
  \begin{align*}
    \mu(y',x,z,u,v):=\E_{x',y,w}\big(&B(y)B(w)C(x'+y')C(x'+y)C(x'+w) \\
                                     &C(x+y)C(z+w)C(u+y)C(v+w) \\
                                     &D(2x+w)D(2z+w)D(2u+w)D(2v+w) \\
                                     &D(2x'+y')D(2x'+y)D(2x'+w) \\
                                     &\Phi(x,y)\Phi(z,w)\Phi(u,y)\Phi(v,w)\Phi(x',y')\Phi(x',y)\Phi(x',w)\big).
  \end{align*}
  Yet more applications of Lemmas~\ref{subspaceavg},~\ref{cs2}, and~\ref{intersection} give the estimate
  \[
    \E_{y',x,z,u,v}A(x)A(z)A(u)A(w)\left|\mu(y',x,z,u,v)-\alpha\beta^2\gamma^7\delta^7\rho^7\right|^2\ll\f{\ve^{\Omega(1)}}{\rho^{O(1)}},
  \]
  so that
  \[
    \E_{y',x,z,u,v}g_S(x,y')g_S(z,y')g_S(u,y')g_S(v,y')\gg \tau^8\alpha^4\beta\gamma^{4}\delta^{4}\rho^{4}.
  \]
  It follows from one more application of the Cauchy--Schwarz inequality in the variables $y',z,u,$ and $v$ and a similar analysis to above that
  \[
    \E_{x,y,h}\Delta_{(h,0)}g_S(x,y')\gg \tau^{16}\alpha^2\beta\gamma^{2}\delta^{2}\rho^{2},
  \]
  which, combined with Lemma~\ref{deg1inverse}, gives the desired density-increment.

  Thus, we may also proceed under the assumption that
  \[
    \E_{(x',y')\in S}\left|\E_{x,y}S(x,y')S(x',y)C(x+y)D(2x+y)\Phi(x,y)-\sigma^2\alpha\beta\gamma^3\delta^3\rho^3\right|^2\ll \tau^4\alpha^2\beta^2\gamma^6\delta^6\rho^6.
  \]
  By Markov's inequality, we therefore have
  \begin{equation}\label{Sdensity}
    \E_{x,y}S(x,y')S(x',y)C(x+y)D(2x+y)\Phi(x,y)=(\sigma^2+O(\tau^2))\alpha\beta\gamma^3\delta^3\rho^3
  \end{equation}
  for all but a $O(\tau^2)$-proportion of $(x',y')\in S$. As a consequence, there exists a
  pair $(x',y')\in S$ for which both~\eqref{Sdensity} holds and
  \[
    \E_{x,y}S(x,y)S(x,y')S(x',y)\geq\left(\sigma^3+\f{\tau}{4}\right)\alpha\beta\gamma^3\delta^3\rho^3,
  \]
  which together imply that
  \[
    \f{|S\cap T'|}{|T'|}\geq \sigma+\Omega(\tau)
  \]
  when we take $A'(x)=S(x,y')$, $B'(y)=S(x',y)$, $C'=C$, $D'=D$, and $\Phi'=\Phi$ in the definition of $T'$.
\end{proof}

\subsection{More preliminaries for $\|\cdot\|_{\star_1}$}

We will similarly need that certain $\|\cdot\|_{\star_1}$-inner products are controlled by the degree $1$ and $2$ directional uniformity norms appearing in the previous subsections. The proof of the lemma below is similar to the proof of Lemma~\ref{deg2to1}, but with an extra application of the Cauchy--Schwarz inequality in some cases.

\begin{lemma}\label{deg3to2}
  Let $d$ be a nonnegative integer, and set $\rho:=p^{-d}$. Suppose that $A,B,C,D\subset\F_p^n$ have densities $\alpha,\beta,\gamma,\delta$, respectively, and that $\Phi\subset\F_p^n\times\F_p^n$ takes the form
  \[
    \Phi=\{(x,y)\in A\times\F_p^n:y\in u+V_x\},
  \]
  where each $V_x$ is a subspace of $\F_p^n$ of codimension $d$. Let $\ve>0$, and assume that
  \[
    \|A-\alpha\|_{U^{10}(\F_p^n)},\|B-\beta\|_{U^{10}(\F_p^n)},\|C-\gamma\|_{U^{10}(\F_p^n)},\|D-\delta\|_{U^{10}(\F_p^n)},\|\Phi-\alpha\rho\|_{U^{8}(\F_p^n\times\F_p^n)}<\ve.
  \]
  Define $T$ by~\eqref{Tform}, and let $S\subset T$ and $\tau>0$. If
  \[
    \left|\E_{x,y,h_1,h_2,h_3}\prod_{\omega\in\{0,1\}^3}g_\omega((x,y)+\omega\cdot((0,h_1),(0,h_2),(h_3,0)))\right|\geq\tau\alpha^2\beta^4\gamma^8\delta^8\rho^6,
  \]
  where at least one $g_\omega$ equals $T$ and another equals $g_S$, then
  \[
    \E_{x,y,h}\Delta_{(0,h)}g_S(x,y)\geq\tau^8\alpha\beta^2\gamma^2\delta^2\rho^2+O\left(\f{\ve^{\Omega(1)}}{(\alpha\beta\gamma\delta\rho)^{O(1)}}\right),
  \]
  \[
    \E_{x,y,h}\Delta_{(h,0)}g_S(x,y)\geq\tau^8\alpha^2\beta\gamma^2\delta^2\rho^2+O\left(\f{\ve^{\Omega(1)}}{(\alpha\beta\gamma\delta\rho)^{O(1)}}\right),
  \]
  \[
    \E_{x,y,h,k}\Delta_{(0,h),(0,k)}g_S(x,y)\geq\tau^8\alpha\beta^4\gamma^4\delta^4\rho^3+O\left(\f{\ve^{\Omega(1)}}{(\alpha\beta\gamma\delta\rho)^{O(1)}}\right),
  \]
  or
  \[
    \E_{x,y,h,k}\Delta_{(h,0),(0,k)}g_S(x,y)\geq\tau^8\alpha^2\beta^2\gamma^4\delta^4\rho^4+O\left(\f{\ve^{\Omega(1)}}{(\alpha\beta\gamma\delta\rho)^{O(1)}}\right).
  \]
\end{lemma}

Our final preliminary lemma says that, for almost every $(x,x+h_3)\in A^2$, the function $\Delta_{(h_3,0)}S(x,\cdot)$ is supported on a Fourier uniform subset of the affine subspace $\{y\in\F_p^n:\Delta_{(h_3,0)}\Phi(x,y)=1\}$.
\begin{lemma}\label{RPhi1}
Let $d$ be a nonnegative integer, and set $\rho:=p^{-d}$. Suppose that $A,B,C,D\subset\F_p^n$ have densities $\alpha,\beta,\gamma,\delta$, respectively, and that $\Phi\subset\F_p^n\times\F_p^n$ takes the form
  \[
    \Phi=\{(x,y)\in A\times\F_p^n:y\in u+V_x\},
  \]
  where each $V_x$ is a subspace of $\F_p^n$ of codimension $d$. Let $\ve>0$, and assume that
  \[
    \|A-\alpha\|_{U^{4}(\F_p^n)},\|B-\beta\|_{U^{4}(\F_p^n)},\|C-\gamma\|_{U^{4}(\F_p^n)},\|D-\delta\|_{U^{4}(\F_p^n)},\|\Phi-\alpha\rho\|_{U^{4}(\F_p^n\times\F_p^n)}<\ve.
  \]
  Setting
  \[
    R_{x,h}(y):=B(y)\Delta_hC(x+y)\Delta_{2h}D(2x+y)
  \]
  and
  \[
    \Phi_{x,h}:=\Delta_{(h,0)}\Phi,
  \]
  then the probability
  \[
    \PP\left((x,x+h)\in A^2:\codim\{y\in\F_p^n:\Phi_{x,h}(y)=1\}\neq 2d\text{ or }\|R_{x,h}\Phi_{x,h}-\beta\gamma^2\delta^2\|_{U^2(\Phi_{x,h})}\geq\f{\ve^{1/32}}{\rho^{3/2}}\right)
  \]
  is $\ll\ve^{\Omega(1)}/\rho^{O(1)}$.
\end{lemma}
\begin{proof}
  By Lemmas~\ref{cs2} and~\ref{intersection}, we have that $\|R_{x,y}-\beta\gamma^2\delta^2\|_{U^2(\F_p^n)}\geq\ve^{1/8}$ or $\codim\{y\in\F_p^n:\Phi_{x,h}(y)=1\}\neq 2d$ for at most a $O(\ve^{\Omega(1)}/\rho^{O(1)})$-proportion of pairs $(x,x+h)\in A^2$. For all of these typical pairs $(x,h)$, we have
  \[
    \|R_{x,h}\Phi_{x,h}-\beta\gamma^2\delta^2\|_{U^2(\Phi_{x,h})}^4=\rho^{-6}\E_{y,k,\ell}\Delta_{k,\ell}[R_{x,h}-\beta\gamma^2\delta^2](y)\Phi_{x,h}(y)\Phi_{x,h}(y+k)\Phi_{x,h}(y+\ell),
  \]
  which is at most
  \begin{align*}
    \sum_{\xi,\eta,\nu\in(\Phi_{x,h}-u)^\perp}\big|\E_{y,k,\ell}&[R_{x,h}-\beta\gamma^2\delta^2](y)e_p(\xi\cdot y)[R_{x,h}-\beta\gamma^2\delta^2](y+k)e_p(\eta\cdot (y+k)) \\
    &[R_{x,h}-\beta\gamma^2\delta^2](y+\ell)e_p(\nu\cdot (y+\ell))[R_{x,h}-\beta\gamma^2\delta^2](y+k+\ell)\big|
  \end{align*}
  by inserting the identity
  \[
    \Phi_{x,h}(z)=\f{1}{p^{2d}}\sum_{\xi\in(\Phi_{x,h}-u)^\perp}e_p(\xi\cdot[z-u])
  \]
  for every $z\in\F_p^n$. For each fixed triple $(\xi,\eta,\nu)$, the interior average above is bounded by $\|R_{x,h}-\beta\gamma^2\delta^2\|_{U^2(\F_p^n)}$ by the Gowers--Cauchy--Schwarz inequality. Since there are $\rho^{-6}$ possible triples to sum over, we must have
  \[
    \|R_{x,h}\Phi_{x,h}-\beta\gamma^2\delta^2\|_{U^2(\Phi_{x,h})}^4\leq\rho^{-6}\|R_{x,h}-\beta\gamma^2\delta^2\|_{U^2(\F_p^n)}<\ve^{1/8}/\rho^6,
  \]
  from which the conclusion of the lemma follows.
\end{proof}

\subsection{Proof of Theorem~\ref{starinverse}}

Now we can finally finish the proof of Theorem~\ref{starinverse}.

\begin{proof}[Proof of Theorem~\ref{starinverse}]
  First assume that
  \[
    \|g_S\|_{\star_1}\geq\tau\alpha^{1/4}\beta^{1/2}\gamma\delta\rho^{3/4}.
  \]
  Analogously to the proof of Lemma~\ref{deg2inverse}, by writing $S=g_S+\sigma T$, we see that $\|S\|_{\star_1}^8$ equals
  \[
    \sigma^8\|T\|_{\star_1}^8+\|g_S\|_{\star_1}^8
  \]
  plus 62 terms of the form
  \begin{equation}\label{star1lower}
    \E_{x,y,h_1,h_2,h_3}\prod_{\omega\in\{0,1\}^3}g_\omega((x,y)+\omega\cdot((0,h_1),(0,h_2),(h_3,0))),
  \end{equation}
  where at least one $g_\omega$ equals $\sigma T$ and at least one other equals $g_S$. Note that
  \[
    \|T\|_{\star_1}^8=\alpha^2\beta^4\gamma^8\delta^8\rho^6+O\left(\f{\ve^{\Omega(1)}}{\rho^{O(1)}}\right)
  \]
  by Lemmas~\ref{subspaceavg},~\ref{cs2}, and~\ref{intersection}. If any of the terms~\eqref{star1lower} have absolute value at least $\f{1}{128}\tau^8\alpha^2\beta^4\gamma^8\delta^8\rho^6$, then combining Lemma~\ref{deg3to2} with Lemma~\ref{deg1inverse} or Lemma~\ref{deg2inverse} produces the desired density increment. We may thus proceed under the assumption that these terms are all small, so that
  \[
    \E_{x,y,h_1,h_2,h_3}\Delta_{(0,h_1),(0,h_2),(h_3,0)}S(x,y)\geq\left(\sigma^8+\f{\tau^8}{2}\right)\alpha^2\beta^4\gamma^8\delta^8\rho^6.
  \]

  For each pair $(x,x+h)\in A^2$, let $R_{x,h}$ and $\Phi_{x,h}$ be as in
  Lemma~\ref{RPhi1}, and set $S_{x,h}(y):=\Delta_{(h,0)}S(x,y)$. By Markov's inequality, either
  \begin{equation}\label{Sxhdensity}
    \PP\left((x,x+h)\in A\times A:\left|\E_{y}S_{x,h}(y)-\sigma^2\beta\gamma^2\delta^2\rho^2\right|\geq \f{\tau^8}{64}\beta\gamma^2\delta^2\rho^2\right)<\f{\tau^8}{4},
  \end{equation}
  or else $\E_{x,x+h\in A}|\E_{y}S_{x,h}(y)-\sigma^2\beta\gamma^2\delta^2\rho^2|^2$ is
  $\gg\tau^{16}\beta^2\gamma^4\delta^4\rho^4$, in which case a combination of
  Lemmas~\ref{deg1inverse},~\ref{deg2to1}, and~\ref{deg2inverse} will produce the desired
  density increment. We may thus proceed under the assumption that~\eqref{Sxhdensity} holds.

  By~\eqref{Sxhdensity} and Lemma~\ref{RPhi1}, there exists a subset $A_0\subset \{(x,h)\in\F_p^n\times\F_p^n:x,x+h\in A\}$ of density $\gg\tau^{O(1)}\alpha^2$ in $\F_p^n\times\F_p^n$ such that
  \[
    \|S_{x,h}\|_{U^2(\Phi_{x,h})}^4\geq \left(\sigma^8+\f{\tau^8}{4}\right)\beta^4\gamma^8\delta^8,
  \]
  \[
    \left|\E_{y}S_{x,h}(y)-\sigma^2\beta\gamma^2\delta^2\rho^2\right|< \f{\tau^8}{64}\beta\gamma^2\delta^2\rho^2,
  \]
  \[
    \|R_{x,h}\Phi_{x,h}-\beta\gamma^2\delta^2\|_{U^2(\Phi_{x,h})}<\f{\ve^{1/32}}{\rho^{3/2}},
  \]
  and $\codim\{y\in\F_p^n:\Phi_{x,h}(y)=1\}=2d$ for all $(x,h)\in A_0$. Note that $S_{x,h}$ is supported on $R_{x,h}\Phi_{x,h}$, and, setting
  \[
    f_{x,h}:=\f{1}{\beta\gamma^2\delta^2}S_{x,h}\qquad\text{and}\qquad \nu_{x,h}:=\f{1}{\beta\gamma^2\delta^2}R_{x,h}\Phi_{x,h},
  \]
  we have $\|f_{x,h}\|^4_{U^2(\Phi_{x,h})}\geq\sigma^8+\tau^8/4$, $0\leq f_{x,h}\leq\nu_{x,h}$, $\E f_{x,h}\leq 1$, $|\E f_{x,h}-\sigma^2|<\tau^8/64$, and $\|\nu_{x,h}-1\|_{U^2(\Phi_{x,h})}<\ve^{1/32}/\beta\gamma^2\delta^2\rho^{3/2}<\exp(-(64/\tau)^{c_3})$, provided that $c_1$ is small enough and $c_2$ is large enough. Applying Lemma~\ref{densemodel} on the affine subspace $\Phi_{x,h}$ yields a function $\tilde{f}_{x,h}:\Phi_{x,h}\to[0,1]$ such that
  \[
    \E \tilde{f}_{x,h}=\E f_{x,h}\qquad\text{ and }\qquad\|\tilde{f}_{x,h}-f_{x,h}\|_{U^2(\Phi_{x,h})}\leq \f{\tau^8}{64},
  \]
  provided that $c_3$ is large enough. We must then also have
  \[
    \|\tilde{f}_{x,h}\|^4_{U^2(\Phi_{x,h})}\geq\sigma^8+\f{\tau^8}{16}.
  \]
  Arguing as in the proof of the first part of Lemma~\ref{deg2inverse}, it follows that, for every pair $(x,h)\in A_0$, there exists a nonzero $\xi_{x,h}\in\widehat{\Phi_{x,h}-u}$ such that
  \[
    \left|\E_{y\in\Phi_{x,h}}(\tilde{f}_{x,h}-\sigma^2)(y)e_p(\xi_{x,h}\cdot y)\right|\geq\f{\tau^8}{32},
  \]
  so that
  \[
    \left|\E_{y\in\Phi_{x,h}}(f_{x,h}-\sigma^2)(y)e_p(\xi_{x,h}\cdot y)\right|\geq\f{\tau^8}{64}
  \]
  as well.

Extend $(x,h)\mapsto \xi_{x,h}$ from $A_0$ to the set $\{(x,h)\in\F_p^n\times\F_p^n:x,x+h\in A\}$ by picking a nonzero $\xi_{x,h}\in\widehat{\Phi_{x,h}-u}$ arbitrarily for all pairs outside of $A_0$. We split the average over $y\in\Phi_{x,h}$ up into an average of averages over cosets of $\langle \xi_{x,h}\rangle^\perp$ and average over all pairs $(x,h)$ such that $x,x+h\in A$ to get that
  \begin{equation}\label{xxhavg}
    \E_{x,x+h\in A}\E_{t\in\F_p}\left|\E_{\substack{y\in\Phi_{x,h} \\ \xi_{x,h}\cdot y=t}}(f_{x,h}-\sigma^2)(y)\right|\gg\tau^{c}
  \end{equation}
  for some absolute constant $c>0$. Note that
  \begin{align*}
    \E_{x,x+h\in A}\E_{t\in\F_p}\E_{\substack{y\in\Phi_{x,h} \\ \xi_{x,h}\cdot y=t}}(f_{x,h}-\sigma^2)(y) &= \f{1}{\alpha^2\beta\gamma^2\delta^2\rho^2}\E_{x,y,h}S(x,y)S(x+h,y)-\sigma^2 \\
                                                                                                          &= \f{1}{\alpha^2\beta\gamma^2\delta^2\rho^2}\E_{x,y,h}g_S(x,y)g_S(x+h,y),
  \end{align*}
  so that, if it were the case that
  \[
    \left|\E_{x,x+h\in A}\E_{t\in\F_p}\E_{\substack{y\in\Phi_{x,h} \\ \xi_{x,h}\cdot y=t}}(f_{x,h}-\sigma^2)(y)\right|\gg\tau^{2c},
  \]
  then we would be able to deduce the desired density-increment from Lemma~\ref{deg1inverse}. Thus, we may proceed under the assumption
  \[
    \left|\E_{x,x+h\in A}\E_{t\in\F_p}\E_{\substack{y\in\Phi_{x,h} \\ \xi_{x,h}\cdot y=t}}(f_{x,h}-\sigma^2)(y)\right|\ll\tau^{2c},
  \]
  which we can combine with~\eqref{xxhavg}, a change of variables, and an application of
  the pigeonhole principle in the $t$ variable to deduce that
  \begin{equation}\label{eq:star1pigeonhole}
    \E_{x,h\in A}\max\left(0,\E_{\substack{y\in\Phi_{x,h-x} \\ \xi_{x,h-x}\cdot y=t}}(f_{x,h-x}-\sigma^2)(y)\right)\gg\tau^{c}
  \end{equation}
  for some fixed $t\in\F_p$. Set
  \begin{equation}
    \label{eq:Psidef}
    \Psi_h:=\left\{(x,y)\in\F_p^n\times\F_p^n:y\in\Phi_{x,h-x}\text{ and }\xi_{x,h-x}\cdot y= t\right\}.
  \end{equation}
  To deduce the desired density-increment by applying the pigeonhole principle
  to~\eqref{eq:star1pigeonhole}, we will have to show that almost every set of the form
  \begin{equation*}
    \{(x,y)\in T: S(h,y)A_1(x,h-x)\Psi_{h}(x,y)=1\}
  \end{equation*}
  has close to the ``correct'' density. Most of the remainder of our argument for
  $\|\cdot\|_{\star_1}$ is devoted to this task.

  Set $Q_{x,h}(y):=S(h,y)C(x+y)D(2x+y)$. We start by showing that either
  $\|Q_{x,h}-\sigma\beta\gamma^2\delta^2\|_{U^2(\Phi_{x,h-x})}$ is small for almost every
  pair $(x,h)\in A\times A$, or else we can deduce a density-increment from
  Lemmas~\ref{deg1inverse} and~\ref{deg2inverse}. Note first that
  \begin{align*}
    \E_{x,y,h}A(x)A(h)\Phi_{x,h-x}(y)Q_{x,h}(y) &= \E_{x,y,h}S(h,y)A(x)C(x+y)D(2x+y)\Phi(x,y) \\
    &=\E_{h,y}S(h,y)\mu(y),
  \end{align*}
  where $\mu(y):=\E_{x}A(x)C(x+y)D(2x+y)\Phi(x,y)$. By
  Lemmas~\ref{subspaceavg},~\ref{cs2}, and~\ref{intersection}, we have
  \[
    \PP\left(y\in\F_p^n:|\mu(y)-\alpha\gamma\delta\rho|>\ve\right)\ll\f{\ve^{\Omega(1)}}{\rho^{O(1)}},
  \]
  so that
  $|\E_{x,h,y}A(x)A(h)\Phi_{x,h-x}(y)Q_{x,h}(y)-\sigma\alpha^2\beta\gamma^2\delta^2\rho^2|<\ve^{\Omega(1)}/\rho^{O(1)}$. Similarly,
  the average
  $\E_{x,h}A(x)A(h)|\E_yQ_{x,h}(y)\Phi_{x,h-x}(y)-\sigma\beta\gamma^2\delta^2\rho^2|^2$
  equals
  \[
    \E_{x,y,h,k}A(x)A(h)\Phi_{x,h-x}(y)\Phi_{x,h-x}(y+k)Q_{x,h}(y)Q_{x,h}(y+k)-\sigma^2\alpha^2\beta^2\gamma^4\delta^4\rho^4+O\left(\f{\ve^{\Omega(1)}}{\rho^{O(1)}}\right),
  \]
  and $\E_{x,y,h,k}A(x)A(h)\Phi_{x,h-x}(y)\Phi_{x,h-x}(y+k)Q_{x,h}(y)Q_{x,h}(y+k)$ equals
  \[
    \E_{y,h,k}S(h,y)S(h,y+k)\mu'(y,k),
  \]
  where
  \[
    \mu'(y,k)=\E_xA(x)C(x+y)C(x+y+k)D(2x+y)D(2x+y+k)\Phi(x,y)\Phi(x,y+k).
  \]
  By Lemmas~\ref{subspaceavg},~\ref{cs2}, and~\ref{intersection}, we have
  \[
    \PP\left((y,k)\in\F_p^n\times\F_p^n:|\mu'(y,k)-\alpha\gamma^2\delta^2\rho^2|>\ve\right)\ll\f{\ve^{\Omega(1)}}{\rho^{O(1)}},
  \]
  so that
  \[
    \E_{x,y,h,k}A(x)A(h)\Phi_{x,h-x}(y)\Phi_{x,h-x}(y+k)Q_{x,h}(y)Q_{x,h}(y+k)
  \]
  equals
  \[
\alpha\gamma^2\delta^2\rho^2\E_{y,h,k}S(h,y)S(h,y+k)+O\left(\f{\ve^{\Omega(1)}}{\rho^{O(1)}}\right).
  \]
  
  Thus, if
  \[
    \E_{x,h}A(x)A(h)|\E_yQ_{x,h}(y)\Phi_{x,h-x}(y)-\sigma\beta\gamma^2\delta^2\rho^2|^2\gg \tau^{4c'}\sigma^2\alpha^2\beta^2\gamma^4\delta^4\rho^4
  \]
  for some $c'>0$ (to be chosen later), and $c_1$ is small enough and $c_2$ is large enough, then
  \[
    \E_{x,y,h}S(x,y)S(x,y+h)\geq \left(\sigma^2+\Omega(\tau^{4c'})\right)\alpha\beta^2\gamma^2\delta^2\rho^2,
  \]
  in which case the desired density increment follows from Lemma~\ref{deg1inverse}. We may
  thus proceed under the assumption that
  \[
    \E_{x,h}A(x)A(h)|\E_yQ_{x,h}(y)\Phi_{x,h-x}(y)-\sigma\beta\gamma^2\delta^2\rho^2|^2\ll \tau^{4c'}\sigma^2\alpha^2\beta^2\gamma^4\delta^4\rho^4,
  \]
  so that, by Markov's inequality, we have
  \[
    \PP\left((x,h)\in A\times A: |\E_yQ_{x,h}(y)\Phi_{x,h-x}(y)-\sigma\beta\gamma^2\delta^2\rho^2|\geq \tau^{2c'}\sigma\beta\gamma^2\delta^2\rho^2\right)\ll \tau^{2c'}.
  \]
  It follows that
  $\E_{x,h}A(x)A(h)\|Q_{x,h}-\sigma\beta\gamma^2\delta^2\|_{U^2(\Phi_{x,h-x})}^4$ equals
  \[
    \E_{\substack{x,h\in\F_p^n \\ y,y+k,y+\ell\in\Phi_{x,h-x}}}A(x)A(h)\Delta_{k,\ell}Q_{x,h}(y)-\sigma^4\alpha^2\beta^4\gamma^8\delta^8+O(\tau^{2c'}\alpha^2\beta^4\gamma^8\delta^8).
  \]
  
  By Lemmas~\ref{subspaceavg} and~\ref{cs2},
  \[
    \E_{\substack{x,h\in\F_p^n \\ y,y+k,y+\ell\in\Phi_{x,h-x}}}A(x)A(h)\Delta_{k,\ell}Q_{x,h}(y)=\f{\alpha\gamma^4\delta^4}{\rho^3}\E_{x,y,k,\ell}\Delta_{(0,k),(0,\ell)}S(x,y)+O\left(\f{\ve^{\Omega(1)}}{\rho^{O(1)}}\right).
  \]
  If
  \[
    \left|\E_{x,y,k,\ell}\Delta_{(0,k),(0,\ell)}S(x,y)-\sigma^4\alpha\beta^4\gamma^4\delta^4\rho^3\right|\gg \tau^{4c'}\alpha\beta^4\gamma^4\delta^4\rho^3,
  \]
  then we could deduce the desired density increment from Lemma~\ref{deg2inverse}. So, we
  may proceed under the assumption that this inequality does not hold, which implies that
  $\E_{x,h}A(x)A(h)\|Q_{x,h}-\sigma\beta\gamma^2\delta^2\|_{U^2(\Phi_{x,h-x})}^4\ll\tau^{4c'}\alpha^2\beta^4\gamma^8\delta^8$. It
  then follows from Markov's inequality that
  \begin{equation}\label{eq:pigeonholerequirement1}
    \PP\left((x,h)\in A\times A: \|Q_{x,h}-\sigma\beta\gamma^2\delta^2\|_{U^2(\Phi_{x,h-x})}\gg \tau^{2c'}\beta\gamma^2\delta^2 \right)\ll \tau^{2c'}.
  \end{equation}

  Next, we will show that, for typical $(x,h)\in A\times A$, the average size of
  $S(x,y)S(h,y)\Psi_{h}(x,y)$ is not very large. Certainly,
  \begin{align*}
    \E_{y}S(x,y)S(h,y)\Psi_h(x,y)&\leq\E_yT(x,y)T(h,y)\Psi_h(x,y)\\
    &=\E_yB(y)C(x+y)C(h+y)D(2x+y)D(2h+y)\Psi_h(x,y)
  \end{align*}
  for every $(x,h)\in A\times A$. Setting
  \begin{equation*}
    F(x,h,y):=B(y)C(x+y)C(h+y)D(2x+y)D(2h+y),
  \end{equation*}
  Lemma~\ref{cs2} says that
  \begin{equation*}
    \PP\left((x,h)\in \F_p^n\times\F_p^n:\|F(x,h,\cdot)-\beta\gamma^2\delta^2\|_{U^2(\F_p^n)}>\ve^{1/8}\right)\ll\sqrt{\ve}.
  \end{equation*}
  Thus, by~\ref{subspaceavg}, as long as $c_2$ is large enough, we have
  \begin{equation}\label{eq:pigeonholerequirement2}
    \PP\left((x,h)\in\F_p^n\times\F_p^n:\left|\E_yS(x,y)S(h,y)\Psi_h(x,y)\right|>100\frac{\beta\gamma^2\delta^2\rho^2}{p}\right)\ll \sqrt{c_1}(\sigma\tau\alpha\beta\gamma\delta\rho)^{16c},
  \end{equation}
  say.

  Now, setting $c'=8c$, the contribution to the left-hand side
  of~\eqref{eq:star1pigeonhole} coming from pairs $(x,h)\in A\times A$ for which
  $\|Q_{x,h}-\sigma\beta\gamma^2\delta^2\|_{U^2(\Phi_{x,h-x})}\gg\tau^{16c}\beta\gamma^2\delta^2$
  or $|\E_y S(x,y)S(h,y)\Psi_h(x,y)|>100\beta\gamma^2\delta^2\rho^2/p$ is
  $\ll\tau^{16c}$. Thus, by the pigeonhole principle, there exists an $h\in A$ and a
  subset $A'\subset A$ of density $\alpha'\gg\tau^{O(1)}\alpha$ in $\F_p^n$ such that
  \begin{equation*}
    \E_{x\in A'}\E_{(x,y)\in\Psi_h}(f_{x,h}-\sigma^2)(y)\gg\tau^c
  \end{equation*}
  and
  $\|Q_{x,h}-\sigma\beta\gamma^2\delta^2\|_{U^2(\Phi_{x,h-x})}\ll\tau^{16c}\beta\gamma^2\delta^2$
  for every $x\in A'$. Recalling the definition of $f_{x,h}$, the above displayed equation
  says that
  \begin{equation*}
    \E_{x,y}S(x,y)A'(x)S(h,y)\Psi_h(x,y)\geq(\sigma+\Omega(\tau^c))\frac{\alpha'\sigma\beta\gamma^2\delta^2\rho^2}{p}.
  \end{equation*}  
  We now use Lemma~\ref{uxtou} to find a subset $A''\subset A'$ of density
  $\alpha''\gg\alpha'\rho\tau^{O(1)}$ and a $u'\in\F_p^n$ such that
  \[
    \E_{x,y}S(x,y)A'(x)S(h,y)\Phi_h'(x,y)\geq\left(\sigma+\Omega(\tau^{c})\right)\f{\alpha''\sigma\beta\gamma^2\delta^2\rho^2}{p},
  \]
  where
  \[
    \Phi_h'=\left\{(x,y)\in \F_p^n\times\F_p^n:y\in\Phi_{x,h-x}\text{ and }\xi_{x,h}\cdot(y-u')=0\right\}.
  \]
  This gives the conclusion of the lemma.

  The proofs of the two remaining cases are very similar to those appearing in earlier subsections, so we will be more brief in our arguments. Next, assume that
  \[
    \|g_S\|_{\star_2}\geq\tau\alpha^{1/2}\beta\gamma^{1/2}\delta\rho.
  \]
  As in the second part of the proof of Lemma~\ref{deg2inverse}, we may proceed under the assumption that
  \[
    \E_{x,y,x',y'}S(x,y)S(x,y'-x)S(-x',x+y+x')S(-x',x'+y')\geq\left(\sigma^4+\f{\tau^4}{2}\right)\alpha^2\beta^4\gamma^2\delta^4\rho^4,
  \]
  and use it to show that either
  \[
    \E_{x,y}B(y)D(2x+y)\Phi(x,y)S(x,y'-x)S(-x',x+y+x')=\left(\sigma^2+O(\tau^5)\right)\alpha\beta^3\gamma\delta^3\rho^3
  \]
  for almost every pair $(x',y')$ for which $(-x',x'+y')\in S$, or else the desired density-increment follows from Lemma~\ref{deg1inverse}. By Lemmas~\ref{subspaceavg},~\ref{cs2}, and~\ref{intersection} and the Cauchy--Schwarz inequality, either
  \[
    \E_{\substack{x',y' \\ (-x',x'+y')\in S}}\E_{x,y}B(y)D(2x+y)\Phi(x,y)S(x,y'-x)S(-x',x+y+x')=\sigma^2\alpha\beta^3\gamma\delta^3\rho^3+O(\tau^{16}\alpha\beta^3\gamma\delta^3\rho^3)
  \]
  and
  \[
    \E_{\substack{x',y' \\ (-x',x'+y')\in S}}\left|\E_{x,y}B(y)D(2x+y)\Phi(x,y)S(x,y'-x)S(-x',x+y+x')-\sigma^2\alpha\beta^3\gamma\delta^3\rho^3\right|^2\ll\tau^{32}\alpha^2\beta^6\gamma^2\delta^6\rho^6
  \]
  or else we have the desired density-increment from Lemma~\ref{deg1inverse}. It then follows from Markov's inequality that
  \[
    \E_{x,y}B(y)D(2x+y)\Phi(x,y)S(x,y'-x)S(-x',x+y+x')=\left(\sigma^2+O(\tau^5)\right)\alpha\beta^3\gamma\delta^3\rho^3
  \]
  for all but a $\tau/4$-proportion of pairs $(x',y')$ for which $(-x',x'+y')\in S$, which means that we can find such a pair for which we also have
  \[
    \E_{x,y}S(x,y)S(x,y'-x)S(-x',x+y+x')\geq\left(\sigma^4+\f{\tau^4}{2}\right)\alpha^2\beta^4\gamma^2\delta^4\rho^4,
  \]
  so that we get the desired density-increment by taking $A'=S(x,y'-x)$, $B'=B$, $C'=S(-x',x+y+x')$, $D'=D$, and $\Phi'=\Phi$ in the definition of $T'$.

  Now assume that
  \[
    \|g_S\|_{\star_3}\geq\tau \alpha\beta\gamma\delta^{1/2}\rho,
  \]
  which means
  \[
    \E_{z\in D}|\E_{2x+y=z}g_S(x,y)|^2\geq \tau^2 \alpha^2\beta^2\gamma^2\rho^2.
  \]
  By Lemma~\ref{fibersize} and the pigeonhole principle, there exists a subset $D_0\subset D$ of density at least $\tau^2/2$ for which
  \[
    |\E_{2x+y=z}g_S(x,y)|\geq \f{\tau \alpha\beta\gamma\rho}{4}
  \]
  whenever $z\in D_0$. As in the proof of Lemma~\ref{deg1inverse}, there is a subset $D_1\subset D_0$ of density at least $1/2$ in $D_0$ such that either $\E_{2x+y=z}g_S(x,y)\geq \tau \alpha\beta\gamma\rho/4$ or $\E_{2x+y=z}g_S(x,y)\leq -\tau \alpha\beta\gamma\rho/4$ for all $z\in D_1$. As in the proof of Lemma~\ref{deg1inverse}, we can simply take $D'=D_1$ and $D'=D\setminus D_1$ in these cases, respectively, since, by Lemma~\ref{fibersize}, the fibers $\{(x,y)\in T:2x+y=z\}$ typically have very close to their average size.
\end{proof}

\section{Pseudorandomization}\label{pseudo}

This section begins with yet more preliminaries. To prove Lemma~\ref{pseudoprop}, we will need a result of Cohen and Tal, which says that, for any finite set of polynomials in $\F_p[x_1,\dots,x_m]$, one can find a partition of $\F_p^m$ into affine subspaces of relatively large dimension on which all of the polynomials are constant.

\begin{theorem}[Cohen and Tal, Theorem~3.6 of~\cite{CohenTal15} specialized to prime fields]\label{CT}
  Let $d,m,$ and $t$ be natural numbers. There exists a positive integer $m'$ satisfying
  \[
    m'\gg \f{m^{1/(d-1)!}}{t^e}
  \]
  such that, for any polynomials $P_1,\dots,P_t\in\F_p[x_1,\dots,x_m]$ of degree at most
  $d$, there is a partition of $\F_p^m$ into affine subspaces of dimension $m'$ such that $P_1,\dots,P_t$
  are constant on each affine subspace.
\end{theorem}
To see that this formulation of Cohen and Tal's theorem is equivalent to Theorem~3.6
of~\cite{CohenTal15}, note that one can make all of the affine subspaces in the partitions
produced by their theorem have the same dimension by simply partitioning each subspace not
of the minimum possible dimension into more subspaces.

We will also need a ``bilinear'' version of Cohen and Tal's result, which we will use to
find partitions of $\F_p^n\times\F_p^n$ into product spaces of the form $(u+V)\times(w+V)$
on which polynomials in two sets of variables $x_1,\dots,x_m,y_1,\dots,y_m$ are constant.

\begin{corollary}\label{bilCT}
  Let $d$ and $m$ be natural numbers. There exists a positive integer $m'$ satisfying
  \[
    m'\gg\f{m^{1/(d-1)!^2}}{d^{e}}
  \]
  such that, for any polynomial $R\in\F_p[x_1,\dots,x_m,y_1,\dots,y_m]$ of degree at most $d$, there is a partition of $\F_p^m\times\F_p^m$ into products of affine subspaces of the form
  \[
    (u+V)\times(w+V),
  \]
  with each $\dim V= m'$, such that $R$ is constant on each product of affine subspaces.
\end{corollary}
\begin{proof}
   Write
  \[
    R(\mathbf{x},\mathbf{y})=\sum_{i=0}^dP_i(\mathbf{x})Q_i(\mathbf{y}),
  \]
  where $P_0,\dots,P_d,Q_0,\dots,Q_d\in\F_p[z_1,\dots,z_m]$ satisfy $\deg{P_i}\leq i$ and $\deg{Q_i}\leq d-i$ for all $0\leq i\leq d$. Applying Theorem~\ref{CT} to $P_0,\dots,P_d$ gives us a positive integer $m'_0\gg\f{m^{1/(d-1)!}}{(d+1)^e}$ and a partition
    \[
      \F_p^m=\coprod_{j\in J}(u_j+V_j)
    \]
    of $\F_p^m$, with $\dim{V_j}=m'_0$ for each $j\in J$, such that $P_0,\dots,P_d$ are all constant on each affine subspace $u_j+V_j$.

    Now write
    \[
      \F_p^m\times\F_p^m=\coprod_{j\in J}(u_j+V_j)\times \F_p^m = \coprod_{j\in J} \coprod_{w+V_j\in \F_p^m/V_j}(u_j+V_j)\times (w+V_j).
    \]
    Since restricting a polynomial to a subspace cannot increase its degree, for each $j\in J$, we can apply Theorem~\ref{CT} on each affine subspace $w+V_j$ to the polynomial
    \[
      Q_{j,w}(\mathbf{y}):=\sum_{i=0}^dP_i(u_j+V_j)Q_i(\mathbf{y})
    \]
    to get that there exists a positive integer $m'$ satisfying
    \[
      m'\gg\f{(m'_0)^{1/(d-1)!}}{(d+1)^e}\gg\f{m^{1/(d-1)!^2}}{d^{e}}
    \]
    and a partition
    \[
      w+V_j = \coprod_{k\in K_{j,w}}(w_k+V_k'),
    \]
    with each subspace $V_k'\leq V_j$ having $\dim V_k'=m'$, such that $Q_{j,w}$ is constant on each $w_k+V_k'$.

    Thus, we can write
    \[
      \F_p^m\times\F_p^m=\coprod_{j\in J}\coprod_{w+V_j\in\F_p^m/V_j}\coprod_{k\in K_{j,w}}(u_j+V_j)\times(w_k+V_k'),
    \]
    where $R$ is constant on each product $(u_j+V_j)\times(w_k+V_k')$. Since $V_k'\leq V_j$ we can further refine this partition to one of the desired form
    \[
      \F_p^m\times\F_p^m=\coprod_{j\in J}\coprod_{w+V_j\in\F_p^m/V_j}\coprod_{k\in K_{j,w}}\coprod_{u+V_k'\in V_j/V_k'}(u+V_k')\times(w_k+V_k'),
    \]
    on each part of which $R$ is still constant.
\end{proof}

Finally, we will recall the recent quantitative inverse theorem of Gowers and Mili\'cevi\'c for the $U^s$-norms on vector spaces over finite fields:

\begin{theorem}[Gowers and Mili\'cevi\'c, Theorem~7 of~\cite{GowersMilicevic20}]\label{GM}
  Let $s$ be a natural number and assume that $p\geq s$. There exist constants $c_s,c_{s,p}'>0$ so that, if $f:\F_p^m\to\C$ is a $1$-bounded function satisfying
  \[
    \|f\|_{U^s}\geq \delta,
  \]
  then there exists a polynomial $P\in\F_p[x_1,\dots,x_m]$ of degree $\deg{P}\leq s-1$ such that
  \[
    \left|\E_{x}f(x)e_p(P(x))\right|\gg_{s,p}\f{1}{\exp^{c_s}(c_{s,p}'/\delta)}.
  \]
\end{theorem}

\subsection{Proof of Lemma~\ref{pseudoprop}}

Our proof of Lemma~\ref{pseudoprop} is modeled after the corresponding pseudorandomization arguments in~\cite{Shkredov06I,Shkredov06II} and~\cite{Green05note}. As was said in the outline, some new features arise from our desire for $B',C',D',$ and $\Phi'$ to be uniform with respect to $U^s$-norms of degree greater than $2$ and from $\Phi$'s particular structure as a union of affine subspaces of the second factor of $\F_p^n\times\F_p^n$. The first of these two complications can be overcome by using Theorem~\ref{GM} and then Theorem~\ref{CT} or Corollary~\ref{bilCT}. We will discuss how $\Phi$'s structure influences our proof shortly.

The proof of Lemma~\ref{pseudoprop} proceeds via an energy-increment argument. Each step of the energy-increment iteration will produce a partition $\mathscr{C}_{j}=(\mathcal{C}_{i,j})_{i\in I_j}$ of $\F_p^n\times \F_p^n$ into cells $\mathcal{C}_{i,j}$ of the form
  \[
    \mathcal{C}_{i,j}=(u_{i,j}+V_{i,j})\times(w_{i,j}+V_{i,j})
  \]
  for some subspace $V_{i,j}\leq \F_p^n$. For each cell $\mathcal{C}=(u+V)\times(w+V)$ in a partition $\mathscr{C}$, we set
  \begin{itemize}
  \item $B_{\mathcal{C}}:= B\cap (w+V)$,
  \item $C_{\mathcal{C}}:= C\cap(u+w+V)$,
  \item $D_{\mathcal{C}}:= D\cap(2u+w+V)$, and
  \item $\Phi_{\mathcal{C}}:=\Phi\cap\mathcal{C}$,
\end{itemize}
and, correspondingly,
\begin{itemize}
  \item $\beta(\mathcal{C}):=\mu_{w+V}(B_{\mathcal{C}})$,
  \item $\gamma(\mathcal{C}):=\mu_{u+w+V}(C_{\mathcal{C}})$,
  \item $\delta(\mathcal{C}):=\mu_{2u+w+V}(D_{\mathcal{C}})$, and
  \item $\phi(\mathcal{C}):=\mu_{\mathcal{C}}(\Phi_{\mathcal{C}})$,
  \end{itemize}
  so that
  \[
    T\cap\mathcal{C}=\left\{(x,y)\in\mathcal{C}:B_{\mathcal{C}}(y)C_{\mathcal{C}}(x+y)D_{\mathcal{C}}(2x+y)\Phi_{\mathcal{C}}(x,y)=1\right\}
  \]
  Analogously to the pseudorandomization procedure for corners, we will show that if $\|B_{\mathcal{C}}-\beta(\mathcal{C})\|_{U^8(\F_p^n)},\|C_{\mathcal{C}}-\gamma(\mathcal{C})\|_{U^8(\F_p^n)},$ or $\|D_{\mathcal{C}}-\delta(\mathcal{C})\|_{U^8(\F_p^n)}$ is large for a substantial portion of cells $\mathcal{C}$ in a partition $\mathscr{C}$, then there exists a refinement $\mathscr{C}'$ of $\mathscr{C}$ which has substantially larger energy. But even if $\Phi_{\mathcal{C}}$ is a union of affine subspaces of the same codimension $d$ for most cells in the partition, this may not be the case for the $\Phi_{\mathcal{C}'}$'s corresponding to the cells $\mathcal{C}'$ in the refinement $\mathscr{C}'$. The codimensions of the affine subspaces $\{y\in w'+V':\Phi_{\mathcal{C}'}(x,y)=1\}$ can range from $0$ to $d$, so before even considering how to obtain a pseudorandom $\Phi'$, we have already found an obstacle to even getting a set of the same general form as $\Phi'$.

  To get around this issue, we will pseudorandomize each of the sets
  \begin{equation}\label{phileqdef}
    \Phi^{\leq i}_{\mathcal{C}}:=\left\{(x,y)\in\Phi_{\mathcal{C}}:\E_{z\in w+V}\Phi(x,z)\geq p^{-i}\right\}
  \end{equation}
  for $0\leq i\leq d$, instead of just $\Phi_{\mathcal{C}}=\Phi^{\leq d}_{\mathcal{C}}$ itself. The definition~\eqref{phileqdef} of $\Phi^{\leq i}_{\mathcal{C}}$ selects all of the subspaces of the second factor of $\mathcal{C}$ comprising $\Phi_{\mathcal{C}}$ that have codimension at most $i$. This will pseudorandomize each of
  \[
    \Phi^i_{\mathcal{C}}:=\left\{(x,y)\in\Phi_{\mathcal{C}}:\E_{z\in w+V}\Phi(x,z)=p^{-i}\right\}=\Phi^{\leq i}_{\mathcal{C}}\setminus\Phi^{\leq i-1}_{\mathcal{C}}
  \]
  as well. At the end of the proof of Lemma~\ref{pseudoprop}, we will use an averaging argument to choose $\Phi'$ to be some suitable $\Phi_{\mathcal{C}}^i$.

  For each $0\leq i\leq d$, set $\phi^{\leq i}(\mathcal{C}):=\mu_{\mathcal{C}}(\Phi_{\mathcal{C}}^{\leq i})$. We define the energy $E(\mathscr{C})$ of a partition $\mathscr{C}$ of $\F_p^n\times\F_p^n$ by
  \[
    E(\mathscr{C}):=\f{1}{4+d}\sum_{\mathcal{C}\in\mathscr{C}}\left(\beta(\mathcal{C})^2+\gamma(\mathcal{C})^2+\delta(\mathcal{C})^2+\sum_{i=0}^d\phi^{\leq i}(\mathcal{C})^2\right)\mu_{\F_p^n\times\F_p^n}(\mathcal{C})
  \]
  Note that the energy of any partition is bounded above by $1$. We will now prove a couple of lemmas concerning this energy.

  \begin{lemma}\label{msdgeq}
    Let $m'>m$ and $d$ be nonnegative integers, $A,B,C,D\subset\F_p^n$, $\Phi\subset\F_p^n\times\F_p^n$ be of the form
    \[
      \Phi=\{(x,y)\in A\times\F_p^n:y\in u+V_x\},
    \]
    where each $V_x$ is a subspace of $\F_p^n$ of codimension between $0$ and $d$, $\mathscr{C}$ be a partition of $\F_p^n\times\F_p^n$ with each cell $\mathcal{C}$ taking the form
  \[
      \mathcal{C}=(u_{\mathcal{C}}+V_{\mathcal{C}})\times(w_{\mathcal{C}}+V_{\mathcal{C}})
    \]
    for some subspace $V_{\mathcal{C}}\leq\F_p^n$ of codimension $m$, and suppose that $\mathscr{C}'$ is a refinement of $\mathscr{C}$ with each cell $\mathcal{C}'$ taking the form
    \[
      \mathcal{C}'=(u_{\mathcal{C}'}'+V_{\mathcal{C}'}')\times(w_{\mathcal{C}'}'+V_{\mathcal{C}'}')
    \]
    for some subspace $V'_{\mathcal{C}'}$ of codimension $m'$. Then
    \[
      \sum_{\mathcal{C}'\in\mathscr{C}'}\beta(\mathcal{C}')^2\mu_{\F_p^n\times\F_p^n}(\mathcal{C}')\geq\sum_{\mathcal{C}\in\mathscr{C}}\beta(\mathcal{C})^2\mu_{\F_p^n\times\F_p^n}(\mathcal{C}),
    \]
    \[
 \sum_{\mathcal{C}'\in\mathscr{C}'}\gamma(\mathcal{C}')^2\mu_{\F_p^n\times\F_p^n}(\mathcal{C}')\geq\sum_{\mathcal{C}\in\mathscr{C}}\gamma(\mathcal{C})^2\mu_{\F_p^n\times\F_p^n}(\mathcal{C}),
    \]
    \[
 \sum_{\mathcal{C}'\in\mathscr{C}'}\delta(\mathcal{C}')^2\mu_{\F_p^n\times\F_p^n}(\mathcal{C}')\geq\sum_{\mathcal{C}\in\mathscr{C}}\delta(\mathcal{C})^2\mu_{\F_p^n\times\F_p^n}(\mathcal{C}),
    \]
    and
    \[
      \sum_{\mathcal{C}'\in\mathscr{C}'}\phi^{\leq i}(\mathcal{C}')^2\mu_{\F_p^n\times\F_p^n}(\mathcal{C}')\geq \sum_{\mathcal{C}\in\mathscr{C}}\phi^{\leq i}(\mathcal{C})^2\mu_{\F_p^n\times\F_p^n}(\mathcal{C})
    \]
    for every $0\leq i\leq d$.
  \end{lemma}
  \begin{proof}
    Note that it suffices to prove the result with the sum over $\mathcal{C}\in\mathscr{C}$ restricted to a single cell $\mathcal{C}_0$ and the sum over $\mathcal{C}'\in\mathscr{C}'$ restricted to all cells contained in $\mathcal{C}_0$, since one can then just sum over $\mathcal{C}_0$ in $\mathscr{C}$ to get the desired result. So, we may assume without loss of generality that $\mathscr{C}$ is the trivial partition $\{\F_p^n\times\F_p^n\}$.

    Let $\beta,\gamma,$ and $\delta$ denote the densities of $B,C,$ and $D$, respectively, in $\F_p^n$. Since $\F_p^n\times B$ has density $\beta$ in $\F_p^n\times\F_p^n$, we have
    \[
      \beta=\sum_{\mathcal{C}'\in\mathscr{C}'}\beta(\mathcal{C}')\mu_{\F_p^n\times\F_p^n}(\mathcal{C}'),
    \]
    so that, by the Cauchy--Schwarz inequality,
    \begin{align*}
      \beta^2&\leq\left(\sum_{\mathcal{C}'\in\mathscr{C}'}\beta(\mathcal{C}')^2\right)\left(\sum_{\mathcal{C}'\in\mathscr{C}'}\mu_{\F_p^n\times\F_p^n}(\mathcal{C}')^2\right) \\
      &=\sum_{\mathcal{C}'\in\mathscr{C}'}\beta(\mathcal{C}')^2\mu_{\F_p^n\times\F_p^n}(\mathcal{C}'),
    \end{align*}
    as desired, since
    \[
      \sum_{\mathcal{C}'\in\mathscr{C}'}\mu_{\F_p^n\times\F_p^n}(\mathcal{C}')^2=\f{p^{2n}}{p^{2n-2m'}}\cdot \f{1}{p^{4m'}}=\f{1}{p^{2m'}}=\mu_{\F_p^n\times\F_p^n}(\mathcal{C}')
    \]
    for all cells $\mathcal{C}'$ of $\mathscr{C}'$. Similarly, since $\{(x,y)\in\F_p^n\times\F_p^n:x+y\in C\}$ has density $\gamma$ in $\F_p^n\times\F_p^n$ and $\{(x,y)\in\F_p^n\times\F_p^n:2x+y\in D\}$ has density $\delta$ in $\F_p^n\times\F_p^n$, we have
    \[
       \gamma=\sum_{\mathcal{C}'\in\mathscr{C}'}\gamma(\mathcal{C}')\mu_{\F_p^n\times\F_p^n}(\mathcal{C}')
    \]
    and
    \[
      \delta=\sum_{\mathcal{C}'\in\mathscr{C}'}\delta(\mathcal{C}')\mu_{\F_p^n\times\F_p^n}(\mathcal{C}'),
    \]
    it follows again from the Cauchy--Schwarz inequality that
    \[
      \sum_{\mathcal{C}'\in\mathscr{C}'}\gamma(\mathcal{C}')^2\mu_{\F_p^n\times\F_p^n}(\mathcal{C}')\geq\gamma^2
    \]
    and
    \[
      \sum_{\mathcal{C}'\in\mathscr{C}'}\delta(\mathcal{C}')^2\mu_{\F_p^n\times\F_p^n}(\mathcal{C}')\geq\delta^2.
    \]

    The argument for the $\phi^{\leq i}$'s is also essentially identical, but with one small difference. For ease of notation, set $\Phi^{\leq i}:=\Phi^{\leq i}_{\F_p^n\times\F_p^n}$ and $\phi^{\leq i}:=\phi^{\leq i}(\F_p^n\times\F_p^n)$ for all $0\leq i\leq d$. Then we actually have
    \[
      \phi^{\leq i}\leq\sum_{\mathcal{C}'\in\mathscr{C}'}\phi^{\leq i}(\mathcal{C}')\mu_{\F_p^n\times\F_p^n}(\mathcal{C}'),
    \]
    instead of equality, since $\Phi^{\leq i}\cap\mathcal{C}'\subset \Phi^{\leq i}_{\mathcal{C}'}$ (which is why we run the energy-increment argument with the $\Phi^{\leq i}$'s, instead of the $\Phi^i$'s). It therefore follows yet again from the Cauchy--Schwarz inequality that
    \[
      \sum_{\mathcal{C}'\in\mathscr{C}'}\phi^{\leq i}(\mathcal{C}')^2\mu_{\F_p^n\times\F_p^n}(\mathcal{C}')\geq(\phi^{\leq i})^2
    \]
    for all $0\leq i\leq d$.
  \end{proof}

  \begin{lemma}\label{energyinc}
    Let $m$ and $d$ be nonnegative integers, $A,B,C,D\subset\F_p^n$, $\Phi\subset\F_p^n\times\F_p^n$ be of the form
    \[
      \Phi=\{(x,y)\in A\times\F_p^n:y\in u+V_x\},
    \]
    where each $V_x$ is a subspace of $\F_p^n$ of codimension between $0$ and $d$,  and $\mathscr{C}$ be a partition of $\F_p^n\times\F_p^n$ with each cell $\mathcal{C}$ taking the form
  \[
      \mathcal{C}=(u_{\mathcal{C}}+V_{\mathcal{C}})\times(w_{\mathcal{C}}+V_{\mathcal{C}})
    \]
    for some subspace $V_{\mathcal{C}}\leq\F_p^n$ of dimension $m$. There exists a positive integer
    \[
      m'\gg m^{1/(8!)^2}
    \]
    and positive integers $c,c'_p>0$, such that the following holds.

    Let $\mathcal{C}\in\mathscr{C}$.
    \begin{enumerate}
    \item If $\|B_{\mathcal{C}}-\beta(\mathcal{C})\|_{U^{10}(w_\mathcal{C}+V_{\mathcal{C}})}\geq\ve$, then there exists a partition $\mathscr{C}'_{\mathcal{C}}$ of $\mathcal{C}$ with each cell $\mathcal{C}'$ taking the form
      \[
        \mathcal{C}'=(u_{\mathcal{C}'}'+V_{\mathcal{C}'}')\times(w_{\mathcal{C}'}'+V_{\mathcal{C}'}'),
      \]
      with each $\dim V_{C'}'=m'$, such that
      \[
        \sum_{\mathcal{C}'\in\mathscr{C}'_{\mathcal{C}}}\beta(\mathcal{C}')^2\mu_{\mathcal{C}}(\mathcal{C}')\geq\beta(\mathcal{C})^2+\Omega\left(\f{1}{\exp^c(c'_p/\ve)^2}\right).
      \]
    \item If $\|C_{\mathcal{C}}-\gamma(\mathcal{C})\|_{U^{10}(u_{\mathcal{C}}+w_\mathcal{C}+V_{\mathcal{C}})}\geq\ve$, then there exists a partition $\mathscr{C}'_{\mathcal{C}}$ of $\mathcal{C}$ with each cell $\mathcal{C}'$ taking the form
      \[
        \mathcal{C}'=(u_{\mathcal{C}'}'+V_{\mathcal{C}'}')\times(w_{\mathcal{C}'}'+V_{\mathcal{C}'}'),
      \]
      with each $\dim V_{C'}'=m'$, such that
      \[
        \sum_{\mathcal{C}'\in\mathscr{C}'_{\mathcal{C}}}\gamma(\mathcal{C}')^2\mu_{\mathcal{C}}(\mathcal{C}')\geq\gamma(\mathcal{C})^2+\Omega\left(\f{1}{\exp^c(c'_p/\ve)^2}\right).
      \]
      \item If $\|D_{\mathcal{C}}-\delta(\mathcal{C})\|_{U^{10}(2u_{\mathcal{C}}+w_\mathcal{C}+V_{\mathcal{C}})}\geq\ve$, then there exists a partition $\mathscr{C}'_{\mathcal{C}}$ of $\mathcal{C}$ with each cell $\mathcal{C}'$ taking the form
      \[
        \mathcal{C}'=(u_{\mathcal{C}'}'+V_{\mathcal{C}'}')\times(w_{\mathcal{C}'}'+V_{\mathcal{C}'}'),
      \]
      with each $\dim V_{C'}'=m'$, such that
      \[
        \sum_{\mathcal{C}'\in\mathscr{C}'_{\mathcal{C}}}\delta(\mathcal{C}')^2\mu_{\mathcal{C}}(\mathcal{C}')\geq\delta(\mathcal{C})^2+\Omega\left(\f{1}{\exp^c(c'_p/\ve)^2}\right).
      \]
      \item Let $0\leq i\leq d$. If $\|\Phi^{\leq i}_{\mathcal{C}}-\phi^{\leq i}(\mathcal{C})\|_{U^{8}(\mathcal{C})}\geq\ve$, then there exists a partition $\mathscr{C}'_{\mathcal{C}}$ of $\mathcal{C}$ with each cell $\mathcal{C}'$ taking the form
      \[
        \mathcal{C}'=(u_{\mathcal{C}'}'+V_{\mathcal{C}'}')\times(w_{\mathcal{C}'}'+V_{\mathcal{C}'}'),
      \]
      with each $\dim V_{C'}'=m'$, such that
      \[
        \sum_{\mathcal{C}'\in\mathscr{C}'_{\mathcal{C}}}\phi^{\leq i}(\mathcal{C}')^2\mu_{\mathcal{C}}(\mathcal{C}')\geq\phi^{\leq i}(\mathcal{C})^2+\Omega\left(\f{1}{\exp^c(c'_p/\ve)^2}\right).
      \]
    \end{enumerate}
  \end{lemma}
  \begin{proof}
    Let $c$ and $c'_p$ denote  the constants $c_{10}$ and $c_{10,p}'$, respectively, from Theorem~\ref{GM}, and $m'$ denote the smaller of the two minimum values of $m'$ appearing in Theorem~\ref{CT} when we take $m$ as in this lemma, $d=9$, and $t=1$ and $m'$ appearing in Corollary~\ref{bilCT} when we take $m$ as in this lemma and $d=7$.

    First assume that $\|B_{\mathcal{C}}-\beta(\mathcal{C})\|_{U^{10}(w_{\mathcal{C}}+V_{\mathcal{C}})}\geq\ve$. Then applying Theorem~\ref{GM} with $s=10$ yields a polynomial $P\in\F_p[x_1,\dots,x_m]$ of degree at most $9$ such that
    \[
      \left|\E_{x\in w_{\mathcal{C}}+V_{\mathcal{C}}}(B_{\mathcal{C}}-\beta(\mathcal{C}))(x)e_p(P(x))\right|\gg\f{1}{\exp^{c}(c'_p/\ve)}.
    \]
    By Theorem~\ref{CT}, there exists a partition $(w_i+V_i)_{i\in I}$ of $w_{\mathcal{C}}+V_{\mathcal{C}}$ into affine subspaces of $w_{\mathcal{C}}+V_{\mathcal{C}}$ of dimension $m'$, on each of which $P$ is constant. Thus, by the triangle inequality,
    \[
      \E_{i\in I}\left|\E_{x\in w_{i}+V_{i}}(B_{\mathcal{C}}-\beta(\mathcal{C}))(x)\right|\gg\f{1}{\exp^{c}(c'_p/\ve)},
    \]
    so that, by the Cauchy--Schwarz inequality,
    \[
      \E_{i\in I}\left|\E_{x\in w_{i}+V_{i}}(B_{\mathcal{C}}-\beta(\mathcal{C}))(x)\right|^2\gg\f{1}{\exp^{c}(c'_p/\ve)^2}.
    \]
    Expanding the square, this means that
   \begin{equation}\label{binc}
      \E_{i\in I}\left|\E_{x\in w_{i}+V_{i}}B_{\mathcal{C}}(x)\right|^2\geq\beta(\mathcal{C})^2+\Omega\left(\f{1}{\exp^{c}(c'_p/\ve)^2}\right).
    \end{equation}
    
    Now we partition the whole cell of interest $\mathcal{C}$ by writing
    \[
      \mathcal{C}=(u_{\mathcal{C}}+V_{\mathcal{C}})\times\coprod_{i\in I}(w_i+V_i)=\coprod_{i\in I}(u_{\mathcal{C}}+V_{\mathcal{C}})\times(w_i+V_i),
    \]
    and, for each $i\in I$, use that $V_i\leq V_{\mathcal{C}}$ to partition $u_{\mathcal{C}}+V_{\mathcal{C}}$ into cosets of $V_i$ to get
    \[
      \mathcal{C}=\coprod_{i\in I}\coprod_{u'+V_i\in V_{\mathcal{C}}/V_i}(u_{\mathcal{C}}+u'+V_i)\times(w_i+V_i)=:\coprod_{\mathcal{C}'\in\mathscr{C}_{\mathcal{C}}}\mathcal{C}'.
    \]
    Since $\mu_{\mathcal{C}}(\mathcal{C}')=|\mathscr{C}'|\1$ for each $\mathcal{C}'\in\mathscr{C}_{\mathcal{C}}'$,~\eqref{binc} reads
    \[
      \sum_{\mathcal{C}'\in\mathscr{C}_{\mathcal{C}}'}\beta(\mathcal{C}')^2\mu_{\mathcal{C}}(\mathcal{C}')\geq\beta(\mathcal{C})^2+\Omega\left(\f{1}{\exp^{c}(c'_p/\ve)^2}\right).
    \]

    The arguments for $C_{\mathcal{C}}$ and $D_{\mathcal{C}}$ are again analogous, but we include them for the sake of completeness. Next, assume that $\|C_{\mathcal{C}}-\gamma(\mathcal{C})\|_{U^{10}(u_{\mathcal{C}}+w_{\mathcal{C}}+V_{\mathcal{C}})}\geq\ve$. Applying Theorem~\ref{GM} yields a polynomial $P\in\F_p[x_1,\dots,x_m]$ of degree at most $9$ such that
    \[
      \left|\E_{x\in u_{\mathcal{C}}+w_{\mathcal{C}}+V_{\mathcal{C}}}(C_{\mathcal{C}}-\gamma(\mathcal{C}))(x)e_p(P(x))\right|\gg\f{1}{\exp^{c}(c'_p/\ve)}.
    \]
    By Theorem~\ref{CT}, there exists a partition $(v_i+V_i)_{i\in I}$ of $u_{\mathcal{C}}+w_{\mathcal{C}}+V_{\mathcal{C}}$ into affine subspaces of $u_{\mathcal{C}}+w_{\mathcal{C}}+V_{\mathcal{C}}$ of dimension $m'$ on each of which $P$ is constant. Thus, by the Cauchy--Schwarz inequality again,
    \[
      \E_{i\in I}\left|\E_{x\in v_{i}+V_{i}}C_{\mathcal{C}}(x)\right|^2\geq\gamma(\mathcal{C})^2+\Omega\left(\f{1}{\exp^{c}(c'_p/\ve)^2}\right).
    \]
    Now we partition the whole cell $\mathcal{C}$ by writing
    \[
      \mathcal{C}=\coprod_{i\in I}\coprod_{v'+V_i\in V_{\mathcal{C}}/V_i}(v_i-w_{\mathcal{C}}+v'+V_i)\times(w_{\mathcal{C}}-v'+V_i)=:\coprod_{\mathcal{C}'\in\mathscr{C}_{\mathcal{C}}}\mathcal{C}'.
    \]
    Since $(v_i-w_{\mathcal{C}}+v'+V_i)+(w_{\mathcal{C}}-v'+V_i)=v_i+V_i$, we conclude that
    \[
      \sum_{\mathcal{C}'\in\mathscr{C}_{\mathcal{C}}'}\gamma(\mathcal{C}')^2\mu_{\mathcal{C}}(\mathcal{C}')\geq\gamma(\mathcal{C})^2+\Omega\left(\f{1}{\exp^{c}(c'_p/\ve)^2}\right).
    \]

   Now assume that $\|D_{\mathcal{C}}-\delta(\mathcal{C})\|_{U^{10}(2u_{\mathcal{C}}+w_{\mathcal{C}}+V_{\mathcal{C}})}\geq\ve$. Applying Theorem~\ref{GM} yields a polynomial $P\in\F_p[x_1,\dots,x_m]$ of degree at most $9$ such that
    \[
      \left|\E_{x\in 2u_{\mathcal{C}}+w_{\mathcal{C}}+V_{\mathcal{C}}}(D_{\mathcal{C}}-\delta(\mathcal{C}))(x)e_p(P(x))\right|\gg\f{1}{\exp^{c}(c'_p/\ve)}.
    \]
    By Theorem~\ref{CT}, there exists a partition $(v_i+V_i)_{i\in I}$ of $2u_{\mathcal{C}}+w_{\mathcal{C}}+V_{\mathcal{C}}$ into affine subspaces of $2u_{\mathcal{C}}+w_{\mathcal{C}}+V_{\mathcal{C}}$ of dimension $m'$ on each of which $P$ is constant. Thus, by the Cauchy--Schwarz inequality yet again,
    \[
      \E_{i\in I}\left|\E_{x\in v_{i}+V_{i}}D_{\mathcal{C}}(x)\right|^2\geq\delta(\mathcal{C})^2+\Omega\left(\f{1}{\exp^{c}(c'_p/\ve)^2}\right).
    \]
    Now we partition the whole cell $\mathcal{C}$ by writing
    \[
      \mathcal{C}=\coprod_{i\in I}\coprod_{v'+V_i\in V_{\mathcal{C}}/V_i}(u_{\mathcal{C}}+v'+V_i)\times(v_i-2u_{\mathcal{C}}-2v'+V_i)=:\coprod_{\mathcal{C}'\in\mathscr{C}_{\mathcal{C}}}\mathcal{C}'.
    \]
    Since $(2u_{\mathcal{C}}+2v'+V_i)+(v_i-2u_{\mathcal{C}}-2v'+V_i)=v_i+V_i$, we conclude that
    \[
      \sum_{\mathcal{C}'\in\mathscr{C}_{\mathcal{C}}'}\delta(\mathcal{C}')^2\mu_{\mathcal{C}}(\mathcal{C}')\geq\delta(\mathcal{C})^2+\Omega\left(\f{1}{\exp^{c}(c'_p/\ve)^2}\right).
    \]

    Finally, suppose that $\|\Phi^{\leq i}_{\mathcal{C}}-\phi^{\leq i}(\mathcal{C})\|_{U^{8}(\mathcal{C})}\geq\ve$ for some $0\leq i\leq d$. Theorem~\ref{GM} then says that there exists a polynomial $R\in\F_p[x_1,\dots,x_m,y_1,\dots,y_m]$ of degree at most $7$ such that
    \[
      \left|\E_{(x,y)\in\mathcal{C}}(\Phi^{\leq i}_{\mathcal{C}}-\phi^{\leq i}(\mathcal{C}))(x,y)e_p(R(x,y))\right|\gg\f{1}{\exp^{c}(c'_p/\ve)}.
    \]
    By Corollary~\ref{bilCT}, there exists a partition $\mathscr{C}'_{\mathcal{C}}$ of $\mathcal{C}$ into affine subspaces of the form $(u+V)\times(w+V)$ with $\dim V=m'$, on each of which $R$ is constant. Thus, by the Cauchy--Schwarz inequality, we have
    \[
      \E_{\mathcal{C}'\in \mathscr{C}'_{\mathcal{C}}}\left|\E_{(x,y)\in\mathcal{C}'}\Phi^{\leq i}_{\mathcal{C}}(x,y)\right|^2\geq\phi^{\leq i}(\mathcal{C})^2+\Omega\left(\f{1}{\exp^{c}(c'_p/\ve)^2}\right).
    \]
    Since
    \[
      \E_{(x,y)\in\mathcal{C}'}\Phi^{\leq i}_{\mathcal{C}}(x,y)\leq
      \E_{(x,y)\in\mathcal{C}'}\Phi^{\leq i}_{\mathcal{C}'}(x,y)=\phi^{\leq i}(\mathcal{C}'),
    \]
    the conclusion
    \[
       \sum_{\mathcal{C}'\in\mathscr{C}'_{\mathcal{C}}}\phi^{\leq i}(\mathcal{C}')^2\mu_{\mathcal{C}}(\mathcal{C}')\geq\phi^{\leq i}(\mathcal{C})^2+\Omega\left(\f{1}{\exp^c(c'_p/\ve)^2}\right)
     \]
     now follows.
   \end{proof}
   Now we can prove Lemma~\ref{pseudoprop}.
   \begin{proof}[Proof of Lemma 2.6]
     We proceed via an energy-increment argument, as described at the beginning of the subsection. A cell $\mathcal{C}=(u+V)\times(w+V)$ in a partition $\mathscr{C}_j$ is said to be \textit{expired} if $\beta(\mathcal{C}),\gamma(\mathcal{C}),\delta(\mathcal{C}),$ or $\phi^{\leq d}(\mathcal{C})$ is less than $\tau\mu_{\F_p^n\times\F_p^n}(T)/4$, and a nonexpired cell $\mathcal{C}$ is said to be \textit{uniform} if
     \[
       \|B_{\mathcal{C}}-\beta(\mathcal{C})\|_{U^{10}(w+V)},\|C_{\mathcal{C}}-\gamma(\mathcal{C})\|_{U^{10}(u+w+V)},\|D_{\mathcal{C}}-\delta(\mathcal{C})\|_{U^{10}(2u+w+V)}<\ve
     \]
     and
     \[
       \|\Phi^{\leq i}_{\mathcal{C}}-\phi^{\leq i}\|_{U^{8}(\mathcal{C})}<\ve
     \]
     for all $0\leq i\leq d$. We will denote the subset of expired cells of $\mathscr{C}_j$ by $\mathscr{E}_j$, the subset of uniform cells by $\mathscr{U}_j$, and the subset of nonexpired, nonuniform cells by $\mathscr{N}_j$, so that $\mathscr{E}_j,\mathscr{U}_j$, and $\mathscr{N}_j$ partition $\mathscr{C}_j$. For any subset $K\subset I_j$, we define $\eta(K)$ to be the measure of all cells indexed by $K$:
     \[
       \eta(K):=\mu_{\F_p^n\times\F_p^n}\left(\coprod_{k\in K}\mathcal{C}_{k,j}\right).
     \]
     Finally, we define a sequence of integers $(m_j)_{j=0}^{\infty}$ by setting $m_0=n$
     and, for every $j>0$, $m_{j}$ to be the minimum of the value of $m'$ appearing in
     Theorem~\ref{CT} when we take $m=m_{j-1}$, $d=9$, and $t=1$ and of $m'$ appearing in Corollary~\ref{bilCT} when we take $m=m_{j-1}$ and $d=7$, so that
     \[
       m_j\geq c_1 n^{c_2^j}
     \]
     for some absolute constants $0<c_1,c_2<1$.

     Set $\mathscr{C}_0$ to be the trivial partition $\{\F_p^n\times\F_p^n\}$ of $\F_p^n\times\F_p^n$. Letting $c=c_{10}$ and $c'=c'_{10,p}$ be as in Lemma~\ref{energyinc}, then, as long as $\eta(\mathscr{N}_j)\geq\tau\mu_{\F_p^n\times\F_p^n}(T)/2$ and $m_{j+1}\geq 1$, there exists a refinement $\mathscr{C}_{j+1}$ of $\mathscr{C}_j$ such that
     \begin{enumerate}
     \item $\dim{V_{i,j+1}}\geq m_{j+1}$ for every $i\in I_{j+1}$ and $\dim{V_{i,j+1}}=m_{j+1}$ whenever $\mathcal{C}_{i,j+1}\in \mathscr{N}_{j+1}$ and
     \item $E(\mathscr{C}_{j+1})\geq E(\mathscr{C}_j)+\Omega\left(\f{\tau\mu_{\F_p^n\times\F_p^n}(T)}{d\exp^c(c'_p/\ve)^2}\right)$.
     \end{enumerate}
     Indeed, suppose that $\eta(\mathscr{N}_j)\geq \tau\mu_{\F_p^n\times\F_p^n}(T)/2$. Each cell $\mathcal{C}$ in $\mathscr{N}_j$ must be of dimension $m_j\times m_j$. By Lemmas~\ref{msdgeq} and~\ref{energyinc}, there exists a partition $\mathscr{C}_{k,j+1}=(\mathcal{C}_{k,j+1})_{k\in K_{\mathcal{C}}}$ of each $\mathcal{C}\in \mathscr{N}_j$ such that
     \[
       \f{1}{4+d}\sum_{k\in K_{\mathcal{C}}}\left(\beta(\mathcal{C}_{k,j+1})^2+\gamma(\mathcal{C}_{k,j+1})^2+\delta(\mathcal{C}_{k,j+1})^2+\sum_{i=0}^d\phi^{\leq i}(\mathcal{C}_{k,j+1})^2\right)\mu_{\mathcal{C}}(\mathcal{C}_{k,j+1})
     \]
     is at least
     \[
       \f{\beta(\mathcal{C})^2+\gamma(\mathcal{C})^2+\delta(\mathcal{C})^2+\sum_{i=0}^d\phi^{\leq i}(\mathcal{C})^2}{4+d}+\Omega\left(\f{1}{d\exp^c(c'/\ve)^2}\right)
     \]
     and each $\mathcal{C}_{k,j+1}$ is of the form $(u'+V')\times(w'+V')$ with
     $\dim{V'}=m_{j+1}\geq 1$. Taking
     \[
     \mathscr{C}_{j+1}:=\left\{\mathcal{C}_{k,j+1}:k\in K_{\mathcal{C}},\mathcal{C}\in\mathscr{N}_j\right\}\cup\mathscr{E}_j\cup\mathscr{U}_{j},
   \]
   we see that multiplying both sides of the above by
   $\mu_{\F_p^n\times\F_p^n}(\mathcal{C})$ and summing over $\mathcal{C}\in\mathscr{N}_j$
   yields
     \[
       E(\mathscr{C}_{j+1})\geq E(\mathscr{C}_j)+\Omega\left(\f{\tau\mu_{\F_p^n\times\F_p^n}(T)}{d\exp^c(c'/\ve)^2}\right).
     \]
     Since $E(\mathscr{C})\leq 1$ for all partitions $\mathscr{C}$, this iteration must
     terminate for some $j=j_0\ll\f{d\exp^c(c'/\ve)^2}{\tau\mu_{\F_p^n\times\F_p^n}(T)}$,
     at which point either $\eta(\mathscr{N}_{j_0})<\tau\mu_{\F_p^n\times\F_p^n}(T)/2$ or
     $m_{j_0+1}<1$. Assuming that $n\geq c_1^{-c_2^{-(j_0+1)}}$ ensures that the latter
     case cannot occur.

     Since $\eta(\mathscr{E}_j)<\tau\mu_{\F_p^n\times\F_p^n}(T)/4$, we have
     \[
       \mu_{\F_p^n\times\F_p^n}\left(S\cap \bigcup_{\mathcal{C}\in\mathscr{U}_{j_0}}\mathcal{C}\right)\geq \left(\sigma+\f{\tau}{4}\right)\mu_{\F_p^n\times\F_p^n}(T).
     \]
     This certainly implies that
     \[
       \sum_{\mathcal{C}\in\mathscr{U}_{j_0}}\mu_{\F_p^n\times\F_p^n}(S\cap\mathcal{C})\geq \left(\sigma+\f{\tau}{4}\right)\sum_{\mathcal{C}\in\mathscr{U}_{j_0}}\mu_{\F_p^n\times\F_p^n}(T\cap\mathcal{C}),
     \]
     so that, by the pigeonhole principle, there exists a cell $\mathcal{C}_0$ in $\mathscr{U}_{j_0}$ for which
     \[
       \mu_{\F_p^n\times\F_p^n}(S\cap\mathcal{C}_0)\geq \left(\sigma+\f{\tau}{4}\right)\mu_{\F_p^n\times\F_p^n}(T\cap\mathcal{C}_0).
     \]
     Since $\Phi_{\mathcal{C}_0}=\coprod_{i=0}^d\Phi^{i}_{\mathcal{C}_0}$, another application of the pigeonhole principle tells us that there exists a $0\leq i\leq d$ for which we also have the density-increment
     \[
       \mu_{\F_p^n\times\F_p^n}(S\cap\mathcal{C}_0\cap\Phi_{\mathcal{C}_0}^i)\geq \left(\sigma+\f{\tau}{4}\right)\mu_{\F_p^n\times\F_p^n}(T\cap\mathcal{C}_0\cap\Phi_{\mathcal{C}_0}^i).
     \]
     As noted at the beginning of this subsection,
     \[
       \|\Phi^{i}_{\mathcal{C}_0}-\phi^i(\mathcal{C}_0)\|_{U^8(\mathcal{C}_0)}=\|\Phi^{\leq i}_{\mathcal{C}_0}-\phi^{\leq i}(\mathcal{C}_0)+\phi^{\leq i-1}(\mathcal{C}_0)-\Phi^{\leq i-1}_{\mathcal{C}_0}\|_{U^8(\mathcal{C}_0)}<2\ve,
     \]
     so that the conclusion of the lemma now follows.
   \end{proof}
\section{The density-increment argument}\label{dinc}

Now we can finally prove Theorem~\ref{main} by iterating Lemma~\ref{densityinc}.

\begin{proof}[Proof of Theorem~\ref{main}]
  Suppose that $S\subset\F_p^n\times\F_p^n$ has density $\sigma$ and contains no
  nontrivial L-shaped configurations. Set $S_0:=S$, $n_0:=n$, $d_0=0$, $\ve_0:=1$,
  $A_0=B_0=C_0=D_0=\F_p^n$, and $\Phi_0:=\F_p^n\times\F_p^n$. Applying
  Lemma~\ref{densityinc} repeatedly produces sequences of $S_i$'s, $n_i$'s, $d_i$'s,
  $\ve_i$'s, $A_i$'s, $B_i$'s, $C_i$'s, $D_i$'s, and $\Phi_i$'s, with
  $A_i,B_i,C_i,D_i\subset\F_p^{n_i}$ and $\Phi_i\subset A_i\times\F_p^{n_i}$ of the form
  \[
    \Phi_i=\left\{(x,y)\in A_i\times\F_p^{n_i}:y\in u+V_x\right\},
  \]
  where each $V_x\leq \F_p^{n_i}$ is a subspace of codimension $d_i$, such that, on setting
  \[
    T_i:=\left\{(x,y)\in\F_p^{n_i}\times\F_p^{n_i}:B_i(y)C_i(x+y)D_i(2x+y)\Phi_i(x,y)=1\right\},
  \]
  $\alpha_i:=\mu_{\F_p^{n_i}}(A_i)$, $\beta_i:=\mu_{\F_p^{n_i}}(B_i)$, $\gamma_i:=\mu_{\F_p^{n_i}}(C_i)$, $\delta_i:=\mu_{\F_p^{n_i}}(D_i)$, and $\rho_i:=\mu_{\F_p^{n_i}\times\F_p^{n_i}}(\Phi_i)/\alpha_i=p^{-d_i}$, we have, for each $i\geq 1$, that
  \begin{enumerate}
  \item $S_i\subset T_i$ has density $\sigma_i$ in $T_i$, where $\sigma_i\geq\sigma_{i-1}+\Omega\left(\sigma^{O(1)}\right)$,
  \item $n_i\gg n_{i-1}^{c_1^{O\left(\exp^c(c'/\ve_i)/(\sigma\alpha_{i-1}\beta_{i-1}\gamma_{i-1}\delta_{i-1}\rho_{i-1})^{O(1)}\right)}}$,
  \item $\ve_i\leq (\sigma\alpha_{i}\beta_{i}\gamma_{i}\delta_{i}\rho_{i})^{O(1)}\exp(-(64/\sigma)^{O(1)})$,
  \item $\alpha_i,\beta_i,\gamma_i,\delta_i\gg(\sigma\alpha_{i-1}\beta_{i-1}\gamma_{i-1}\delta_{i-1}\rho_{i-1})^{O(1)}$,
  \item $d_i\leq d_{i-1}+1$,
  \item $\|A_i-\alpha_i\|_{U^{10}(\F_p^{n_i})},\|B_i-\beta_i\|_{U^{10}(\F_p^{n_i})},\|C_i-\gamma_i\|_{U^{10}(\F_p^{n_i})},\|D_i-\delta_i\|_{U^{10}(\F_p^{n_i})},\|\Phi_i-\alpha_i\rho_i\|_{U^8(\F_p^{n_i}\times\F_p^{n_i})}<\ve_i$,
  \item and $S_i$ contains no nontrivial L-shaped configurations,
  \end{enumerate}
  provided that
  \begin{equation}\label{lowern}
    n_{i-1}\geq \exp^2\left(O\left(\f{\exp^c(c'/\ve_i)}{(\sigma\alpha_{i-1}\beta_{i-1}\gamma_{i-1}\delta_{i-1}\rho_{i-1})^{c_4}}\right)\right).
  \end{equation}
  Since no set can have density larger than $1$, the lower bound~\eqref{lowern} must fail for some $i=i_0+1\ll\sigma^{-O(1)}$. Thus, there exists an absolute constant $c''>1$ such that
  \[
    n_{i_0}\ll\exp^{c''}(O(1/\sigma^{O(1)}))
  \]
  while, on the other hand
  \[
    n_{i_0}\gg n^{c_1^{O\left(\exp^{c''}\left(O\left(1/\sigma^{O(1)}\right)\right)\right)}}.
  \]
  Comparing the upper and lower bounds for $n_{i_0}$ and taking the $c''$-fold iterated logarithm of both sides yields the bound in Theorem~\ref{main}.
\end{proof}

\bibliographystyle{plain}
\bibliography{bib}

\end{document}